\documentclass{article}

\usepackage{arxiv}
\usepackage{enumitem}
\usepackage[utf8]{inputenc}
\usepackage[T1]{fontenc}
\usepackage{hyperref}
\usepackage{url}
\usepackage{booktabs}
\usepackage{algorithm}
\usepackage{algorithmic}
\usepackage{amsmath}
\usepackage{amsfonts}
\usepackage{amssymb}
\usepackage{amsthm}
\usepackage{nicefrac}
\usepackage{microtype}
\usepackage{fancyhdr}
\usepackage{graphicx}
\usepackage{xcolor}
\usepackage{tikz}
\usetikzlibrary{arrows.meta,calc}
\usepackage{subcaption}
\usepackage{bbm}
\usepackage{lyu}

\newif\ifusesupplement\usesupplementfalse

\newtheorem{theorem}{Theorem}[section]
\newtheorem{corollary}{Corollary}[theorem]
\newtheorem{lemma}[theorem]{Lemma}
\newtheorem{proposition}[theorem]{Proposition}
\theoremstyle{definition}

\newtheorem{assumption}{Assumption}
\theoremstyle{remark}

\graphicspath{{../figure/}{./figure/}}
\title{Wasserstein Moment Nudging for Vlasov–Poisson Data Assimilation}

\author{
  Liyao Lyu\thanks{Equal contribution.}\\
  Department of Mathematics \\
  University of California, Los Angeles \\
  Los Angeles, CA 90095, USA\\
  \texttt{lyuliyao@math.ucla.edu} \\
  \AND
   Xinyue Yu\footnotemark[1] \\
  Department of Mathematics \\
  University of California, Los Angeles \\
  Los Angeles, CA 90095, USA\\
  \texttt{tracy@math.ucla.edu}\\ \AND David Schneidinger \\ Department of Physics \& Astronomy \\ University of California, Los Angeles \\ Los Angeles, CA 90095, USA\\ \texttt{dschneidinger@ucla.edu} \\
  \AND
   Hayden Schaeffer  \\
  Department of Mathematics \\
  University of California, Los Angeles \\
  Los Angeles, CA 90095, USA\\
  \texttt{hayden@math.ucla.edu} \\
}

\begin{document}
\maketitle

\begin{abstract}
We introduce a continuous data assimilation method for particle-in-cell
simulations of the Vlasov-Poisson equation when only hydrodynamic moments are
observed. The forecast state is an empirical measure on phase space, whereas
the observed fields (density, bulk velocity, and temperature) constrain only
a few velocity moments and leave the velocity-space shape of the distribution
undetermined. We construct the moment feedback as a Wasserstein gradient flow
of a moment-mismatch functional over phase-space measures. The resulting drift
acts directly on particle positions and velocities, couples the density,
momentum, and energy residuals through a single variational structure, and
vanishes on the entire moment-compatible set.  Under the standard Wasserstein metric, the energy residual produces a position correction that grows quadratically with the particle
speed, and the particle system falls outside standard well-posedness theory.
Our primary formulation pairs the quadratic moment mismatch with a
velocity-weighted Wasserstein metric that penalizes spatial transport at large peculiar velocity relative to the observed bulk flow, which removes this growth. A direction-split variant retains the plain metric instead. Under the same
weighted metric, an alternative moment-relative-entropy functional yields an
affine, shape-preserving velocity correction and explicit global moment balances for the space-inhomogeneous system. We prove that the finite-particle scheme with linear Lenard-Bernstein collisions is globally well posed. In 1D1V and 2D2V experiments with several collision models, the nudged formulations reduce bulk-velocity and temperature errors by up to two orders of magnitude relative to an unassimilated run.
\end{abstract}




\section{Introduction}

Kinetic simulation is central to predicting and designing plasma systems in
regimes where collective effects and phase-space structure govern the
dynamics. A canonical model is the Vlasov-Poisson equation, which evolves a charged-particle distribution function under its self-consistent electrostatic field. Forecasting realistic kinetic dynamics remains difficult for several reasons. First, the unknown is a distribution on six-dimensional phase space. Grid-based Eulerian and semi-Lagrangian solvers resolve it in low
dimensions~\cite{cheng1976integration,sonnendrucker1999semi,filbet2003comparison},
and sparse and low-rank compressions extend their
application~\cite{kormann2015semi,einkemmer2018low,guo2022low}, but resolution on a full
tensor-product grid remains prohibitive.
Particle-in-cell (PIC)
methods~\cite{verboncoeur1995object,birdsall2018plasma,hockney2021computer,dimarco2014numerical,christlieb2006grid}
avoid this cost by representing the distribution with particles and solving
the field on a mesh. Second, a simulated particle ensemble, such as one obtained from approximating the governing system~\cite{lu2019nonparametric,lu2021learning,lyu2026mvnn,chen2021data,du2022discovery,sharrock2021parameter,zhang2024bayesian, schaeffer2020extracting,liu2023random,lu2024learning, schaeffer2017learning, schaeffer2017sparse}, may deviate from the physical or high-fidelity reference system because of compounding error.
Data assimilation approaches aim to reduce this drift by incorporating observations as the simulation or model evolves~\cite{asch2016data,reich2015probabilistic,law2015data}.

In kinetic plasma problems, this assimilation task is intrinsically
multiscale~\cite{lee2015multiscale,harlim2013test,deng2025lemda,deng2024particle}. The forecast state is microscopic; a PIC simulation carries the
distribution on its particles, so the forecast is the set of particle
positions and velocities, an empirical approximation of the phase-space
distribution. The available observations are macroscopic and represent averaged or large scale behavior. Experimental
diagnostics~\cite{hutchinson2002principles} and reduced outputs of
high-fidelity simulations typically provide density, bulk velocity,
temperature, or field quantities. Thus, the model's governing equations follow a phase-space law
through particles, while the data specify only a few hydrodynamic moments.

In this paper, we extend our previous work~\cite{lyu2026multiscale} from
smoothed observations of the full phase-space law to observations of
hydrodynamic moments. These moments give a reduced order description of the law and thus leaves its velocity-space shape undetermined. For example, a two-stream distribution can
match the moments of a Maxwellian while remaining far from it in velocity
space. The observations therefore determine only a \emph{moment-compatible
set} of phase-space distributions. We design the assimilation feedback to drive the particle ensemble onto this set.

Classical data-assimilation methods do not directly provide the particle-level
feedback required in this setting. Ensemble Kalman and variational methods map observation residuals back to the coordinates of a labeled forecast
state~\cite{lorenc1986analysis,courtier1994strategy,evensen2003ensemble,evensen2009data,calvello2025ensemble}. Particle filters reweigh an ensemble of candidate forecast
states~\cite{moral2004feynman,chopin2020introduction}, and their transport and
particle-flow variants move the samples by feedback
drifts~\cite{daum2010exact,yang2011feedback,pulido2019sequential,gregory2016multilevel}.
Continuous data assimilation adds an interpolated observation residual
into the model state
equation~\cite{azouani2014continuous,bessaih2015continuous,albanez2016continuous,farhat2020data}.
By contrast, in a PIC simulation, the macro-particles are not candidate states
but quadrature points of a single phase-space distribution; the state is their
empirical measure, and the particle labels do not have physical meaning. Hydrodynamic moments are collective functionals of the ensemble, so a residual
in density, bulk velocity, and temperature does not lead to a direct particle-wise correction. Nudging the moment fields directly does fix this issue since the moment hierarchy is not closed at the observed level and thus a
moment-level scheme would require an additional closure. The assimilation
therefore needs a lift from the moment residual to a particle drift that
depends on the ensemble only through its empirical measure. The lift must
vanish on the entire moment-compatible set and provide a feedback process that relaxes the ensemble toward the correct representation. 
Constructing this lift is the core methodological problem of this paper.

Two design choices determine the lift: a moment-mismatch functional and a
transport metric on phase-space measures. The functional quantifies the
discrepancy between the predicted and the observed moments. It should be nonnegative and vanish exactly on the moment-compatible set. We choose the Wasserstein metric~\cite{jordan1998variational,ambrosio2005gradient,villani2008optimal}, where
the gradient of the functional is a phase-space velocity field that can be
evaluated at the particle positions and velocities and added directly into
the PIC simulator. We show that this construction meets both requirements, where the drift sees the particles only through the moment residuals,
which are functionals of the empirical measure, and it vanishes on the
moment-compatible set, where the residuals themselves are zero. The density,
momentum, and energy corrections thus arise as three channels of one
variational construction.

The kinetic energy is quadratic in velocity, so the spatial gradient
of the energy mismatch inherits a factor of $|\bv|^2$. The position correction then grows quadratically with the particle speed. Under the standard Wasserstein metric, this drift is not
globally Lipschitz, and the particle system falls outside classical
well-posedness theory. We therefore develop three methods, each resolving this through a different combination of mismatch functional and
metric. Method~A, our default, keeps the quadratic mismatch and its
single-functional gradient-flow structure but replaces the standard metric by a velocity-weighted Wasserstein metric that raises the cost of spatial transport for fast particles~\cite{benamou2000computational,dolbeault2009new}. The weight removes the quadratic growth and restores well-posedness. Method B keeps the standard $W_2$ geometry but splits the mismatch by phase-space direction, driving positions with the density residual and velocities with the momentum and energy residuals, so the quadratic term does not enter the position correction. Method~C replaces the quadratic mismatch by a moment relative entropy under the same weighted metric, yielding an affine, shape-preserving velocity correction and explicit global moment balances. All three compare states at the moment level and correct them at the particle level.

We establish several analytical results. First, we verify that the feedback vanishes on the moment-compatible set, and we prove a two-Maxwellian obstacle, namely, that a residual-driven feedback cannot synchronize homogeneous Maxwellian states unless the observations distinguish bulk velocity and temperature. Second, for Method~A, we prove that the moment mismatch is a Lyapunov functional of the pure feedback flow with frozen observations. Third, for Method~C, we derive
explicit balances for the mass-weighted global moments of the full
inhomogeneous Vlasov-Poisson flow. Lastly, we prove global well-posedness of
the continuous-time finite-particle scheme closed with a Lenard-Bernstein
collision surrogate, for arbitrary scaling parameters. Numerical experiments in 1D1V and
2D2V show that all three methods reduce the bulk-velocity and temperature errors by up to two orders of magnitude relative to unassimilated runs, and a two-stream experiment delimits what moment observations alone can recover.

The remainder of the paper is organized as follows. Section~\ref{sec:method} introduces the
Vlasov-Poisson-collision model, the PIC discretization, and the
collision closures used in the analysis and the experiments. Section~\ref{sec:nudging}
derives the moment-nudging methods, identifies the moment-compatible
zero set and
the global moment balances for Method~C. Section~\ref{sec:well-posed}
establishes global well-posedness of the finite-particle scheme for the
Lenard-Bernstein surrogate. Section~\ref{sec:numerics} presents the numerical experiments.
Technical proofs and implementation details can be found in the
\ifusesupplement supplementary materials\else appendices\fi.
 
\section{Vlasov-Poisson model and particle discretization}
\label{sec:method}
\subsection{Vlasov-Poisson dynamics and hydrodynamic moments}
\label{sec:VP}
We consider the Vlasov-Poisson system with a collision operator:
\begin{equation}
    \partial_t f(\bx,\bv,t)
    + \bv\cdot \nabla_\bx f
    + \bE(\bx,t)\cdot \nabla_\bv f
    = Q[f],
    \label{eq:VPC}
\end{equation}
where \(f(\bx,\bv,t)\ge 0\) denotes the electron distribution function on
\(\mathbb T^d\times \mathbb R^d\), and \(\mathbb T^d\) is a $d\mhyphen$dimensional torus. The self-consistent electric field is given by
\(\bE(\bx,t)=-\nabla_\bx \phi(\bx,t)\), where the electrostatic potential
solves
\begin{equation}
    -\Delta_\bx \phi(\bx,t) = \rho(\bx,t)-1,
    \qquad
    \rho(\bx,t) = \int_{\mathbb R^d} f(\bx,\bv,t) \diff\bv.
    \label{eq:poisson}
\end{equation}
We write
$\field{f}:=-\nabla_\bx(-\Delta_\bx)^{-1}\bigl(\densr{f}-1\bigr)$
for this self-consistent field when its dependence on $f$ matters.
The constant ion background has been normalized to one, and the potential is
fixed by the zero-mean convention. The collision operator \(Q\) is local in $\bx$ and acts only on the 
velocity dependence, that is, \(Q[f]\) evaluated at $\bx$ depends on $f$ only through $f(\bx,\cdot)$. When we require moment-conserving collisions, we
assume the system satisfies the local conservation identities
\begin{equation}
    \int_{\mathbb R^d} Q[f] \diff\bv = 0,
    \qquad
    \int_{\mathbb R^d} \bv Q[f] \diff\bv = 0,
    \qquad
    \int_{\mathbb R^d} \tfrac12 |\bv|^2 Q[f] \diff\bv = 0.
    \label{eq:Q-conservation}
\end{equation}
The integrands in \eqref{eq:Q-conservation} define the local conserved
moments of $f$, i.e. the mass, momentum, and kinetic-energy densities
\begin{equation}
    \densr{f}
    = \int_{\mathbb R^d} f \diff\bv,
    \qquad
    \momr{f}
    = \int_{\mathbb R^d} \bv\, f \diff\bv,
    \qquad
    \kinr{f}
    = \int_{\mathbb R^d} \tfrac12|\bv|^2 f \diff\bv,
    \label{eq:conserved-moments}
\end{equation}
so that $\rho=\densr{f}$ in \eqref{eq:poisson}. Wherever $\densr{f}>0$,
the associated primitive variables are the bulk velocity and temperature
\begin{align}
    \velr{f}(\bx)
    &= \frac{\momr{f}(\bx)}{\densr{f}(\bx)}
     = \frac{1}{\rho(\bx)}
       \int_{\mathbb R^d} \bv f(\bx,\bv) \diff\bv,
       \label{eq:moments}\\
    \tempr{f}(\bx)
    &= \frac{2}{d}\frac{\kinr{f}(\bx)}{\densr{f}(\bx)}
       -\frac{1}{d}\bigl|\velr{f}(\bx)\bigr|^2
     = \frac{1}{d \rho(\bx)}
       \int_{\mathbb R^d}
       |\bv-\velr{f}(\bx)|^2 f(\bx,\bv) \diff\bv .
       \label{eq:primitive-from-conserved}
\end{align}
The conserved triple $(\densr{f},\momr{f},\kinr{f})$ and the primitive
triple $(\densr{f},\velr{f},\tempr{f})$ therefore determine each other,
with the inverse relations $\momr{f}=\densr{f}\,\velr{f}$ and
$\kinr{f}=\tfrac12\densr{f}|\velr{f}|^2+\tfrac d2\densr{f}\,\tempr{f}$.
The corresponding local Maxwellian is
\begin{equation}
    \mathcal M[f](\bx,\bv)
    =
    \frac{\rho(\bx)}{(2\pi \tempr{f}(\bx))^{d/2}}
    \exp\!\left(
        -\frac{|\bv-\velr{f}(\bx)|^2}{2\tempr{f}(\bx)}
    \right).
    \label{eq:local-maxwellian}
\end{equation}
For each fixed \(\bx\), the local Maxwellian
\(\mathcal M[f](\bx,\cdot)\) is the unique minimizer of the Boltzmann
\(H\)-functional:
\[
    H(g):=\int_{\mathbb R^d} g\log g\,\diff v
\]
over the nonnegative velocity distributions \(g\) with the same mass, momentum, and kinetic energy as \(f(\bx,\cdot)\). For collision operators satisfying the corresponding \(H\)-theorem, the collision contribution to \(H\) is nonincreasing, and its entropy production vanishes precisely
at local Maxwellian equilibria
~\cite{cercignani1988boltzmann,villani2002review}.

\subsection{PIC representation and field solve}
\label{sec:PIC}
PIC methods~\cite{birdsall2018plasma,hockney2021computer} are the standard
approach to large-scale Vlasov-Poisson simulation. In a PIC approximation, the
distribution function is represented by a weighted empirical measure of
\(N_p\) macro-particles,
\begin{equation}
    f^{N_p}(\bx,\bv,t)
    =
    \sum_{i=1}^{N_p}
    w_i\,
    \delta\bigl(\bx-\bX_t^i\bigr)
    \delta\bigl(\bv-\bV_t^i\bigr),
    \qquad
    \sum_{i=1}^{N_p} w_i=|\mathbb T^d|,
    \label{eq:empirical-measure}
\end{equation}
so the spatially averaged electron density is normalized to one. 
When the system is collisionless, the particles
\(\{(\bX_t^i,\bV_t^i)\}_{i=1}^{N_p}\) follow the differential equations
\begin{align}
    \diff  \bX_t^i &= \bV_t^i \diff  t,
        \label{eq:pic-x}\\
    \diff  \bV_t^i &= \bE(\bX_t^i,t) \diff  t,
        \label{eq:pic-v}
\end{align}
where the self-consistent field \(\bE\) is obtained by solving
the Poisson equation~\eqref{eq:poisson} on a fixed spatial grid. Specifically, let
\(\{\bx_\ell\}_{\ell=1}^{N_x}\) denote the grid nodes 
\begin{equation}
    \rho(\bx_\ell,t)
    =
    \sum_{i=1}^{N_p}
    w_i\,
    S_{\Delta x}\bigl(\bx_\ell-\bX_t^i\bigr)
    =
    \frac{1}{\Delta x^d}
    \sum_{i=1}^{N_p}
    w_i\,
    W\!\left(\frac{\bx_\ell-\bX_t^i}{\Delta x}\right),
    \label{eq:deposition}
\end{equation}
where \(
    S_ {\Delta x}(\bx)=\Delta x^{-d}W(\bx/\Delta x)
\)
is the deposition kernel at scale \(\Delta x\), \(W\) is a nonnegative compactly
supported shape function with \(\int_{\mathbb R^d}W=1\), so that the
deposition conserves the total mass, and \(\Delta x\) denotes the grid
spacing. 
The discrete potential \(\phi\) is then obtained by solving the discrete
Poisson equation with right-hand side \(\rho-1\), for instance by FFT under
periodic boundary conditions. The grid field \(\bE_\ell\) is interpolated back to
the particle locations through
\begin{equation}
    \bE(\bX_t^i,t)
    \approx
    \sum_{\ell=1}^{N_x}
    \bE_\ell(t) 
    W\!\left(\frac{\bx_\ell-\bX_t^i}{\Delta x}\right).
    \label{eq:interpolation}
\end{equation}
\subsection{Collision Closures and Particle Collision Maps}
\label{sec:collisions}

The realization of the collision substep depends on the collision operator.
This choice is independent of the nudging construction, where the feedback derived below is a particle-level correction added
to the PIC update and can be composed with any collision step. Three collision closures are used: the linear Lenard-Bernstein surrogate is used in the well-posedness analysis
(\S\ref{sec:well-posed}), and the BGK and Dougherty-type closures in the
experiments (\S\ref{sec:numerics}). The collision substep acts
within fixed spatial cells, the \emph{collision cells}, which coincide with
the cells of the deposition grid of \S\ref{sec:PIC}.

\paragraph{Linear Lenard-Bernstein surrogate}
The linear Lenard-Bernstein operator~\cite{lenard1958plasma} is
\begin{equation}
    \label{eq:LB-linear}
    Q_{\mathrm{LB}}(f)
    =
    \nu\nabla_\bv\cdot\bigl(\bv f+\Theta_{\mathrm{LB}}\nabla_\bv f\bigr),
    \qquad \nu>0,\ \Theta_{\mathrm{LB}}>0.
\end{equation}
It conserves mass only; acting on the conserved moments
\eqref{eq:conserved-moments}, it relaxes momentum and energy toward a fixed
thermal bath,
\begin{equation}
    \partial_t \momr{f}\big|_{Q_{\mathrm{LB}}}
    =-\nu\,\momr{f},
    \qquad
    \partial_t \kinr{f}\big|_{Q_{\mathrm{LB}}}
    =-2\nu\,\kinr{f}+d\nu\Theta_{\mathrm{LB}}\,\densr{f},
    \label{eq:LB-defect}
\end{equation}
so it does not satisfy \eqref{eq:Q-conservation}. Its velocity-space
particle dynamics are the fixed-bath Ornstein-Uhlenbeck process
\begin{equation}
    \diff\bV_t^i
    =
    -\nu\bV_t^i \diff t
    +
    \sqrt{2\nu\Theta_{\mathrm{LB}}} \diff\bW_t^i ,
    \label{eq:LB-particle}
\end{equation}
whose exact one-step map over a time step $\Delta t$ is
\begin{equation}
    \bV^{n+1}
    =
    e^{-\nu\Delta t}\,\bV^{n}
    +\sqrt{\Theta_{\mathrm{LB}}\bigl(1-e^{-2\nu\Delta t}\bigr)}\,
     \boldsymbol\xi,
    \qquad \boldsymbol\xi\sim\mathcal N(0,I_d).
    \label{eq:LB-map}
\end{equation}
The collision operator therefore enters the particle scheme as an
additional drift-diffusion term, and the collisional system remains a
single stochastic particle system.

\paragraph{BGK relaxation}
The driven experiments of \S\ref{sec:numerics} use the BGK
operator~\cite{bhatnagar1954model}:
\begin{equation}
    \label{eq:Q-BGK}
    Q_{\mathrm{BGK}}(f)
    =
    \nu\bigl(\mathcal M[f]-f\bigr),
    \qquad \nu\ge0,
\end{equation}
where $\mathcal M[f]$ is the local Maxwellian \eqref{eq:local-maxwellian}.
Since $\mathcal M[f]$ has, by construction, the same mass, momentum, and
energy densities as $f$, the difference $\mathcal M[f]-f$ has vanishing
collision-invariant moments, and the continuum operator satisfies
\eqref{eq:Q-conservation}. Its standard particle
realization is a jump map. Over a time step $\Delta t$, each particle is
redrawn, with probability $1-e^{-\nu\Delta t}$, from the Maxwellian with the
moments of its collision cell.

\paragraph{Conservative Dougherty-type map}
The Dougherty operator~\cite{dougherty1964model} is the nonlinear
Fokker-Planck operator
\begin{equation}
    \label{eq:Q-D}
    Q_{\mathrm D}(f)
    =
    \nu \nabla_\bv\cdot
    \bigl[(\bv-\velr{f}) f+\tempr{f} \nabla_\bv f\bigr],
\end{equation}
which satisfies \eqref{eq:Q-conservation}. Our collision step realizes it in
Langevin form with per-cell conservation corrections, in the manner
of~\cite{manheimer1997langevin,lemons2009small}. It freezes the bulk velocity
and temperature $(\bu_\ell,T_\ell)$ of each collision cell, applies the exact
Ornstein-Uhlenbeck update toward the local Maxwellian,
\begin{equation}
    \widetilde{\bV}_i
    =
    \bu_\ell
    +e^{-\nu\Delta t}\bigl(\bV_i-\bu_\ell\bigr)
    +\sqrt{T_\ell\bigl(1-e^{-2\nu\Delta t}\bigr)}\,
     \boldsymbol\xi_i ,
    \label{eq:D-map}
\end{equation}
and then recenters and rescales the peculiar velocities in each cell, so that
mass, momentum, and kinetic energy are conserved exactly cell by cell.

\section{Moment nudging for Vlasov-Poisson particles}
\label{sec:nudging}
\subsection{Observation operator and compatible set}
\label{sec:obs-operator}
Consistent with the mass convention of \S\ref{sec:PIC}, all phase-space measures in this section
carry the fixed total mass $M:=|\mathbb T^d|$, and we write
\[
  \mathcal P_{2,M}
  :=
  \Bigl\{\mu\ \text{a nonnegative measure on }
  \mathbb T^d\times\mathbb R^d:
  \mu(\mathbb T^d\times\mathbb R^d)=M,\
  \textstyle\int|\bv|^2\diff\mu<\infty\Bigr\},
\]
so that the weighted empirical measure \eqref{eq:empirical-measure}
belongs to $\mathcal P_{2,M}$.
For \(\mu\in\mathcal P_{2,M}\), define
\[
  \mathcal O_h\mu
  :=
  \bigl(\rho_h[\mu],\bj_h[\mu],\mathcal K_h[\mu]\bigr),
\]
the \(K_h\)-smoothed counterparts of the conserved moments
\eqref{eq:conserved-moments},
\begin{equation}
\begin{aligned}
  \rho_h[\mu](\bx)
  &:=
  \int_{\mathbb T^d\times\mathbb R^d}
  K_h(\bx-\by)\,\mu(\diff\by,\diff\bv),\\
  \bj_h[\mu](\bx)
  &:=
  \int_{\mathbb T^d\times\mathbb R^d}
  K_h(\bx-\by)\,\bv\,\mu(\diff\by,\diff\bv),\\
  \mathcal K_h[\mu](\bx)
  &:=
  \int_{\mathbb T^d\times\mathbb R^d}
  K_h(\bx-\by)\,\frac12|\bv|^2\,
  \mu(\diff\by,\diff\bv).
\end{aligned}
\label{eq:smoothed-hydrodynamic-moments}
\end{equation}
Here $K_h(\bx)=K_h(-\bx)$ is a smooth even observation kernel at scale $h$ with
unit mass, $\int_{\mathbb T^d}K_h=1$. The subscript $h$ denotes the smoothed
moments, and when the field does not have a subscript it is a raw local moment. Whenever \(\rho_h[\mu]>0\), the smoothed primitive variables are defined
from the smoothed conserved fields by the relations
\eqref{eq:moments}-\eqref{eq:primitive-from-conserved} of \S\ref{sec:VP},
\begin{equation}
  \bu_h[\mu]
  :=
  \frac{\bj_h[\mu]}{\rho_h[\mu]},
  \qquad
  T_h[\mu]
  :=
  \frac{2}{d}\frac{\mathcal K_h[\mu]}{\rho_h[\mu]}
  -
  \frac{1}{d}|\bu_h[\mu]|^2 ,
  \label{eq:smoothed-derived-fields}
\end{equation}
so the two smoothed triples again determine each other. Because the
conserved moments depend linearly on the law \(f\), we formulate the
residuals in conserved variables. Given observed primitive fields
\((\rho^\obs,\bu^\obs,T^\obs)\), the observed conserved fields are
\(\bj^\obs:=\rho^\obs\bu^\obs\) and
\(\mathcal K^\obs:=\tfrac12\rho^\obs|\bu^\obs|^2
+\tfrac d2\rho^\obs T^\obs\), by the inverse relations of \S\ref{sec:VP}.
For a forecast law \(f\), define the moment residuals
\begin{equation}
  r_0[f]
  :=
  \rho_h[f]-\rho^\obs,
  \qquad
  \br_1[f]
  :=
  \bj_h[f]-\bj^\obs,
  \qquad
  r_2[f]
  :=
  \mathcal K_h[f]-\mathcal K^\obs .
  \label{eq:moment-residuals}
\end{equation}
The states that are indistinguishable from the observations form the smoothed
moment-compatible set
\begin{equation}
  \mathcal S_h
  :=
  \left\{
    f\in\mathcal P_{2,M}:
    \rho_h[f]=\rho^\obs,\ 
    \bj_h[f]=\bj^\obs,\ 
    \mathcal K_h[f]=\mathcal K^\obs
  \right\}.
  \label{eq:Sh-manifold}
\end{equation}
Equivalently, $\mathcal S_h$ is the zero set of $(r_0,\br_1,r_2)$, so any
feedback built from these residuals vanishes on all of $\mathcal S_h$. We
assume throughout that the observations are generated from the reference solution by the
same smoothed moment operator,
\begin{equation}
  \rho^\obs=\rho_h[f^\dagger],
  \qquad
  \bj^\obs=\bj_h[f^\dagger],
  \qquad
  \mathcal K^\obs=\mathcal K_h[f^\dagger],
  \label{eq:consistent-observations}
\end{equation}
so that the compatible set contains the true solution, $f^\dagger\in\mathcal S_h$.

\subsection{Limits of Identifiability}
\label{sec:obstruction}
The limits of identifiability under moment observations depends on the unobserved directions. The usual principle behind continuous data assimilation is that partial observations synchronize a system when the
unobserved directions are stable, dissipated, or otherwise determined by the
observed trajectory~\cite{azouani2014continuous,titi2024inadequacy}. In a
collisional kinetic equation, the intrinsic dissipation in velocity space is supplied by the collision operator, so the relevant question is which directions collisions damp and which they leave untouched.  For
moment-conserving collision models, the collision invariants are the linear
combinations of \(1,\bv,|\bv|^2\), and the zero entropy-production states are the local Maxwellians
\(\mathcal M_{\rho,\bu,T}(\bv)\)
\cite{cercignani1988boltzmann,villani2002review,golse2005boltzmann}.
Linearized around such a Maxwellian, the collision operator annihilates the
span of these invariant directions and, under the usual near-equilibrium
spectral-gap hypotheses, is coercive on the orthogonal
complement~\cite{mouhot2006explicit}. Coercivity lets collisions damp the
non-hydrodynamic velocity-shape error, while the null space leaves errors
in the hydrodynamic variables undamped.

\begin{proposition}[Two-Maxwellian obstruction]
\label{prop:obs-necessary}
On the neutral torus $\mathbb T^d\times\mathbb R^d$, consider the assimilated
Vlasov-Poisson dynamics
\begin{equation}\label{eq:assim}
    \partial_t f + \bv\cdot\nabla_\bx f + \field{f}\cdot\nabla_\bv f
    = Q[f] + \mathcal N\bigl(f,\mathcal O[f]-\mathcal O[f^\dagger]\bigr),
\end{equation}
where $\mathcal O$ is the observation map, $\mathcal N$ is the feedback
operator, which may depend on the current state $f$ but only depends on the reference
$f^\dagger$ through the observation residual, and $f^\dagger$ solves the unassimilated system
\eqref{eq:VPC}-\eqref{eq:poisson}. Assume:
\begin{enumerate}[label=\textup{(A\arabic*)}]
    \item\label{A:coll} $Q[\mathcal M_{1,\bu,T}]=0$ for every homogeneous
        Maxwellian $\mathcal M_{1,\bu,T}$;
    \item\label{A:fb} $\mathcal N(f,0)=0$ for every admissible $f$;
    \item\label{A:deg} there exist $(\bu_1,T_1)\neq(\bu_2,T_2)$ with
        \begin{equation}\label{eq:deg}
            \mathcal O[\mathcal M_{1,\bu_1,T_1}]
            = \mathcal O[\mathcal M_{1,\bu_2,T_2}];
        \end{equation}
    \item\label{A:wp} the Cauchy problems for the unassimilated and the
        assimilated equations admit unique solutions in a class containing
        the homogeneous Maxwellians.
\end{enumerate}
Let $f^\dagger(0)=\mathcal M_{1,\bu_1,T_1}$ and
$f(0)=\mathcal M_{1,\bu_2,T_2}$. Then
\[
    f(t)\equiv\mathcal M_{1,\bu_2,T_2},
    \qquad
    f^\dagger(t)\equiv\mathcal M_{1,\bu_1,T_1},
    \qquad t\ge 0;
\]
in particular,
$d\bigl(f(t),f^\dagger(t)\bigr)
=d\bigl(\mathcal M_{1,\bu_2,T_2},\mathcal M_{1,\bu_1,T_1}\bigr)$
is constant in time for every metric $d$ on phase-space densities.
\end{proposition}

\begin{corollary}[Necessity of moment observability]
\label{cor:moment-obs}
Under \ref{A:coll}-\ref{A:fb} and \ref{A:wp}, a necessary condition
for the assimilated dynamics \eqref{eq:assim} to synchronize from arbitrary
Maxwellian initializations is that $\mathcal O$ separate the homogeneous
Maxwellians, i.e.\ that the map
$(\bu,T)\mapsto \mathcal O[\mathcal M_{1,\bu,T}]$ be injective. In
particular, an observation map that does not distinguish bulk velocity
and temperature on homogeneous Maxwellians cannot synchronize.
\end{corollary}

The proofs are in \ref{app:proof:obs-necessary} and
\ref{app:proof:moment-obs}. The three methods below construct feedback directly from the
density, momentum, and kinetic-energy residuals in
\eqref{eq:moment-residuals}.

\subsection{Quadratic mismatch and Method A}
\label{sec:formA}

Methods~A and~B use the quadratic moment-mismatch functional:
\begin{equation}
  J_M[f]
  :=
  \frac{\gamma_1}{2}
  \|r_0[f]\|_{L^2}^2
  +
  \frac{\gamma_2}{2}
  \|\br_1[f]\|_{L^2}^2
  +
  \frac{\gamma_3}{2}
  \|r_2[f]\|_{L^2}^2,
  \qquad
  \gamma_1,\gamma_2,\gamma_3>0 .
  \label{eq:moment-loss}
\end{equation}
Since the maps
\(f\mapsto\rho_h[f]\), \(f\mapsto\bj_h[f]\), and
\(f\mapsto\mathcal K_h[f]\) are linear, and since \(K_h\) is even, the first
variation of \(J_M\) is given by
\begin{equation}
  \Phi_M[f](\bx,\bv)
  :=
  \frac{\delta J_M}{\delta f}(\bx,\bv)
  =
  \gamma_1 K_h*r_0[f](\bx)
  +
  \gamma_2 \bv\cdot K_h*\br_1[f](\bx)
  +
  \frac{\gamma_3}{2}|\bv|^2 K_h*r_2[f](\bx).
  \label{eq:moment-first-variation}
\end{equation}
 It remains to pick a descent geometry that maps \(\Phi_M[f]\) into a particle-level
correction. 

The most direct way is to set $J_M$ to be a Wasserstein gradient-flow energy and pick the descent direction as the negative phase-space gradient of $\Phi_M$. From the
Jordan-Kinderlehrer-Otto~\cite{jordan1998variational} perspective, this corresponds to the
minimizing-movement method
\[
    f_{k+1}
    =
    \operatorname*{arg min}_{f \in \mathcal P_{2,M}}
    \left\{
        J_M[f]
        +
        \frac{1}{2\tau} W_2^2(f,f_k)
    \right\} .
\]
As $\tau\to 0$, the corresponding particle descent direction is $-\nabla_{\bx,\bv}\Phi_M$, with
\[
\begin{pmatrix}
    \nabla_\bx \Phi_M\\
    \nabla_\bv \Phi_M
\end{pmatrix}
=
\begin{pmatrix}
    &\gamma_1 \nabla_\bx (K_h*r_0)
    +
    \gamma_2
    \bigl(\nabla_\bx(K_h*\br_1)\bigr)^{\!\top} \bv
    +
    \frac{\gamma_3}{2}|\bv|^2
    \nabla_\bx(K_h*r_2)\\
    &\gamma_2  K_h*\br_1
    +
    \gamma_3  \bv  K_h*r_2 .
\end{pmatrix}
\]
However, adding this descent direction directly at the particle level produces an SDE that fails the global Lipschitz assumptions underlying standard well-posedness theory. The issue is the kinetic-energy moment, since the position-component drift carries the term $\frac{\gamma_3}{2}|\bv|^2\nabla_\bx(K_h*r_2)$, which is quadratic in $\bv$ and therefore not globally Lipschitz in $(\bx,\bv)$. To resolve this within the variational framework, we modify the geometry of $\mathcal P_{2,M}$ rather than the functional $J_M$.
Recall the Benamou-Brenier dynamic formulation ~\cite{benamou2000computational}, for
measures \(f_0,f_1\in \mathcal P_{2,M}\),
\[
W_2^2(f_0,f_1)
=
\inf_{(\mu_s,\xi_x,\xi_v)}
\int_0^1\int_{\mathbb T^d\times\mathbb R^d}
\left(|\xi_x|^2+|\xi_v|^2\right) \diff\mu_s(x,v) \diff s,
\]
where the infimum is over narrowly continuous curves
\(\mu_s\in\mathcal P_{2,M}\) and velocity fields
\((\xi_x,\xi_v)\) satisfying
\[
\partial_s\mu_s+\nabla_x\cdot(\mu_s\xi_x)
+\nabla_v\cdot(\mu_s\xi_v)=0
\]
in the sense of distributions, with endpoint conditions
\(\mu_0=f_0\) and \(\mu_1=f_1\).
The formula above suggests a geometric way to control the \(|\bv|^2\)
growth, where it penalizes large velocity by increasing the cost of spatial transport. We gauge the size of a velocity by the peculiar speed relative to the
observed bulk flow, $|\bv-\bu^\obs|$, where the field
$\bu^\obs$ is supplied by the observations. The peculiar speed is Galilean invariant and on the moment-compatible set it
coincides with the intrinsic peculiar speed $|\bv-\vel{f}|$.
At each physical time \(t\), freeze the observed field
\(\bu^\obs(\cdot,t)\) and define the pointwise metric tensor
\[
    G_t(\bx,\bv)
    :=
    \operatorname{diag}\!\left(
        \left(
            1+\frac{|\bv-\bu^\obs(\bx,t)|^2}{V_*^2}
        \right)I_d, 
        I_d
    \right).
\]
Here \(V_*>0\) is a reference thermal speed, for instance chosen
from the observed temperature scale, or \(V_*=1\) in normalized
units.
For tangent velocity fields
\(
    \boldsymbol\xi=(\xi_x,\xi_v),
    \boldsymbol\eta=(\eta_x,\eta_v),
\)
we define
\[
    g_{f,t}(\boldsymbol\xi,\boldsymbol\eta)
    :=
    \int_{\mathbb T^d\times\mathbb R^d}
    \left\langle
        G_t(\bx,\bv)\boldsymbol\xi(\bx,\bv),
        \boldsymbol\eta(\bx,\bv)
    \right\rangle
    f(\bx,\bv) \diff\bx \diff\bv ,
\]
which equips \(\mathcal P_{2,M}\) with a formal time-dependent weighted
Wasserstein metric. 
Therefore, the weighted metric gradient of the moment mismatch is
\[
    \operatorname{grad}_{G_t}J_{M,t}[f]
    =
    G_t^{-1}\nabla_{\bx,\bv}\Phi_{M,t}[f],
\]
and the corresponding steepest-descent velocity is
\begin{equation}
    \begin{aligned}
    \bb_A[f](\bx,\bv,t)
    &:=
    -\operatorname{grad}_{G_t}J_{M,t}[f](\bx,\bv)
    \\[0.3em]
    &=
    -G_t^{-1}(\bx,\bv)
    \nabla_{\bx,\bv}\Phi_{M,t}[f](\bx,\bv)
    \\
    &=
    \left(
        -\frac{
            \nabla_{\bx}\Phi_{M,t}[f](\bx,\bv)
        }{
            1+|\bv-\bu^\obs(\bx,t)|^2/V_*^2
        },
        \;
        -\nabla_{\bv}\Phi_{M,t}[f](\bx,\bv)
    \right).
    \end{aligned}
    \label{eq:weighted-descent-velocity}
\end{equation}
The factor $(1+|\bv-\bu^\obs|^2/V_*^2)^{-1}$ absorbs the quadratic-in-$\bv$ growth of $\nabla_\bx\Phi_M$. In fact, since $|\bv|^2\le 2|\bv-\bu^\obs|^2+2|\bu^\obs|^2$, we have that as long as $\bu^\obs$ is bounded, the destabilizing term remains uniformly controlled.

Denoting $\bb_A=(\bb_{A,\bx},\bb_{A,\bv})$ for the two components of
\eqref{eq:weighted-descent-velocity}, the particle realization of
Method~A evolves the weighted empirical measure $\mu_t^{N_p}$ of
\eqref{eq:empirical-measure} by
\begin{align}
    \diff \bX_t^i
    &=
    \bigl[\bV_t^i
    +\bb_{A,\bx}[\mu_t^{N_p}](\bX_t^i,\bV_t^i,t)\bigr]\diff t,
    \label{eq:A-sde-x}\\
    \diff \bV_t^i
    &=
    \bigl[\field{\mu_t^{N_p}}(\bX_t^i)
    +\bb_{A,\bv}[\mu_t^{N_p}](\bX_t^i,\bV_t^i,t)\bigr]\diff t ,
    \label{eq:A-sde-v}
\end{align}
and the collision step is composed with this update, using
the particle collision maps of \S\ref{sec:collisions}.

\begin{lemma}[Identification of the moment-compatible set]
\label{lem:Sh-identification}
Let $f\in\mathcal P_{2,M}$ be absolutely continuous with respect to
$\diff\bx\diff\bv$, and denote its density again by $f$. Assume that
$\densr{f}(\bx)>0$ for almost every
$\bx\in\mathbb T^d$. Then $f\in\mathcal S_h$ if and only if
$\nabla_{\bx,\bv}\Phi_M[f]$ vanishes almost everywhere with respect to
the phase-space measure $f(\bx,\bv)\diff\bx\diff\bv$.
\end{lemma}
The proof is in \ref{app:proof:Sh-identification}.

With the observations frozen, the contribution of the feedback term to the dynamics in Method~A is as follows: 
\[
    \partial_t f_t
    +
    \nabla_{\bx,\bv}\cdot
    \bigl(f_t\bb_A[f_t]\bigr)
    =0.
\]
We have that
\begin{equation*}
\begin{aligned}
\dfrac{\diff}{\diff t} J_M[f_t] &= \int_{\mathbb T^d \times \mathbb R^d} \dfrac{\delta J_M}{\delta f_t} \partial_t f_t \diff \bx \diff \bv\\
&= -\int_{\mathbb T^d \times \mathbb R^d} \Phi_M[f_t] \nabla_{\bx,\bv} \cdot (f_t\bb_A[f_t])\diff \bx \diff \bv\\
&= \int_{\mathbb T^d \times \mathbb R^d} f_t \nabla_{\bx,\bv} \Phi_M[f_t] \cdot \bb_A[f_t]\diff \bx \diff \bv\\
&= -\int_{\mathbb T^d \times \mathbb R^d} f_t \langle G_tG_t^{-1}\nabla_{\bx,\bv}\Phi_M[f_t],G_t^{-1}\nabla_{\bx,\bv}\Phi_M[f_t]\rangle \diff \bx \diff \bv\\
&= -g_{f_t}(\operatorname{grad}_{g}J_M[f_t],\operatorname{grad}_{g}J_M[f_t])\\
&= -\|\operatorname{grad}_{g}J_M[f_t]\|^2_{g_{f_t}}\\
&\leq 0,
\end{aligned}
\end{equation*}
where the boundary terms vanished due to periodic boundary conditions on $\bx \in \mathbb T^d$ and the finite second moment of $\bv$, which handles the $\bv$-boundary at infinity. Under the assumptions of Lemma~\ref{lem:Sh-identification}, equality
holds when $f_t\in\mathcal S_h$.

\subsection{Direction-split Method B}
\label{sec:formB}
Method~B retains the standard $W_2$ geometry and assigns different
components of the moment mismatch to different phase-space directions, that is, the
spatial correction is driven by the density residual, and the velocity
correction by the momentum and kinetic-energy residuals. We split the
mismatch into its density part and its momentum-energy part,
\[
    J_M^{\rho}[f] = \frac{\gamma_1}{2}\|r_0\|_{L^2}^2,
    \qquad
    J_M^{\bj,\mathcal K}[f]
    = \frac{\gamma_2}{2}\|\br_1\|_{L^2}^2
    + \frac{\gamma_3}{2}\|r_2\|_{L^2}^2,
\]
so that \(J_M=J_M^{\rho}+J_M^{\bj,\mathcal K}\), with first variations
\[
\begin{aligned}
    \Phi_M^{\rho}[f](\bx)
    &:= \frac{\delta J_M^{\rho}}{\delta f}
     = \gamma_1 K_h*r_0[f](\bx),\\
    \Phi_M^{\bj,\mathcal K}[f](\bx,\bv)
    &:= \frac{\delta J_M^{\bj,\mathcal K}}{\delta f}
    = \gamma_2 \bv\cdot K_h*\br_1[f](\bx)
    + \frac{\gamma_3}{2}|\bv|^2 K_h*r_2[f](\bx).
\end{aligned}
\]
The split dynamics add the $\bx$-gradient of the first and the
$\bv$-gradient of the second to the Vlasov characteristics, where the particle
realization dynamics are governed by the weighted empirical measure $\mu_t^{N_p}$ of
\eqref{eq:empirical-measure} by
\begin{align}
    \diff \bX_t^i
    &= \bigl[\bV_t^i
       -\nabla_\bx\Phi_M^{\rho}[\mu_t^{N_p}](\bX_t^i)\bigr]\diff t,
    \label{eq:split-x}\\
    \diff \bV_t^i
    &= \bigl[\field{\mu_t^{N_p}}(\bX_t^i)
       -\nabla_\bv\Phi_M^{\bj,\mathcal K}[\mu_t^{N_p}](\bX_t^i,\bV_t^i)\bigr]
       \diff t
    \nonumber\\
    &= \bigl[\field{\mu_t^{N_p}}(\bX_t^i)
       -\gamma_2\,K_h*\br_1[\mu_t^{N_p}](\bX_t^i)
       -\gamma_3\,\bV_t^i\,K_h*r_2[\mu_t^{N_p}](\bX_t^i)\bigr]\diff t,
    \label{eq:split-v}
\end{align}
and the collision step is composed with this update by splitting, as in
Method~A. The position drift is velocity-independent and the velocity
drift is affine in $\bv$, so the quadratic growth of \S\ref{sec:formA} does not appear.

\subsection{Entropic Method C}
\label{sec:formC}
Methods~A and~B measure the moment mismatch in the Euclidean geometry
of \eqref{eq:moment-loss}. Method~C instead uses an
information-geometric measure of discrepancy. Let
\(
    \pi^\obs := \mathcal M_{\rho^\obs,\bu^\obs,T^\obs}
\)
be the local Maxwellian carrying the observed fields. We assume throughout
that $\rho^\obs$ and $T^\obs$ are bounded away from zero, so that the
logarithms and inverse temperatures below are well defined.
The most direct entropy-based choice is the full phase-space relative
entropy $\KL(f\|\pi^\obs)$, where for nonnegative densities
$p,q$, we use the \emph{generalized relative entropy}
\begin{equation}
    \KL(p \| q):=\int\Bigl(p \log\tfrac pq-p+q\Bigr)\ \ge\ 0,
    \label{eq:gen-kl}
\end{equation}
which coincides with the Kullback-Leibler divergence when $p$ and $q$ have
equal total mass. However, this choice fails at the level of its zero set. By Gibbs' inequality,
$\KL(f\|\pi^\obs)=0$ if and only if $f=\pi^\obs$, which is a single state rather than
the moment-compatible set $\mathcal S_h$, so its descent flow keeps deforming
a non-Maxwellian state whose moments already match the observations, driving
the unresolved velocity-space shape toward $\pi^\obs$. We therefore separate the shape component from the moment mismatch.

\begin{proposition}[Moment-shape splitting of the relative entropy]
\label{prop:pythagoras}
For every $f$ with positive density and finite second moments,
\begin{equation}
    \KL(f \| \pi^\obs)
    =
    \underbrace{\KL(f \| \mathcal M[f])}_{\text{shape (non-Maxwellianity)}}
    \;+\;
    \underbrace{\KL(\mathcal M[f] \| \pi^\obs)}_{\text{moment mismatch}} .
    \label{eq:pythagoras}
\end{equation}
\end{proposition}
The proof is in \ref{app:proof:pythagoras}.
We set the moment-mismatch part as the design
functional, with a temperature floor $\varepsilon\ge0$, since the temperature
enters the denominators and logarithms below and may degenerate. For fixed
$\varepsilon\ge0$, set
\[
    \regtemp{\mu}:=\temp{\mu}+\varepsilon,
    \qquad
    \Theta^\obs:=T^\obs+\varepsilon,
\]
and let $\pi^\obs_\varepsilon:=\mathcal M_{\rho^\obs,\bu^\obs,\Theta^\obs}$ and
$\mathcal M_{h,\varepsilon}[\mu]:=
\mathcal M_{\dens{\mu},\vel{\mu},\regtemp{\mu}}$, so that we can define
\begin{equation}
\label{eq:G-definition-1}
\begin{aligned}
G_\varepsilon[\mu]
    :=&
    \KL\bigl(\mathcal M_{h,\varepsilon}[\mu]\big\|\pi^\obs_\varepsilon\bigr)\\
    =& \int_{\mathbb T^d \times \mathbb R^d} \mathcal M_{h,\varepsilon}[\mu] \log \tfrac{\mathcal M_{h,\varepsilon}[\mu]}{\pi^{\obs}_\varepsilon}-\mathcal M_{h,\varepsilon}[\mu]+\pi^{\obs}_\varepsilon \diff \bx\diff\bv.
\end{aligned}
\end{equation}
We have that
\begin{equation}
\label{eq:log-ratio}
\begin{aligned}
\log \tfrac{\mathcal M_{h,\varepsilon}[\mu]}{\pi^{\obs}_\varepsilon} = \log\tfrac{\rho_h[\mu]}{\rho^\obs}-\tfrac{d}{2}\log\tfrac{\Theta[\mu]}{\Theta^\obs}-\tfrac{|\bv - \bu_h[\mu]|^2}{2\Theta[\mu]}+\tfrac{|\bv-\bu^\obs|^2}{2\Theta^\obs}.    
\end{aligned}
\end{equation}
The first two terms on the right-hand side are independent of $\bv$, so 
\[
\int_{\mathbb R^d}\mathcal M_{h,\varepsilon}[\mu]\left(\log\tfrac{\rho_h[\mu]}{\rho^\obs}-\tfrac{d}{2}\log\tfrac{\Theta[\mu]}{\Theta^\obs}\right)\diff \bv = \left(\log\tfrac{\rho_h[\mu]}{\rho^\obs}-\tfrac{d}{2}\log\tfrac{\Theta[\mu]}{\Theta^\obs}\right)\rho_h[\mu].
\]
For the third term on the right-hand side, 
we have
\[
\int_{\mathbb R^d} \mathcal M_{h,\varepsilon}[\mu]\tfrac{|\bv-\bu_h[\mu]|^2}{2\Theta[\mu]}\diff \bv = \tfrac{d\rho_h[\mu]}2{}
\]
by \eqref{eq:primitive-from-conserved} and \eqref{eq:smoothed-derived-fields}.
Additionally, for the last term,
\begin{equation}
\label{eq:log-fourth-term}
\begin{aligned}
&\int_{\mathbb R^d} \mathcal M_{h,\varepsilon}[\mu]\tfrac{|\bv-\bu^\obs|^2}{2\Theta^\obs}
\diff\bv\\= &\int_{\mathbb R^d} \mathcal M_{h,\varepsilon}[\mu]\left(\tfrac{|\bv-\bu_h[\mu]|^2}{2\Theta^\obs}+\tfrac{(\bv-\bu_h[\mu])\cdot(\bu_h[\mu]-\bu^\obs)}{\Theta^\obs}+\tfrac{|\bu_h[\mu]-\bu^\obs|^2}{2\Theta^\obs}\right)
\diff\bv.
\end{aligned}
\end{equation}
The first term on the right-hand side of \eqref{eq:log-fourth-term} becomes $\tfrac{d\rho_h[\mu]}{2}\tfrac{\Theta[\mu]}{\Theta^\obs}$ by \eqref{eq:primitive-from-conserved} and \eqref{eq:smoothed-derived-fields}. The cross term in \eqref{eq:log-fourth-term} becomes
\begin{equation*}
\begin{aligned}
\tfrac{(\bu_h[\mu]-\bu^\obs)}{\Theta^\obs}\cdot\int_{\mathbb R^d}\mathcal M_{h,\varepsilon}[\mu](\bv-\bu_h[\mu])\diff\bv = \tfrac{(\bu_h[\mu]-\bu^\obs)}{\Theta^\obs}\cdot\int_{\mathbb R^d}\bj_h[\mu]-\bu_h[\mu]\rho_h[\mu]\diff\bv=0,
\end{aligned}
\end{equation*}
where we used that $\mathcal M_{h,\varepsilon}[\mu]$ and $\mu$ have the same moments $\rho_h[\mu]$ and $\bj_h[\mu]$ and also \eqref{eq:smoothed-derived-fields}. For the last term on the right-hand side of \eqref{eq:log-fourth-term}, we have
\[
\tfrac{|\bu_h[\mu]-\bu^\obs|^2}{2\Theta^\obs}\int_{\mathbb R^d} \mathcal M_{h,\varepsilon}[\mu]\diff\bv = \tfrac{|\bu_h[\mu]-\bu^\obs|^2}{2\Theta^\obs}\rho_h[\mu].
\]
Hence, simplifying \eqref{eq:G-definition-1} yields the equation
\begin{equation}
\begin{aligned}
    \mathcal G_\varepsilon[\mu]
    :=&
    \KL\bigl(\mathcal M_{h,\varepsilon}[\mu]\big\|\pi^\obs_\varepsilon\bigr)\\
    =&
    \int_{\mathbb T^d}\!
    \Bigl[
        \dens{\mu}\log\tfrac{\dens{\mu}}{\rho^\obs}-\dens{\mu}+\rho^\obs
        +\tfrac{\dens{\mu}\,|\vel{\mu}-\bu^\obs|^2}{2\Theta^\obs}
        \\
        &+\tfrac{d}{2}\dens{\mu}
         \Bigl(\tfrac{\regtemp{\mu}}{\Theta^\obs}-1
         -\log\tfrac{\regtemp{\mu}}{\Theta^\obs}\Bigr)
    \Bigr]\diff\bx .
\end{aligned}
    \label{eq:Ch-def}
\end{equation}
By \eqref{eq:gen-kl} and Gibbs' inequality, $\mathcal G_\varepsilon[\mu]=0$ if and only if
$\mathcal M_{h,\varepsilon}[\mu]=\pi^\obs_\varepsilon$, that is,
$\mu\in\mathcal S_h$ by \eqref{eq:smoothed-derived-fields} since two Maxwellians are equal if and only if their defining moments $(\rho, \bj, \mathcal K)$, or their primitive fields $(\rho, \bu, T)$, agree. For every
$\varepsilon\ge0$, the functional is moment-compatible, while
$\KL(f\|\pi^\obs)$ is not.

Let $\tilde{F}(\rho,\bj,\mathcal K) = F(\rho,\bu(\rho,\bj), \Theta(\rho,\bj,\mathcal K))$ be the integrand of \eqref{eq:Ch-def}. Since the smoothed moments \eqref{eq:smoothed-hydrodynamic-moments} are
linear in $\mu$ and $K_h$ is even, differentiating \eqref{eq:Ch-def} gives:
\begin{equation*}
\begin{aligned}
\Psi_\varepsilon[\mu]
:=
\frac{\delta\mathcal G_\varepsilon}{\delta\mu}
&= \int_{\mathbb T^d} K_h(\by-\bx)\partial_\rho \tilde F + K_h(\by-\bx)\bv\cdot \partial_\bj \tilde F + \tfrac{1}{2}|\bv|^2K_h(\by-\bx)\partial_{\mathcal K} \tilde F \diff\by\\
&= K_h*(\partial_\rho \tilde F+ \bv \cdot \partial_\bj \tilde F + \tfrac{1}{2}|\bv|^2 \partial_{\mathcal K} \tilde F)
\end{aligned}
\end{equation*}
We have that
\begin{equation*}
\begin{aligned}
\partial_\rho F &= \log \tfrac{\rho_h[\mu]}{\rho^\obs}+\tfrac{|\bu_h[\mu]-\bu^\obs|^2}{2\Theta^\obs} +\tfrac{d}{2}\left(\tfrac{\Theta[\mu]}{\Theta^\obs}-1-\log\tfrac{\Theta[\mu]}{\Theta^\obs}\right)\\
\partial_\bu F &= \tfrac{\rho_h[\mu](\bu_h[\mu]-\bu^\obs)}{\Theta^\obs}\\
\partial_\Theta F &= \tfrac{d}{2}\rho_h[\mu]\left(\tfrac{1}{\Theta^\obs}-\tfrac{1}{\Theta[\mu]}\right),  
\end{aligned}
\end{equation*}
so that
\begin{equation*}
\begin{aligned}
\partial_\rho \tilde F =& \partial_\rho F + \partial_\bu F \cdot \partial_\rho \bu + \partial_\Theta F \partial_\rho \Theta\\
=& \log \tfrac{\rho_h[\mu]}{\rho^\obs}+\tfrac{|\bu_h[\mu]-\bu^\obs|^2}{2\Theta^\obs} +\tfrac{d}{2}\left(\tfrac{\Theta[\mu]}{\Theta^\obs}-1-\log\tfrac{\Theta[\mu]}{\Theta^\obs}\right) + \tfrac{\rho_h[\mu](\bu_h[\mu]-\bu^\obs)}{\Theta^\obs}\cdot\left(-\tfrac{\bu}{\rho_h[\mu]}\right) \\ &+ \tfrac{d}{2}\rho_h[\mu]\left(\tfrac{1}{\Theta^\obs}-\tfrac{1}{\Theta[\mu]}\right) \left(-\tfrac{\Theta[\mu]-\varepsilon+\tfrac{1}{d}|\bu_h[\mu]|^2}{\rho_h[\mu]}+\tfrac{2|\bu_h[\mu]|^2}{d\rho_h[\mu]}\right) \\
=&\log\tfrac{\rho_h[\mu]}{\rho^\obs} -\tfrac{|\bu_h[\mu]|^2}{2\Theta[\mu]} +\tfrac{|\Theta^\obs|^2}{2\Theta^\obs} -\tfrac{d}{2}\log\tfrac{\Theta[\mu]}{\Theta^\obs}+\tfrac{d\varepsilon}{2}\left(\tfrac{1}{\Theta^\obs}-\tfrac{1}{\Theta[\mu]}\right)\\
\partial_\bj \tilde F = & (\partial_\bj \bu)^\top \partial_\bu F + \partial_\Theta F\partial_\bj \Theta\\
=& \tfrac{1}{\rho_h[\mu]}\tfrac{\rho_h[\mu](\bu_h[\mu]-\bu^\obs)}{\Theta^\obs} + \tfrac{d}{2}\rho_h[\mu]\left(\tfrac{1}{\Theta^\obs}-\tfrac{1}{\Theta[\mu]}\right) \left(-\tfrac{2\bu_h[\mu]}{d\rho_h[\mu]}\right)\\
=& \tfrac{\bu_h[\mu]}{\Theta[\mu]} - \tfrac{\bu^\obs}{\Theta^\obs}\\
\partial_{\mathcal K}\tilde F = & \partial_\Theta F\partial_{\mathcal K}\Theta\\
=&  \tfrac{d}{2}\rho_h[\mu]\left(\tfrac{1}{\Theta^\obs}-\tfrac{1}{\Theta[\mu]}\right) \tfrac{2}{d\rho_h[\mu]}\\
=& \tfrac{1}{\Theta^\obs}-\tfrac{1}{\Theta[\mu]},
\end{aligned}
\end{equation*}
where we used \eqref{eq:smoothed-derived-fields}.
Hence, by \eqref{eq:log-ratio},
\[
\log \tfrac{\mathcal M_{h,\varepsilon}[\mu]}{\pi^{\obs}_\varepsilon} = \partial_\rho \tilde F - \tfrac{d\varepsilon}{2}\left(\tfrac{1}{\Theta^\obs}-\tfrac{1}{\Theta[\mu]}\right) + 
\bv \cdot \partial_\bj \tilde F + \tfrac{1}{2} |\bv|^2 \partial_{\mathcal K}\tilde F.   
\]
Thus, the simplified first variation is:
\begin{equation}
    \Psi_\varepsilon[\mu]
    :=
    \frac{\delta\mathcal G_\varepsilon}{\delta\mu}
    =
    K_h*
    \Bigl[
        \log\frac{\mathcal M_{h,\varepsilon}[\mu]}{\pi^\obs_\varepsilon}
        +\frac{d\varepsilon}{2}
         \Bigl(\frac{1}{\Theta^\obs}-\frac{1}{\regtemp{\mu}}\Bigr)
    \Bigr],
    \label{eq:Ch-firstvar}
\end{equation}
where the convolution acts in the $\bx$ variable for each fixed $\bv$. We realize $\mathcal G_\varepsilon$ as a gradient flow under the same
velocity-weighted metric introduced for Method~A. With scaling parameter
$\gamma>0$, the descent velocity is
\begin{equation}
    \bb_{C,\bx}[\mu]
    =
    -\gamma\,
    \frac{\nabla_\bx\Psi_\varepsilon[\mu]}
         {1+|\bv-\bu^\obs|^2/V_*^2},
    \qquad
    \bb_{C,\bv}[\mu]
    =
    -\gamma\,K_h*
    \Bigl[
        \frac{\bv-\bu^\obs}{\Theta^\obs}
        -\frac{\bv-\vel{\mu}}{\regtemp{\mu}}
    \Bigr],
    \label{eq:Ch-drift}
\end{equation}
where the position channel carries the density adjustment, with its
quadratic-in-$\bv$ growth removed by the weight as in Method~A. The
velocity channel is affine in $\bv$ with a scalar coefficient; at each $\bx$, it
moves velocities by maps $\bv\mapsto\lambda\bv+\bc$, which send Maxwellians to
Maxwellians and leave $\KL(f\|\mathcal M[f])$ unchanged. It therefore drives the
bulk velocity toward $\bu^\obs$ and the temperature toward $T^\obs$ without
changing the velocity-space shape.
This affine structure is preserved after integration over phase space.
\begin{proposition}[Global moment balances for the nudged Vlasov-Poisson flow]
\label{prop:C-global}
Let $\bu^\obs\in\mathbb R^d$ and $T^\obs>0$ be constant in space and time,
and set $\varepsilon=0$. Let $g_t(\bx,\bv)\ge0$ be a sufficiently regular
solution on $\mathbb T^d\times\mathbb R^d$, with finite second velocity
moments and $\temp{g_t}>0$ pointwise, of
\begin{equation}\label{eq:C-VP}
    \partial_t g
    + \nabla_\bx\!\cdot\!\bigl(g\,(\bv+\bb_{C,\bx}[g])\bigr)
    + \nabla_\bv\!\cdot\!\bigl(g\,(\field{g}+\bb_{C,\bv}[g])\bigr)
    = Q[g],
\end{equation}
where $\bb_{C,\bv}[g]$ is the velocity drift of \eqref{eq:Ch-drift},
$\bb_{C,\bx}[g]$ is any sufficiently regular position drift (in particular
that of \eqref{eq:Ch-drift}), and $Q$ satisfies \eqref{eq:Q-conservation}.
Define
\begin{equation}\label{eq:C-global-moments}
    M:=\int g\diff\bx\diff\bv,
    \qquad
    \overline{\bu}[g]:=\frac1M\int_{\mathbb T^d}\momr{g}\diff\bx,
    \qquad
    \overline{T}[g]:=\frac1{dM}\int
        \bigl|\bv-\overline{\bu}[g]\bigr|^2 g\diff\bx\diff\bv .
\end{equation}
Then, along the flow \eqref{eq:C-VP}:
\begin{enumerate}[label=\textup{(\roman*)}]
    \item\label{it:C-mass} $M$ is conserved, and $M=|\mathbb T^d|$ by the
        solvability of \eqref{eq:poisson};
    \item\label{it:C-u} the global bulk velocity obeys the closed linear
        equation
        \begin{equation}\label{eq:C-global-u}
            \frac{\diff}{\diff t}\overline{\bu}[g_t]
            =-\frac{\gamma}{T^\obs}
              \bigl(\overline{\bu}[g_t]-\bu^\obs\bigr),
        \end{equation}
        hence
        $\overline{\bu}[g_t]-\bu^\obs
         =e^{-\gamma t/T^\obs}\bigl(\overline{\bu}[g_0]-\bu^\obs\bigr)$;
    \item\label{it:C-T} the global temperature obeys
        \begin{equation}\label{eq:C-global-T}
            \frac{\diff}{\diff t}\overline{T}[g_t]
            =-\frac{2\gamma}{T^\obs}
              \bigl(\overline{T}[g_t]-T^\obs\bigr)
            +\frac{2}{dM}\int_{\mathbb T^d}
              \field{g_t}\cdot\momr{g_t}\diff\bx .
        \end{equation}
\end{enumerate}
\end{proposition}
The proof is in \ref{app:proof:C-global}.
Transport, collisions, the position channel, and the self-consistent field
all drop out of the momentum balance, which is why \eqref{eq:C-global-u} is
closed. The temperature balance \eqref{eq:C-global-T} does not drop out and its
electric-work term is the classical Vlasov-Poisson exchange between kinetic
and electrostatic energy.

\section{Well-posedness of the finite-particle scheme}
\label{sec:well-posed}
We verify that the nudged particle systems of \S\ref{sec:nudging}, closed
with the linear Lenard--Bernstein collision surrogate, admit unique global
strong solutions for finite particle number $N_p$, every scaling parameter, and every
method $a\in\{A,B,C\}$, whenever the initial particle velocities have
finite second moment in expectation. This covers the continuous-time particle SDE, with Methods~A and~B taken as
written and Method~C taken with $\varepsilon>0$, under the kernel and
observed-field conditions of Assumption~\ref{ass:wp-reg} below. The
time-discrete PIC implementations of \S\ref{sec:numerics}, which use BGK
and Dougherty collisions, are outside its scope of this analysis.

Throughout this section, the self-consistent field is the
deposition-smoothed Poisson field
\begin{equation}
  \field{\mu}=-\nabla_\bx(-\Delta_\bx)^{-1}\bigl(S_{\Delta x}*B_0\mu-1\bigr),
  \label{eq:true-field}
\end{equation}
which extends the field of \S\ref{sec:VP} to empirical measures, where $B_0\mu = \int_{\mathbb R^d} \mu \diff \bv$. The torus
has unit volume and the deposition kernel has unit mass,
$\int_{\mathbb T^d}S_{\Delta x}=1$, so the source $S_{\Delta x}*B_0\mu-1$ has zero mean for
every probability measure $\mu$ and $(-\Delta_\bx)^{-1}$ is well defined.
Under this normalization, the fixed mass of \S\ref{sec:obs-operator} is
$M=|\mathbb T^d|=1$, so $\mathcal P_{2,M}$ coincides with the space of
probability measures and the particle weights are $w_i=1/N_p$. The
collision operator is the linear Lenard--Bernstein surrogate
\eqref{eq:LB-linear}, whose particle realization is the fixed-bath
Ornstein--Uhlenbeck process \eqref{eq:LB-particle}. The finite
$N_p$-particle system is
\begin{equation}
\begin{aligned}
  \diff\bX_t^{i}&=\bigl[\bV_t^{i}+\bb^a_\bx(t,\bZ_t^{i},\mu_t^{N_p})\bigr]\diff t,\\
  \diff\bV_t^{i}&=\bigl[\field{\mu_t^{N_p}}(\bX_t^{i})+\bb^a_\bv(t,\bZ_t^{i},\mu_t^{N_p})-\nu\bV_t^{i}\bigr]\diff t
            +\sqrt{2\nu\Theta_{\mathrm{LB}}}\,\diff\bW_t^{i},
\end{aligned}
\label{eq:nudged-particles}
\end{equation}
with $\bZ_t=(\bX_t,\bV_t)$,
$\mu_t^{N_p}=\frac1{N_p}\sum_{i=1}^{N_p}\delta_{\bZ_t^{i}}$, and
$\{\bW^i\}$ independent standard Brownian motions, independent of the
initial array $(\bZ_0^{i})_i$. The nudging drifts $\bb^a_\bx,\bb^a_\bv$
are defined in \S\ref{sec:wp-setup} below.

\subsection{Setup and assumptions}
\label{sec:wp-setup}

Fix $h>0$ and, for Method~C, $\varepsilon>0$. The drifts $\bb^a_\bx,\bb^a_\bv$ of
\eqref{eq:nudged-particles} are those of Method~$a$: Methods~A and~B, \eqref{eq:weighted-descent-velocity} and
\eqref{eq:split-x}--\eqref{eq:split-v}, are analyzed as written. For
Method~C, the drifts are those of \eqref{eq:Ch-drift}, generated by
the first-variation potential $\Psi_\varepsilon$ of
\eqref{eq:Ch-firstvar}, which is
\begin{equation}\label{equ:regularized}
\begin{aligned}
     \Psi_{\varepsilon}[\mu]
  =K_h*\biggl(&\log\frac{\dens{\mu}}{\rho^\obs}
   -\frac d2\log\frac{\regtemp{\mu}}{\Theta^{\obs}}
   -\frac{|\bv-\vel{\mu}|^2}{2\regtemp{\mu}}
   +\frac{|\bv-\bu^\obs|^2}{2\Theta^{\obs}}\\
   &+\frac{d\varepsilon}{2}\Bigl(\frac1{\Theta^{\obs}}-\frac1{\regtemp{\mu}}\Bigr)\biggr).
\end{aligned}
\end{equation}

\begin{assumption}[Kernels and observed fields]
\label{ass:wp-reg}
The deposition kernel $S_{\Delta x}$ and the observation kernel $K_h$ lie in
$W^{2,\infty}_\bx$, and the observation kernel has a strictly positive
lower bound,
\begin{equation}
  \kappa_h:=\inf_{\mathbb T^d}K_h>0.
  \label{eq:Kh-floor}
\end{equation}
The observed fields $\rho^\obs,\bu^\obs,T^\obs$ are Borel measurable in
time and lie in $W^{2,\infty}_\bx$ uniformly in time,
\[
  \sup_{t\ge0}\bigl(\|\rho^\obs(t)\|_{W^{2,\infty}_\bx}
   +\|\bu^\obs(t)\|_{W^{2,\infty}_\bx}+\|T^\obs(t)\|_{W^{2,\infty}_\bx}\bigr)<\infty,
\]
the observed density is bounded away from zero uniformly in time,
$\rho^\obs_{\min}:=\inf_{t,\bx}\rho^\obs(t,\bx)>0$, and the velocity scale
of Methods~A and~C is strictly positive, $V_*>0$.
\end{assumption}

The a priori moment estimates are stated in terms of the velocity-moment
production coefficients
\begin{equation}
\bar\alpha_A=\bar\alpha_B
:=
\gamma_3
\|\mathcal K^\obs\|_{L^\infty_{t,\bx}},
\qquad
\bar\alpha_C
:=
\frac{\gamma}{\varepsilon}
\left(
1+\sqrt{\frac{\|K_h\|_{L^\infty_\bx}}{\kappa_h}}
\right),
  \label{eq:gain-diss}
\end{equation}
which bound the largest linear-in-$\bv$ growth that the nudging feedback
can generate. No relation between $\bar\alpha_a$ and the collision rate
$\nu$ is assumed, and the finite-time moment bounds and the global
well-posedness below hold for every scaling parameter.

\subsection{From Local to Global Well-Posedness}\label{sec:wp-finite-N}
We first prove that the particle system admits a unique local strong
solution, by verifying that all coefficients of
\eqref{eq:nudged-particles} are locally Lipschitz functions of the
particle configuration.
\begin{lemma}[Field regularity]
\label{lem:field}
Under Assumption~\ref{ass:wp-reg}, the electric field
$\field{\mu}=K_E*B_0\mu$ with
$K_E:=-\nabla_\bx(-\Delta_\bx)^{-1}S_{\Delta x}\in C^1_b(\mathbb T^d)$ satisfies
\[
  \|\field{\mu}\|_{W^{1,\infty}_\bx}\le\|K_E\|_{W^{1,\infty}},
  \qquad
  \|\field{\mu}-\field{\mu'}\|_{L^\infty_\bx}\le\|\nabla K_E\|_\infty\,W_2(\mu,\mu').
\]
\end{lemma}
The proof is in \ref{subapp:lem:field}.

\begin{lemma}[Coefficient bounds on moment-bounded sets]
\label{lem:reg-bounds}
Let Assumption~\ref{ass:wp-reg} hold, fix a method $a\in\{A,B,C\}$
and radius $R_4<\infty$, and let
$\mathcal P_{R_4}:=\{\mu\in\mathcal P_4:\int|\bv|^4\diff\mu\le R_4\}$,
where $\mathcal P_q$ denotes the probability measures on
$\mathbb T^d\times\mathbb R^d$ with finite $q$th velocity moment. Then
there exists $L$, depending on $R_4$, $\varepsilon$, and the fixed data
in Assumption~\ref{ass:wp-reg}, such that, for all $t$, $\bz,\bz'$ and
$\mu,\mu'\in\mathcal P_{R_4}$ the position drift is bounded and Lipschitz,
\begin{equation}
  |\bb^a_\bx(t,\bz,\mu)|\le L,
  \qquad
  |\bb^a_\bx(t,\bz,\mu)-\bb^a_\bx(t,\bz',\mu')|\le L\bigl(|\bz-\bz'|+W_2(\mu,\mu')\bigr),
  \label{eq:wp-x}
\end{equation}
and the velocity drift is affine in $\bv$,
\begin{equation}
  \bb^a_\bv(t,\bz,\mu)=\alpha^a[\mu](t,\bx)\,\bv+\beta^a[\mu](t,\bx),
  \label{eq:wp-v}
\end{equation}
with coefficients regular in $\bx$ and Lipschitz in $\mu$,
\begin{equation}
\begin{aligned}
\|\alpha^a[\mu](t)\|_{W^{1,\infty}_\bx}+\|\beta^a[\mu](t)\|_{W^{1,\infty}_\bx}&\le L,\\
\|\alpha^a[\mu](t)-\alpha^a[\mu'](t)\|_{L^\infty_\bx}
   +\|\beta^a[\mu](t)-\beta^a[\mu'](t)\|_{L^\infty_\bx}&\le L\,W_2(\mu,\mu').
\end{aligned}
  \label{eq:wp-vmu}
\end{equation}
\end{lemma}
The proof is in \ref{app:sub:reg-bounds}.

\begin{lemma}[Local well-posedness of the \(N_p\)-particle system]
\label{lem:local-wp}
Let Assumption~\ref{ass:wp-reg} hold and fix
\(a\in\{A,B,C\}\) and \(N_p\ge1\).
Let
\(
  (\bZ_0^1,\ldots,\bZ_0^{N_p})
\)
be an \(\mathcal F_0\)-measurable random variable with values in
\((\mathbb T^d\times\mathbb R^d)^{N_p}\), independent of
\((\bW^1,\ldots,\bW^{N_p})\). Then the particle system
\eqref{eq:nudged-particles} admits a unique strong solution up to the
blow-uo time
\[
  \tau_\infty:=\lim_{R\to\infty}\tau_R,
  \qquad
  \tau_R:=
  \inf\left\{
    t\ge0:
    \max_{1\le i\le N_p}|\bV_t^i|\ge R
  \right\}.
\]
\end{lemma}
The proof is in \ref{app:sub:local-wp}. There is no moment assumption on the
initial law enters for the local proof, and since the positions take
values in the compact torus, finite-time blow-up can occur only through
the velocities. However, the velocity moments cannot blow up in finite time. For a
particle configuration $\mathbf Z=(\bz^1,\dots,\bz^{N_p})$ with
$\bz^i=(\bx^i,\bv^i)$, write
\[
  m_q^{N_p}(\mathbf Z):=\frac1{N_p}\sum_{i=1}^{N_p}|\bv^i|^q,
  \qquad
  m_q^{N_p}(t):=m_q^{N_p}(\bZ_t).
\]

\begin{lemma}[A priori velocity-moment generator inequality]
\label{lem:moment-generator}
Let Assumption~\ref{ass:wp-reg} hold, fix \(a\in\{A,B,C\}\) and \(q\ge2\),
and for Methods~A and~B, assume \(\gamma_3>0\). Then for every
\(\zeta>0\), there is a constant \(C_q(\zeta)\), independent of \(t\) and
\(N_p\), such that the time-dependent generator \(\mathcal L_{N_p,t}\) of
\eqref{eq:nudged-particles} satisfies, for every particle configuration
\(\mathbf Z\),
\begin{equation}
  \mathcal L_{N_p,t}m_q^{N_p}(\mathbf Z)
  \le
  q(\bar\alpha_a-\nu+\zeta)m_q^{N_p}(\mathbf Z)+C_q(\zeta).
  \label{eq:Mq-generator}
\end{equation}
\end{lemma}
The proof is in \ref{app:sub:moment-generator}.

\begin{corollary}[Finite-time moment bounds and finite-particle non-blow-up]
\label{cor:finite-time-moment}
Under the hypotheses of Lemma~\ref{lem:moment-generator}, fix
\(T_{\mathrm{fin}}<\infty\) and \(\zeta>0\). If the initial data satisfy
\(\mathbb E\,m_q^{N_p}(0)<\infty\), then \(\tau_\infty=\infty\) almost
surely and
\begin{equation}
  \sup_{0\le t\le T_{\mathrm{fin}}}
  \mathbb E\,m_q^{N_p}(t)
  \le
  \bigl(\mathbb E\,m_q^{N_p}(0)+C_q(\zeta)T_{\mathrm{fin}}\bigr)
  \exp\!\left(
    \max\{q(\bar\alpha_a-\nu+\zeta),0\}\,T_{\mathrm{fin}}
  \right)
  <\infty .
  \label{eq:finite-time-Mq-bound}
\end{equation}
\end{corollary}
The proof is in \ref{app:sub:finite-time-moment}.

\begin{proposition}[Global well-posedness of the \(N_p\)-particle system]
\label{prop:wp-finiteN}
Under the hypotheses of Lemma~\ref{lem:moment-generator} with \(q=2\),
assume \(\mathbb E\,m_2^{N_p}(0)<\infty\). Then the \(N_p\)-particle
system \eqref{eq:nudged-particles} has a unique global strong solution on
\([0,\infty)\).
\end{proposition}
The proof is in \ref{app:sub:wp-finiteN}.

\section{Numerical experiments}
\label{sec:numerics}
We test the proposed moment-nudging methods on particle-in-cell
discretizations of the Vlasov-Poisson system. 
Each experiment advances five runs with the same PIC structure, the same
collision model, and the same external driver in the driven test: a
true run, an unassimilated run ($S=\mathrm{none}$), and three nudged
runs ($S=A,B,C$). The four runs $S$ start from a common initial
distribution that differs from the true one, so they differ from one
another only in the feedback. In all Method~C runs, we set
$\varepsilon=10^{-3}$.

For each run $S$, let $\rho^S(t,\bx)$, $\bu^S(t,\bx)$, $T^S(t,\bx)$ denote
the hydrodynamic fields deposited from the particles and converted to
primitive variables by
\eqref{eq:moments}-\eqref{eq:primitive-from-conserved}, and let
$(\rho^\dagger,\bu^\dagger,T^\dagger)$ denote the true fields.
Errors are measured in the root-mean-square norm and its
true-density-weighted variant:
\[
  \|g\|_{\mathrm{rms}}
  :=
  \Bigl(
    \frac{1}{|\mathbb T^d|}
    \int_{\mathbb T^d} |g|^2 \diff \bx
  \Bigr)^{1/2},
  \qquad
  \|g\|_{\mathrm{rms},\rho^\dagger}
  :=
\left(
    \frac{\int_{\mathbb T^d} \rho^\dagger\,|g|^2 \diff \bx}
         {\int_{\mathbb T^d} \rho^\dagger \diff \bx}
  \right)^{1/2}.
\]
The moment errors of run $S$ are
\[
\begin{aligned}
    e_\rho^S(t)
  :=&
  \|\rho^S(t)-\rho^\dagger(t)\|_{\mathrm{rms}},\\
  e_u^S(t)
  :=&
  \|\bu^S(t)-\bu^\dagger(t)\|_{\mathrm{rms},\rho^\dagger},\\
  e_T^S(t)
  :=&
  \|T^S(t)-T^\dagger(t)\|_{\mathrm{rms},\rho^\dagger} .
\end{aligned}
\]
Let $f_h^S$ and $f_h^\dagger$ denote the mass-normalized phase-space
histograms of the particles on a common fixed $(\bx,\bv)$ grid. The
phase-space error is
\[
  e_f^S(t)
  :=
  \|f_h^S(t)-f_h^\dagger(t)\|_{L^2_{\bx,\bv}} .
\]
For a comparison window $[\tau,T_{\mathrm{fin}}]$ and each error channel
$*\in\{\rho,u,T,f\}$, define
\[
  \bar e_*^S(\tau,T_{\mathrm{fin}})
  :=
  \frac{1}{T_{\mathrm{fin}}-\tau}
  \int_\tau^{T_{\mathrm{fin}}} e_*^S(t) \diff t,
  \qquad
  R_*^S(\tau,T_{\mathrm{fin}})
  :=
  \frac{\bar e_*^S(\tau,T_{\mathrm{fin}})}
       {\bar e_*^{\mathrm{none}}(\tau,T_{\mathrm{fin}})} .
\]
All tables report $R_*^S(\tau,T_{\mathrm{fin}})$.
\subsection{Driven recovery with BGK collisions}
\label{sec:num-driven}
The first experiment tests recovery from a wrong initial Maxwellian
under external driving.  On the torus of length \(L=4\pi\), the reference solution is initialized as the weakly perturbed Maxwellian
\[
  f^\dagger_0(x,v)
  =
  \bigl(1+\alpha\cos(k_0x)\bigr)\mathcal M_{1,0,1}(v),
  \qquad
  \alpha=10^{-2},\quad k_0=0.5,
\]
whereas the assimilating simulations start from the wrong Maxwellian
\[
  g_0(x,v)=\bigl(1+\beta\cos(k_0x)\bigr)\mathcal M_{1,0.3,1.5}(v),
\]
with wrong bulk velocity and temperature and a mis-specified density
modulation of amplitude \(\beta\), set per the setup below.
All runs are evolved under the external traveling wave
\[
  E_{\mathrm{ext}}(x,t)=E_0\sin(k_d x-\omega_d t),
  \qquad k_d=1,
\]
and the BGK collisions \eqref{eq:Q-BGK} with collision frequency \(\nu\).
We use two driven BGK setups.  Setup~I has a weak external wave,
moderate collisions, and a density-modulated initialization, with
\[
  E_0=0.05,\qquad \omega_d=1.3,\qquad \nu=0.5,\qquad \beta=0.3 .
\]
Setup~II has a stronger resonant wave, weak collisions, and a spatially
uniform initialization, with
\[
  E_0=0.2,\qquad \omega_d=2.0,\qquad \nu=0.05,\qquad \beta=0 .
\]
For each run, we use \(\Delta t=0.05\) and \(1000\) steps, so
\(T_{\mathrm{fin}}=50\). Setup~I is repeated over five independent particle
initializations. All three methods use unit scaling parameters, and
Methods~A and~C use \(V_*=1\).

Table~\ref{tab:e2} reports the window-averaged improvement ratios
\(R_x^S(10,50)\).  In Setup~I, all three methods reduce the bulk-velocity and temperature errors by an order of magnitude (\(R_u\le0.119\), \(R_T\le0.065\)) and the phase-space error by a factor of four (\(R_f\approx0.25\)). Method~A gives the smallest \(R_\rho\), \(R_u\), and \(R_f\), and Method~C the smallest \(R_T\).
In Setup~II, the recovery is smaller, and the same ordering occurs. The
density channel shows little improvement (\(R_\rho\) between \(0.959\)
and \(1.044\)) and with the uniform initialization (\(\beta=0\)), the unassimilated
density already stays close to the true density, so \(R_\rho\) compares two
small errors.
\begin{table}[t]
\centering\small
\begin{tabular}{@{}llcccc@{}}
\toprule
Setup & Method & \(R_\rho\) & \(R_u\) & \(R_T\) & \(R_f\) \\
\midrule
Setup I: weak driver, moderate BGK
& \textbf{A} & \(\mathbf{0.811}\) & \(\mathbf{0.102}\) & \(0.065\) & \(\mathbf{0.246}\) \\

& B          & \(0.830\) & \(0.107\) & \(0.065\) & \(0.248\) \\
& C          & \(0.821\) & \(0.119\) & \(\mathbf{0.049}\) & \(0.247\) \\
\midrule
Setup II: strong driver, weak BGK
& \textbf{A} & \(\mathbf{0.959}\) & \(\mathbf{0.193}\) & \(0.133\) & \(\mathbf{0.371}\) \\

& B          & \(1.044\) & \(0.217\) & \(0.141\) & \(0.401\) \\
& C          & \(0.996\) & \(0.226\) & \(\mathbf{0.126}\) & \(0.379\) \\
\bottomrule
\end{tabular}
\caption{Driven recovery from a wrong initial Maxwellian
(\(N_p=5\times10^5\), \(N_x=128\), \(\Delta t=0.05\), \(T_{\mathrm{fin}}=50\)).
Each entry is \(R_*^S(10,50)\), the mean error of method \(S\) over \(t\in[10,50]\) relative to the unassimilated run.
Setup~I entries are means over five particle initializations
(seed-to-seed standard deviation at most \(0.011\)).  Setup~I uses a weak
external wave, moderate BGK collisions, and a density-modulated initialization,
\((E_0,k_d,\omega_d,\nu,\beta)=(0.05,1,1.3,0.5,0.3)\).  Setup~II uses a
stronger resonant wave, weaker collisions, and a uniform initialization,
\((E_0,k_d,\omega_d,\nu,\beta)=(0.2,1,2.0,0.05,0)\).}
\label{tab:e2}
\end{table}

The error curves (Figure~\ref{fig:e2-errors}) show that the nudged methods reduce the bulk-velocity and temperature errors rapidly in both regimes. The profile snapshots at \(t=25\) (Figure~\ref{fig:e2-profiles}) show that the nudged runs move the bulk velocity and temperature close to the true values, while the unassimilated run remains near its wrong initial Maxwellian.

\begin{figure}
    \centering
    \includegraphics[width=\linewidth]{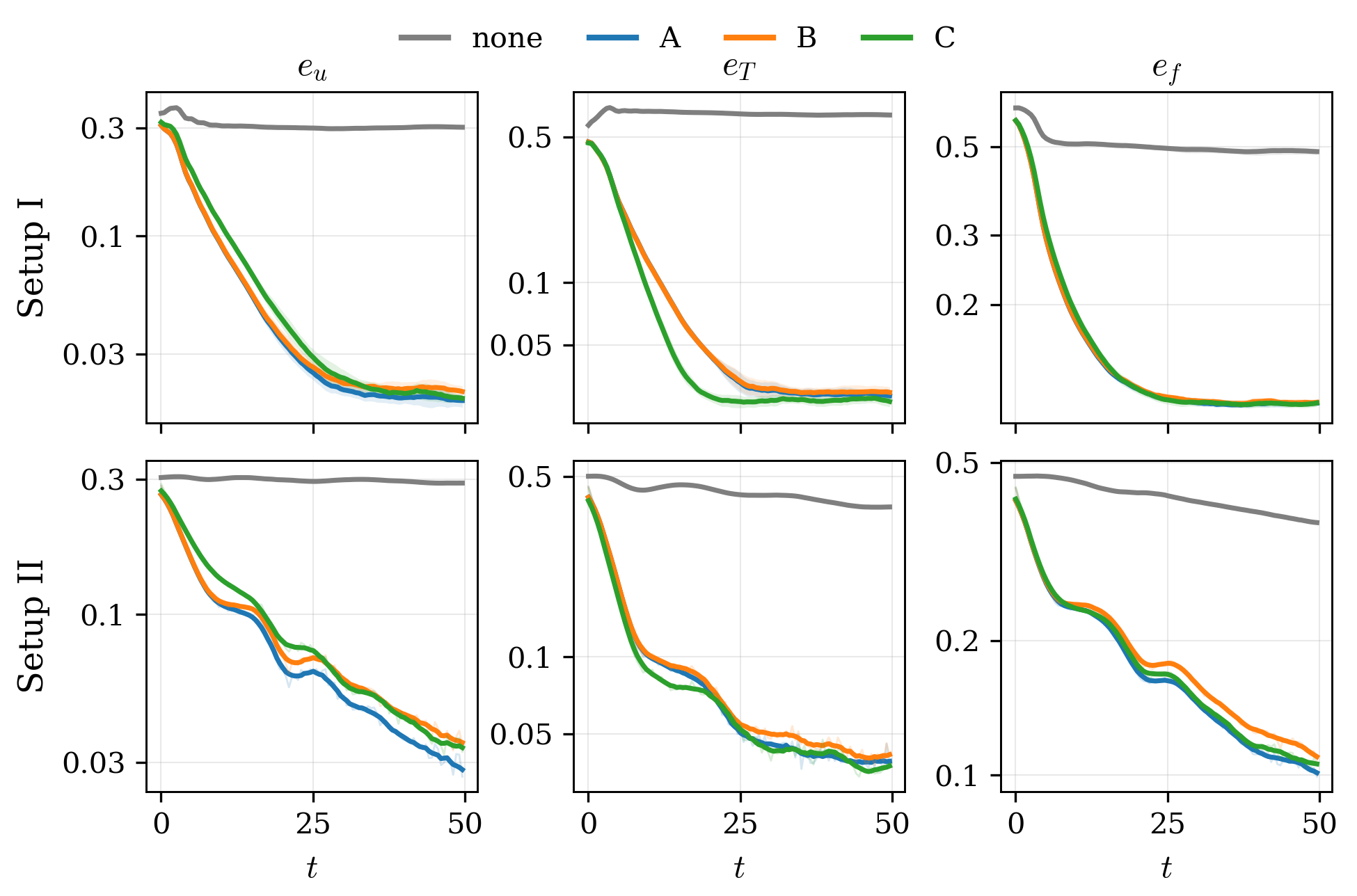}
    \caption{\textbf{Driven recovery error dynamics.}
    Synchronization errors \(e_u\), \(e_T\), and \(e_f\) versus time for the two
    driven BGK setups. The top row shows Setup~I (weak driver, moderate
    collisions), the bottom row Setup~II (stronger near-resonant driver, weak
    collisions). All three methods track one another closely, with
    Method~A attaining the lowest bulk-velocity and phase-space errors and
    Method~C the lowest temperature error.}
    \label{fig:e2-errors}
\end{figure}

\begin{figure}[t]
    \centering
    \includegraphics[width=\linewidth]{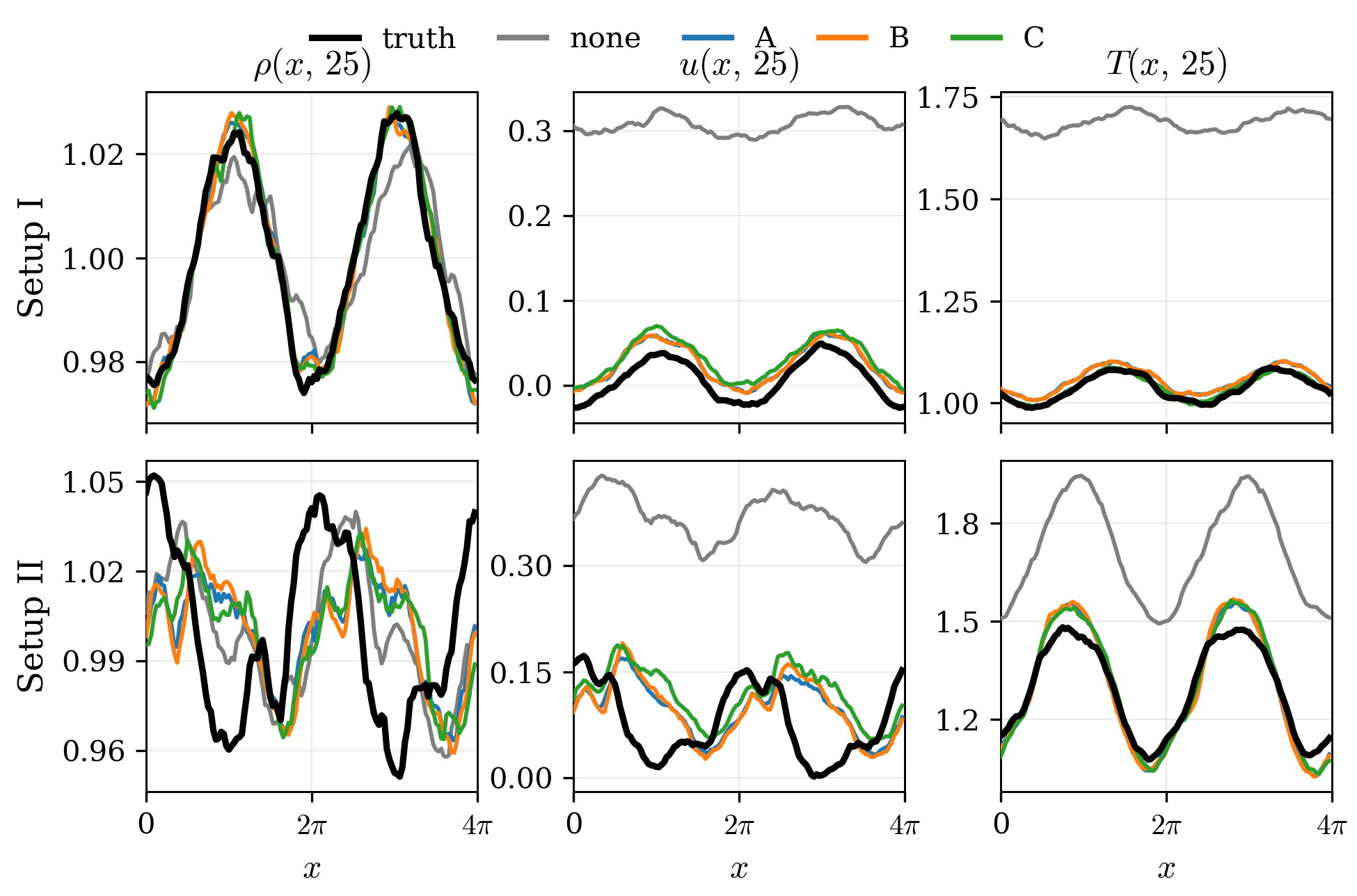}
    \caption{\textbf{Moment-field profiles at \(t=25\).}
    Density \(\rho(x)\), bulk velocity \(u(x)\), and temperature \(T(x)\) for
    the true solution, the unassimilated run, and the three nudged methods. The
    unassimilated solution remains close to the wrong initialization in \(u\) and \(T\),
    while the nudged methods synchronize these hydrodynamic moments with the
    true solution in both driven BGK regimes.}
    \label{fig:e2-profiles}
\end{figure}

\subsection{Undriven relaxation with conservative collisions}
\label{sec:num-conservative}

The second experiment removes the external driver and makes the collision step
exactly conservative, so collisions relax the velocity-space shape but does not touch the collision-invariant moments \eqref{eq:Q-conservation}.

On the torus of length \(L=4\pi\), the reference solution is initialized as the Maxwellian
current wave
\[
  f_0^\dagger(x,v)
  =
  \mathcal M_{1,u_0^\dagger(x),T_0}(v),
  \qquad
  u_0^\dagger(x)=0.25+0.8\sin(x/2),
  \qquad
  T_0=0.5,
\]
and the assimilating simulation starts from a current wave with the wrong bulk
velocity and a bimodal velocity shape,
\[
  g_0(x,v)
  =
  \tfrac12 \mathcal M_{1,u_0(x)-a,\theta}(v)
  +\tfrac12 \mathcal M_{1,u_0(x)+a,\theta}(v),
  \quad
  u_0(x)=-0.25-0.3\sin\bigl(x/2+\pi/3\bigr),
\]
with \(a=0.55\) and \(\theta=T_0-a^2\), so that its local temperature equals
the true temperature \(T_0\), while its bulk velocity and its velocity-space shape are
both wrong.  All runs use the conservative Dougherty-type particle map \eqref{eq:Q-D}-\eqref{eq:D-map} with \(\nu=0.5\), which conserves mass, momentum, and kinetic energy exactly in each collision cell.

Each run uses \(N_p=10^5\) particles, \(N_x=128\), \(\Delta t=0.002\), and
\(T_{\mathrm{fin}}=15\), with feedback at every step, matched scaling parameters
\(\gamma_1=\gamma_2=\gamma_3=3\) for Methods~A and~B and \(\gamma=3\) for
Method~C, and metric scale \(V_*=1\).
Table~\ref{tab:cc1d} reports the improvement ratios \(R_x^S(10,15)\) over
three independent particle initializations.  All three methods drive the
density, bulk-velocity, and temperature errors one to two orders of magnitude
below the unassimilated level and cut the phase-space error by more than half
(\(R_f\approx0.45\)).  Method~C has the smallest bulk-velocity and
temperature errors (\(R_u=0.002\), \(R_T=0.019\)), while Method~B
gives the smallest density error (\(R_\rho=0.008\)).

\begin{table}[t]
\centering\small
\begin{tabular}{@{}lcccc@{}}
\toprule
Method & \(R_\rho\) & \(R_u\) & \(R_T\) & \(R_f\) \\
\midrule
\textbf{A} & \(0.015\) & \(0.005\) & \(0.045\) & \(\mathbf{0.444}\) \\
B          & \(\mathbf{0.008}\) & \(0.016\) & \(0.061\) & \(0.447\) \\
C          & \(0.012\) & \(\mathbf{0.002}\) & \(\mathbf{0.019}\) & \(0.446\) \\
\bottomrule
\end{tabular}
\caption{Undriven relaxation with exactly conservative Dougherty-type
collisions (\(N_p=10^5\), \(N_x=128\), \(\Delta t=0.002\),
\(T_{\mathrm{fin}}=15\), \(\nu=0.5\)).  Each entry is \(R_*^S(10,15)\), the
mean error of method \(S\) over \(t\in[10,15]\) and over three particle
initializations, divided by the corresponding mean error of the unassimilated
run.}
\label{tab:cc1d}
\end{table}

Figure~\ref{fig:cc1d-errors} shows the error dynamics: the nudged runs
synchronize within a few time units and hold that level, while the
unassimilated errors stay at the level set by the wrong equilibrium.
In the final velocity marginal
\(\int f(x,v,T_{\mathrm{fin}})\diff x\) (Figure~\ref{fig:cc1d-velmarg}), the
three nudged methods lie on the true Maxwellian, with excess kurtosis
\(0.24\) against the true value of \(0.20\), whereas the unassimilated marginal is a
Maxwellian centered at the wrong bulk velocity.

\begin{figure}[t]
    \centering
    \includegraphics[width=\linewidth]{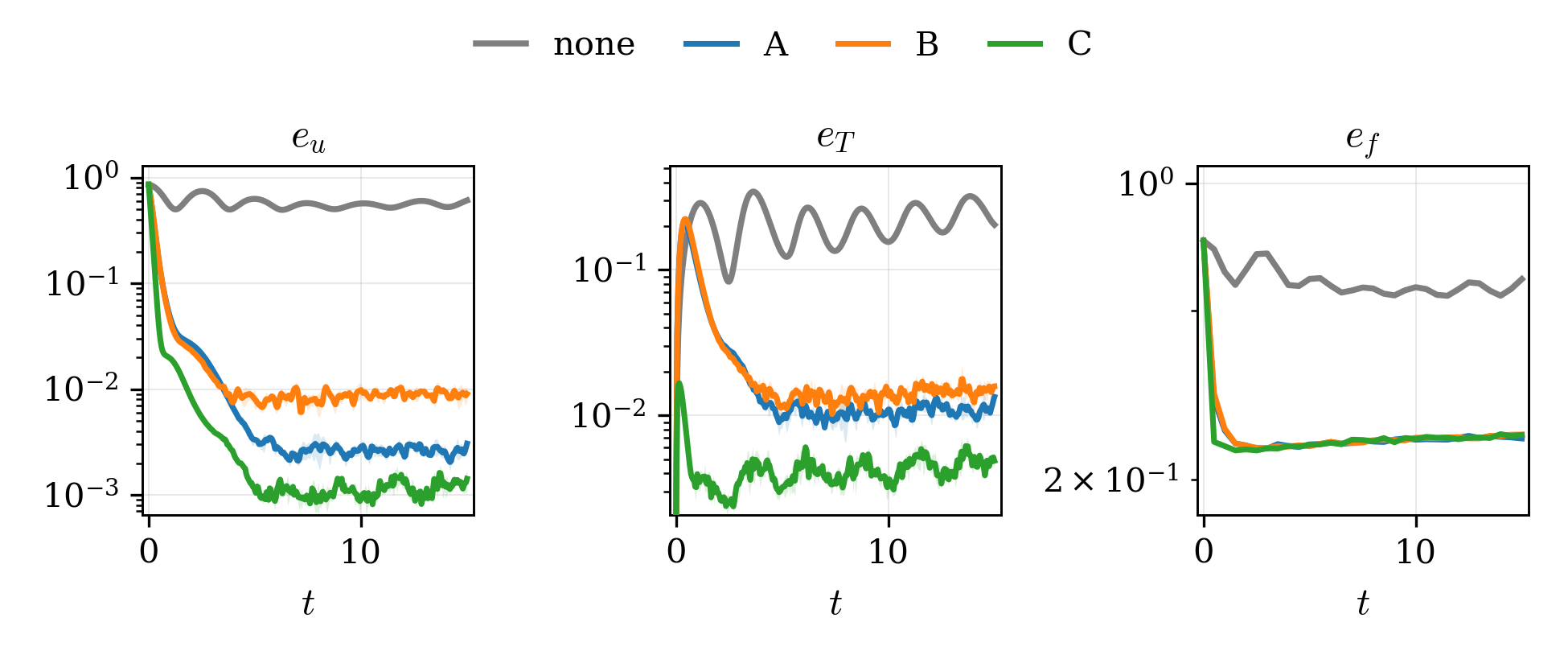}
    \caption{\textbf{Conservative-collision relaxation error dynamics.}
    Synchronization errors \(e_u\), \(e_T\), and \(e_f\) versus time for the
    unassimilated run and the three nudged methods (mean over three
    particle initializations, with a band of one standard deviation).
    The nudged methods drive \(e_u\) and \(e_T\) one
    to two orders of magnitude below the unassimilated level and cut \(e_f\)
    by more than half.}
    \label{fig:cc1d-errors}
\end{figure}

\begin{figure}[t]
    \centering
    \includegraphics[width=0.5\linewidth]{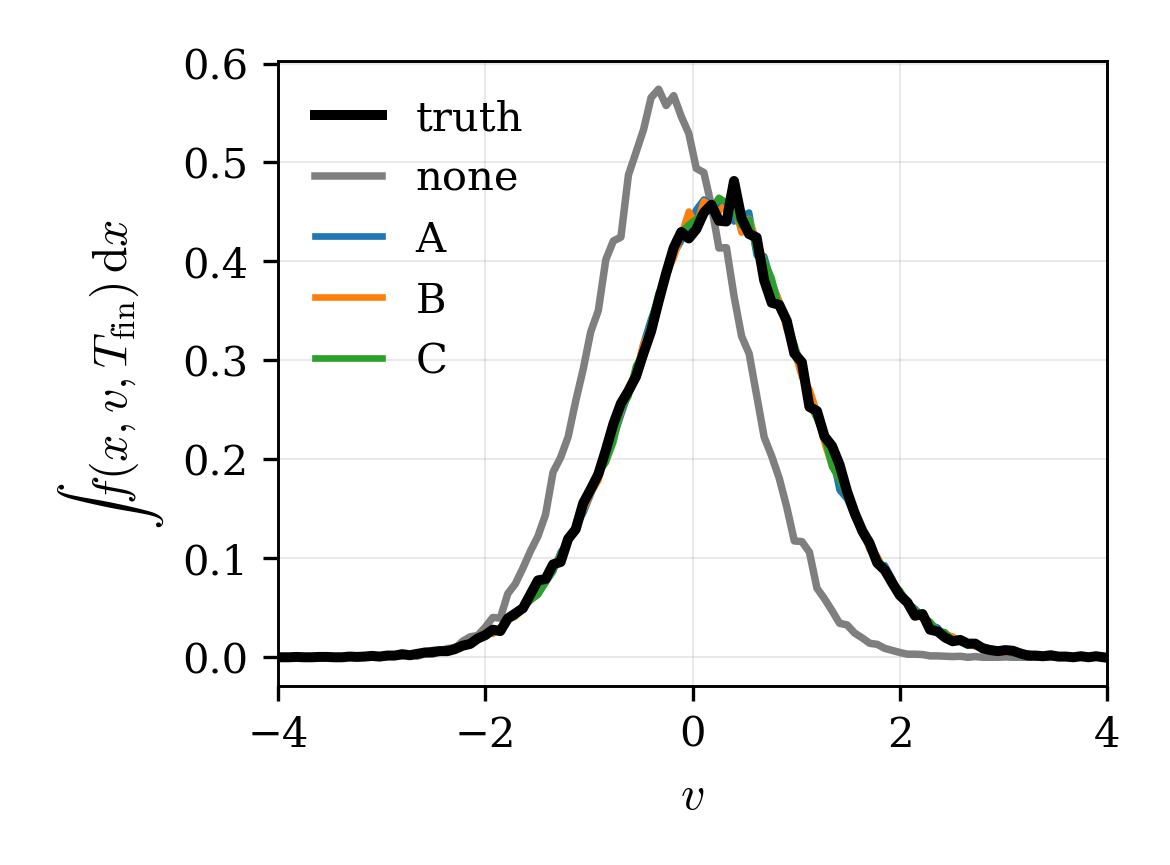}
    \caption{\textbf{Velocity marginal at the final time.}
    \(\int f(x,v,T_{\mathrm{fin}})\diff x\) for the true solution, the unassimilated
    run, and the three nudged methods.  The nudged methods recover the true marginal, while the unassimilated marginal is a Maxwellian with the wrong bulk velocity and temperature.}
    \label{fig:cc1d-velmarg}
\end{figure}

\subsection{Two-dimensional Taylor-Green relaxation}
\label{sec:num-2d}

The third experiment carries the undriven relaxation over to 2D2V, with the current wave replaced by a Taylor-Green vortex.
On \(\mathbb T^2=[0,2\pi)^2\), the reference solution is initialized
as
\[
  f_0^\dagger(\bx,\bv)
  =
  \mathcal M_{1,\bu_0^\dagger(\bx),T_0}(\bv),
  \quad
  \bu_0^\dagger(x,y)
  =
  \bigl(0.25+\sin x\cos y,\; -\cos x\sin y\bigr),
  \quad
  T_0=0.5,
\]
and the assimilating simulation starts from an initial distribution whose
bulk velocity is a weaker, phase-shifted, counter-rotating vortex with the
wrong mean flow,
\[
  \bu_0(x,y)
  =
  \bigl(-0.25-0.45\sin(x+\tfrac{\pi}{3})\cos(y+\tfrac{\pi}{4}),\;
        0.15+0.45\cos(x+\tfrac{\pi}{3})\sin(y+\tfrac{\pi}{4})\bigr),
\]
and whose peculiar velocity is an anisotropic centered Gaussian, with
variances \(0.8\) and \(0.2\) in the two velocity components but the true
temperature \(T_0\).  All runs use the conservative Dougherty-type particle map \eqref{eq:Q-D}-\eqref{eq:D-map} with \(\nu=0.5\), \(N_p=2\times10^5\) particles, a \(32^2\) spatial grid, \(\Delta t=0.002\), \(T_{\mathrm{fin}}=12\), and the same feedback cadence, scaling parameters, and metric scale \(V_*\) as in \S\ref{sec:num-conservative}.  Collisional
relaxation again fixes the long-time behavior, where the true solution relaxes toward
\(\mathcal M_{1,(0.25,0),0.75}\), while the unassimilated run relaxes toward
\(\mathcal M_{1,(-0.25,0.15),0.55}\). The vortex decays on the collisional
time scale, with the spatial standard deviation of \(u_x^\dagger\) falling from
\(0.50\) at \(t=0\) to \(0.28\) at \(t=1\) and \(0.014\) at
\(T_{\mathrm{fin}}\).

Table~\ref{tab:cc2d} reports \(R_*^S(8,12)\) over three particle
initializations, and Figure~\ref{fig:cc2d-errors} the corresponding error
dynamics.  The diagnostics are the moment errors: the four-dimensional
histogram behind \(e_f\) is not resolved at this particle count.  For the
one-dimensional behavior, all three methods synchronize the
bulk velocity and temperature, with \(R_u\le0.027\) and \(R_T\le0.044\), and
reduce the density error to \(R_\rho\) between \(0.266\) and \(0.393\).
Method~A gives the smallest bulk-velocity error, Method~C the smallest temperature error, and Method~B the smallest density error.

\begin{table}[t]
\centering\small
\begin{tabular}{@{}lccc@{}}
\toprule
Method & \(R_\rho\) & \(R_u\) & \(R_T\) \\
\midrule
\textbf{A} & \(0.293\) & \(\mathbf{0.019}\) & \(0.039\) \\
B          & \(\mathbf{0.266}\) & \(0.027\) & \(0.044\) \\
C          & \(0.393\) & \(0.020\) & \(\mathbf{0.032}\) \\
\bottomrule
\end{tabular}
\caption{2D2V Taylor-Green relaxation with exactly conservative
Dougherty-type collisions (\(N_p=2\times10^5\), \(32^2\) grid,
\(\Delta t=0.002\), \(T_{\mathrm{fin}}=12\), \(\nu=0.5\)).  Each entry is
\(R_*^S(8,12)\), the mean error of method \(S\) over \(t\in[8,12]\) and over
three initializations, divided by the corresponding mean error of
the unassimilated run.}
\label{tab:cc2d}
\end{table}

\begin{figure}[t]
\centering
\includegraphics[width=\linewidth]{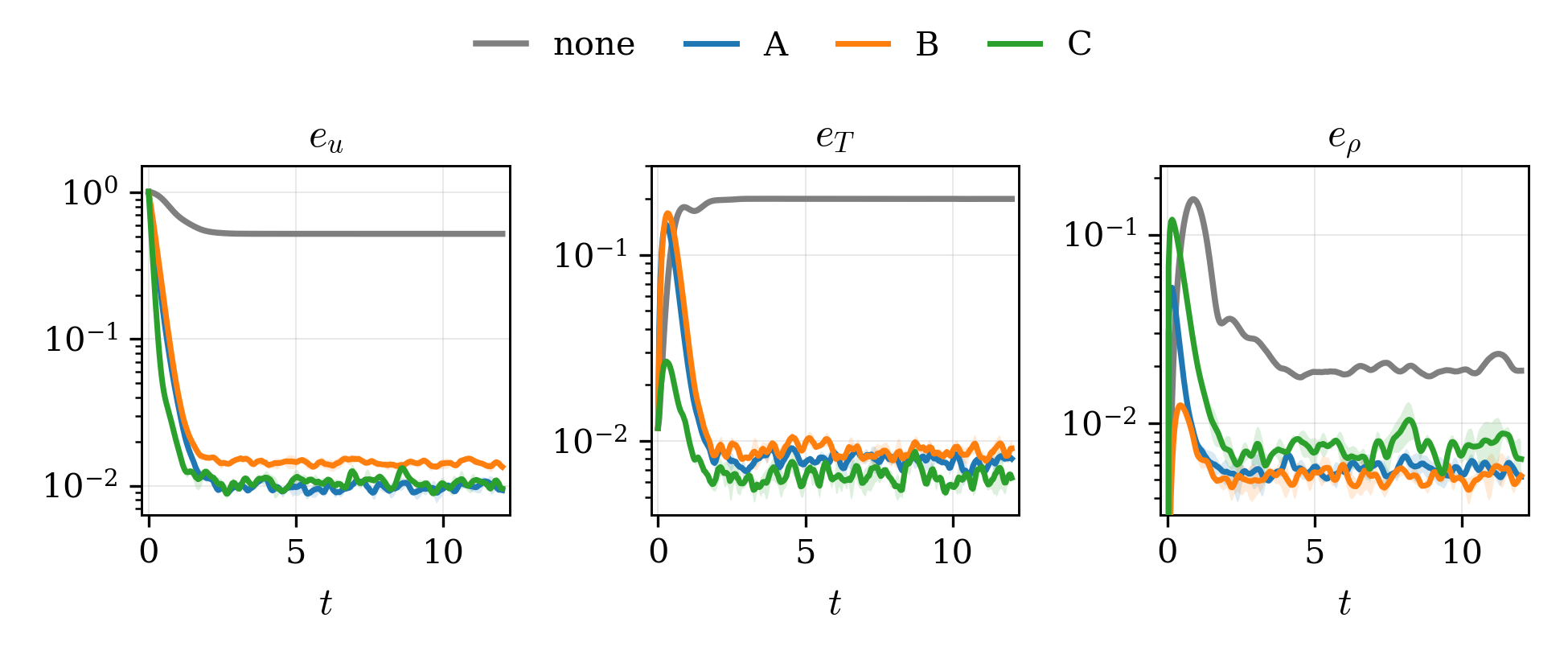}
\caption{\textbf{2D2V Taylor-Green relaxation error dynamics.}
Synchronization errors \(e_u\), \(e_T\), and \(e_\rho\) versus time for the
unassimilated run and the three nudged methods (mean over three
initializations, with a band of one standard deviation).
The nudged methods synchronize the bulk velocity and temperature and
hold the density error well below the unassimilated level.}
\label{fig:cc2d-errors}
\end{figure}

Figure~\ref{fig:cc2d-ufield} shows the bulk-velocity component \(u_x(x,y)\) at
\(t=1\), while the vortex is still strong.  The unassimilated run carries its
counter-rotating, mean-shifted vortex.  The three nudged methods
reproduce the true vortex cells.

\begin{figure}[t]
\centering
\includegraphics[width=\linewidth]{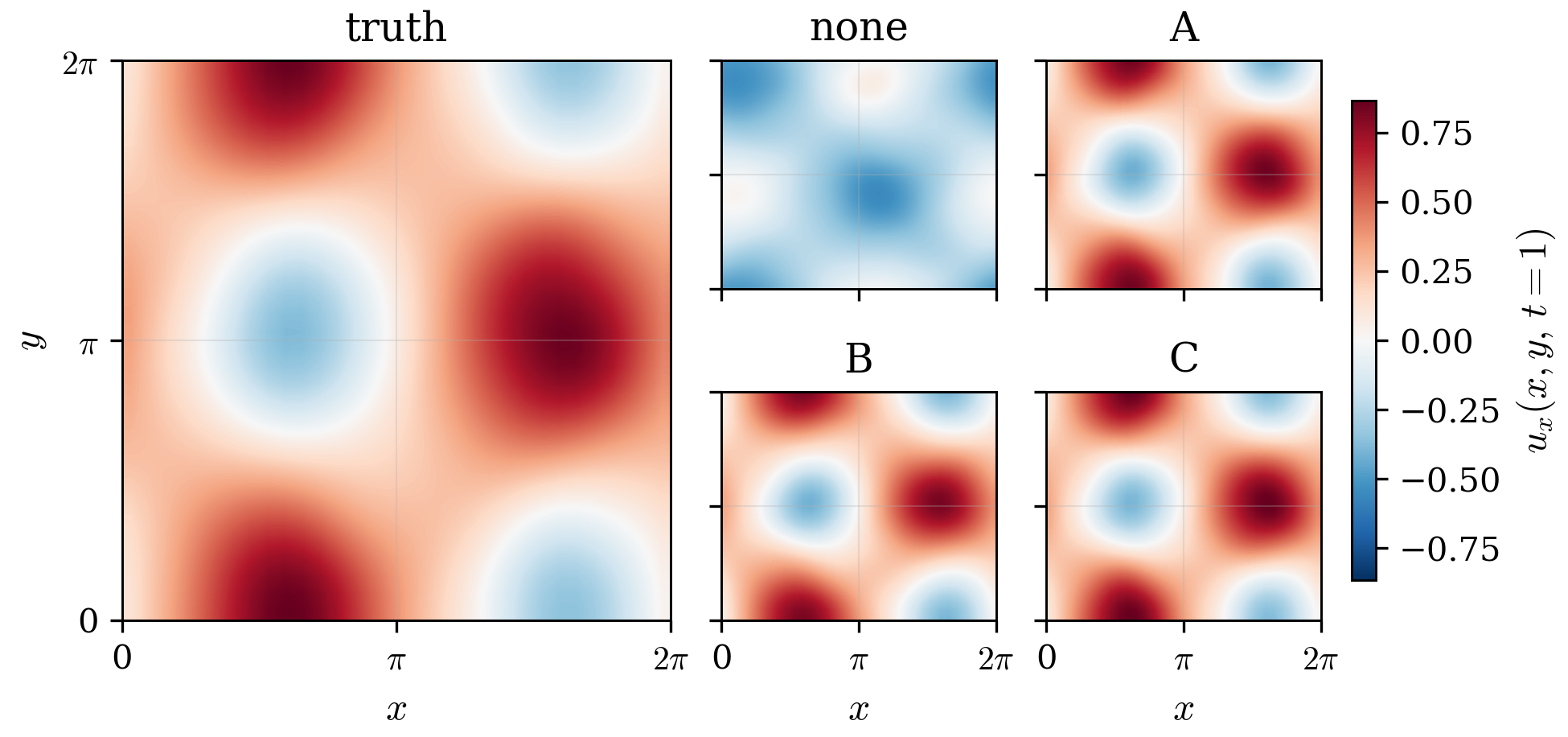}
\caption{\textbf{Bulk-velocity field at \(t=1\).}
The component \(u_x(x,y)\) of the 2D2V Taylor-Green run, shown before the
vortex decays under the collisions.  The unassimilated run carries the wrong
vortex and mean flow.  The nudged methods recover the true vortex
cells.}
\label{fig:cc2d-ufield}
\end{figure}

\subsection{Limitation: two-stream instability}
\label{sec:num-limit}

The final experiment quantifies the limitations of moment feedback.  On the torus of length \(L=4\pi\), the reference solution is initialized as the two-stream state
\[
  f_0^\dagger(x,v)
  =
  \frac{1}{\sqrt{2\pi}}\,v^2 e^{-v^2/2}\,\bigl(1+\alpha\cos(k_0x)\bigr),
  \qquad \alpha=0.05,\quad k_0=0.5,
\]
with moments \((\rho,u,T)=(1,0,3)\). Its dynamics are governed by its own
self-consistent field, with weak BGK collisions (with $\nu=0.05$) and no external driver. In this case,
the two-stream instability grows from the seeded perturbation and
saturates into a phase-space vortex (Figure~\ref{fig:obstruction-fxv},
top row).  All runs
use \(N_p=5\times10^5\), \(N_x=128\), \(\Delta t=0.05\),
\(T_{\mathrm{fin}}=30\), and unit scaling parameters,
with \(V_*=1\) for Methods~A and~C.

The two assimilating initializations carry the same wrong hydrodynamic moments,
$(u,T)=(0.3,2)$ against the true $(0,3)$, and differ only in the
velocity-space shape.  The first prior has the wrong shape, from
the uniform Maxwellian \(\mathcal M_{1,0.3,2}\) (with \(\nu=0.05\)),
Methods~A and~B synchronize the observed bulk velocity and
temperature, with \(R_u=0.21\) and
\(R_T=0.13\), and Method~C does so more
slowly, with \(R_u=0.51\) and \(R_T=0.29\).  The phase-space
error improves modestly for all three, with \(R_f=0.65\)-\(0.71\), but the final velocity
marginals are Maxwellians with the observed bulk velocity and
temperature, which is not the true bimodal profile
(Figure~\ref{fig:obstruction}, left).

Second, using the true two-stream shape but with wrong moments, given by the affine image
\(v\mapsto u_*+\sqrt{T_*/3}\,v\) with \((u_*,T_*)=(0.3,2)\) and
\(\nu=0.01\), the model provides the right form, and the feedback corrects the
moments.  Figure~\ref{fig:obstruction-fxv} shows \(f(x,v)\) at
\(t=0,15,20,30\). All three methods correct the moments while leaving
the two-stream structure intact, and their runs develop the two-stream
instability and track the true phase-space vortex, with
\(R_f=0.48\), \(0.48\), and \(0.72\) for A, B, and~C.
Figure~\ref{fig:obstruction} shows the final velocity marginal, where the nudged
marginals lie on the true bimodal profile, while the unassimilated run
has the wrong moments. 
This experiment delimits, rather than contradicts, the mechanism of \S\ref{sec:obstruction}. The mechanism there relies on either collisions strong enough to
damp the velocity-shape error outright, or states close enough to local
equilibrium that the shape is dissipatively slaved to the observed moments. In the
present regime, weak collisions and a true solution held far from equilibrium by
its self-consistent field, the shape is neither observed nor damped, so
moment feedback synchronizes the moments it observes while the shape must be
supplied by the prior. When the prior does supply the correct shape, the
corrected moments suffice for the assimilated run to develop the true
instability and track the phase-space vortex.
\begin{figure}[t]
\centering
\includegraphics[width=\linewidth]{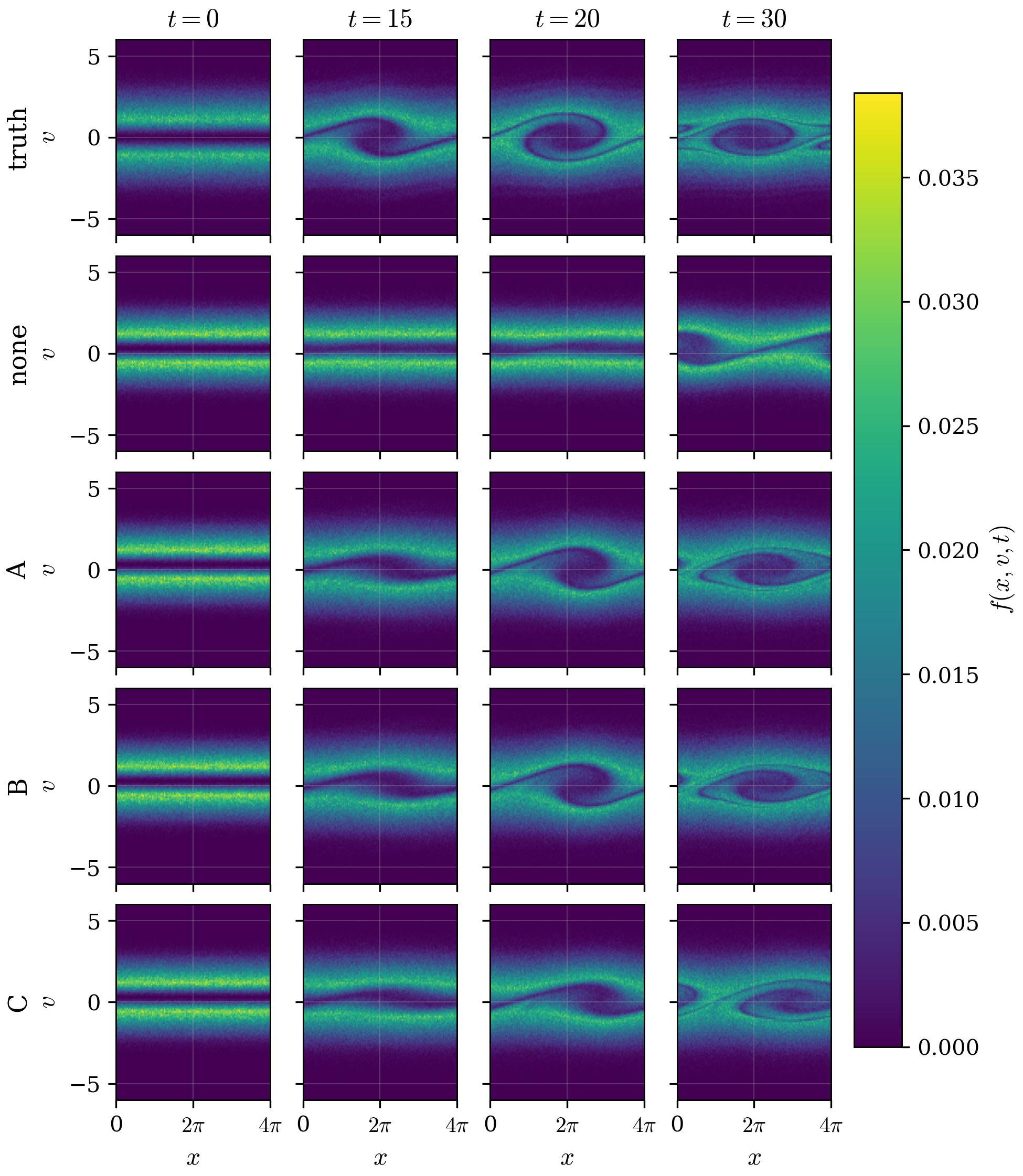}
\caption{\textbf{Two-stream phase space.}  \(f(x,v)\) at \(t=0,15,20,30\)
for the true solution and Methods~A, B, and~C.  Each assimilating run starts
from the true two-stream shape with wrong bulk velocity and temperature.
The true two-stream instability saturates into a phase-space vortex.  The
nudged runs correct the moments, develop the instability, and track the
true vortex.}
\label{fig:obstruction-fxv}
\end{figure}

\begin{figure}[t]
\centering
\includegraphics[width=0.8\linewidth]{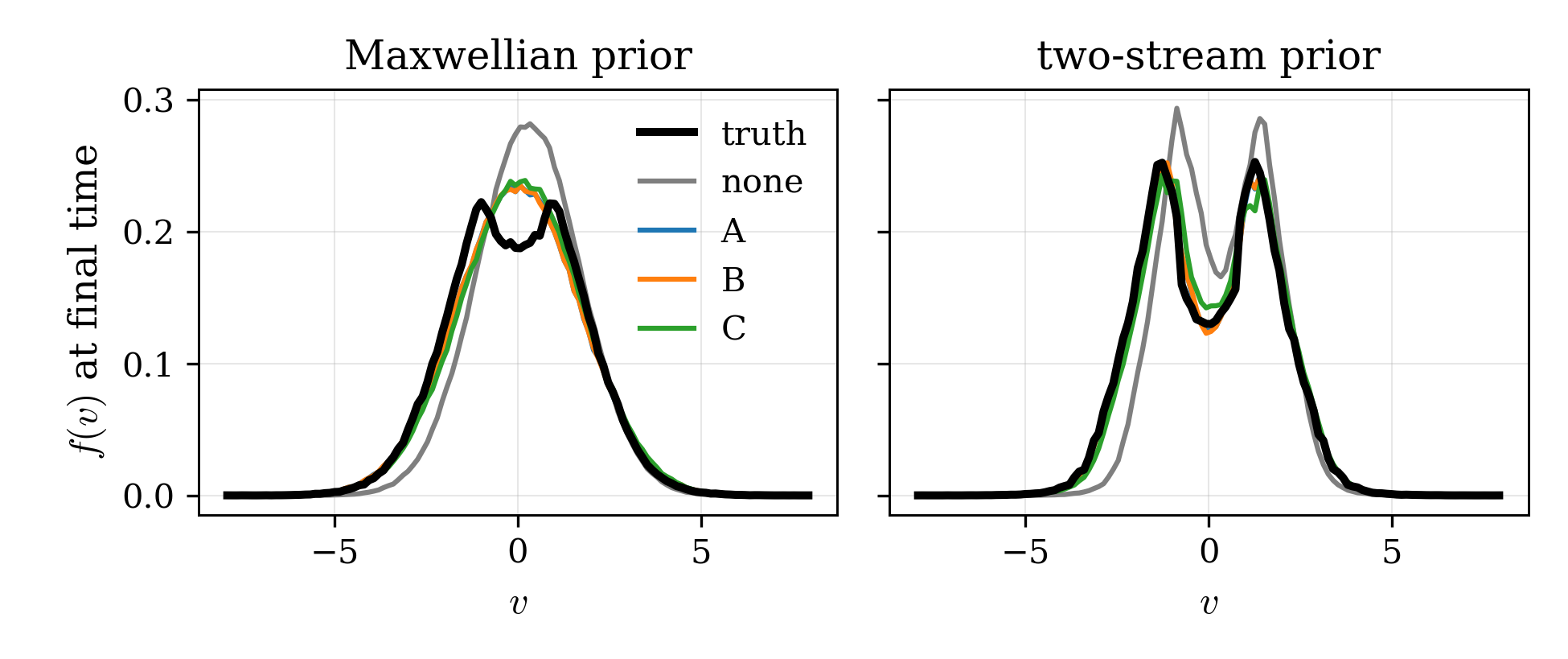}
\caption{\textbf{Final velocity marginals of the two initializations.}
\(\int f(x,v,T_{\mathrm{fin}})\diff x\) for the true solution, the unassimilated
run, and Methods~A, B, and~C.  Left: from the Maxwellian initialization, the
nudged marginals are Maxwellians carrying the observed bulk velocity and
temperature and miss the true bimodal profile.  Right: from the
two-stream initialization, the nudged marginals lie on the true bimodal profile,
whereas the unassimilated marginal keeps the initialization's wrong bulk velocity and
temperature.}
\label{fig:obstruction}
\end{figure}

\section{Conclusion}

We developed a continuous data assimilation framework for PIC simulations
of collisional Vlasov-Poisson systems from hydrodynamic-moment
observations. The main difficulty is that density, bulk velocity, and
temperature are not particle state variables but instead are are functionals of the
empirical phase-space measure. We addressed this by constructing the
feedback as a phase-space gradient flow of a moment mismatch,
so that the density, momentum, and energy channels are realized by
particle operations rather than by adding separate equations for the
moments. The feedback vanishes on the moment-compatible set, and the
two-Maxwellian obstruction shows that synchronization from arbitrary
neutral homogeneous Maxwellian initializations requires observations that
distinguish bulk velocity and temperature. For Method~C, the
mass-weighted global moments of the full inhomogeneous nudged flow obey
explicit balances, with a closed equation for the global bulk velocity.
The continuous-time finite-particle scheme coupled with the
Lenard-Bernstein collision surrogate is globally well posed. 
The numerical experiments show the balance between feedback and model, where in the driven BGK and conservative-collision relaxation tests, the nudged methods reduce the bulk-velocity and temperature errors by up to two orders of magnitude in 1D1V and 2D2V, and the two-stream test shows that the feedback corrects the moments, while the velocity-space shape must be supplied by the model.

There are several possible directions for future work. First, one can establish the synchronization of the nudged dynamics or the
mean-field limit of the finite-particle scheme. Both depend on the
specific method, since the three forms differ in mismatch
functional and metric and thus each require a separate analysis. Second, real problems involve magnetic fields and non-periodic
boundary conditions, whereas the field solve and the well-posedness
theory here rely on the electrostatic neutral torus. Third, practical diagnostics often deliver unstructured and sparse data, with scattered
measurement sites and intermittent observation times, rather than the
spatially resolved, fixed-cadence moments assumed here.

\section*{Acknowledgments}
This work used Anvil at Purdue through allocation MTH260007 from the Advanced Cyberinfrastructure Coordination Ecosystem: Services \& Support (ACCESS) program\cite{boerner2023access}, which is supported by U.S. National Science Foundation grants \#2138259, \#2138286, \#2138307, \#2137603, and \#2138296. The work was support by the NVIDIA Academic Grant Program.

\ifusesupplement\else
\appendix
\section{Proofs for the moment-nudging construction}
\label{app:nudging-proofs}
This appendix includes the proofs of the observability and structure results
of \S\ref{sec:nudging}.

\subsection{Proof of Proposition~\ref{prop:obs-necessary}}
\label{app:proof:obs-necessary}
\begin{proof}[Proof of Proposition~\ref{prop:obs-necessary}]
Let $\mathcal M=\mathcal M_{1,\bu,T}$ be any homogeneous Maxwellian. Its
spatial density is $\densr{\mathcal M}\equiv 1$, so
$\field{\mathcal M} = -\nabla_\bx(-\Delta_\bx)^{-1}\bigl(\densr{\mathcal M}-1\bigr)=0$. Spatial homogeneity gives
$\bv\cdot\nabla_\bx\mathcal M=0$, and $Q[\mathcal M]=0$ by \ref{A:coll}.
So $t\mapsto\mathcal M$ solves the unassimilated equation
\eqref{eq:VPC}--\eqref{eq:poisson}, and the uniqueness in \ref{A:wp} gives
$f^\dagger(t)=\mathcal M_{1,\bu_1,T_1}$ for all $t\ge0$.
We claim that $t\mapsto\mathcal M_{1,\bu_2,T_2}$ solves the assimilated
equation \eqref{eq:assim}. The transport, field, and collision terms vanish
along this candidate as before. For the feedback term, the identity $f^\dagger(t)=\mathcal M_{1,\bu_1,T_1}
$ and the degeneracy \eqref{eq:deg} give
\[
    \mathcal O[\mathcal M_{1,\bu_2,T_2}]-\mathcal O[f^\dagger(t)]
    =\mathcal O[\mathcal M_{1,\bu_2,T_2}]
     -\mathcal O[\mathcal M_{1,\bu_1,T_1}]
    =0,
    \qquad t\ge0,
\]
so $\mathcal N\bigl(\mathcal M_{1,\bu_2,T_2},0\bigr)=0$ by \ref{A:fb}. Thus
the candidate solves \eqref{eq:assim} with initial datum $f(0)$, and the
uniqueness in \ref{A:wp} gives $f(t)=\mathcal M_{1,\bu_2,T_2}$ for all
$t\ge0$. Substituting the two frozen trajectories into $d(\cdot,\cdot)$
completes the proof.
\end{proof}
\subsection{Proof of Corollary~\ref{cor:moment-obs}}
\label{app:proof:moment-obs}
\begin{proof}[Proof of Corollary~\ref{cor:moment-obs}]
If $(\bu,T)\mapsto \mathcal O[\mathcal M_{1,\bu,T}]$ is not injective, a
degenerate pair as in \ref{A:deg} exists, and Proposition~\ref{prop:obs-necessary} exhibits initial data for which
the assimilation error is a positive constant for all time. Since a homogeneous Maxwellian is uniquely
determined by its momentum and energy (equivalently, velocity and temperature)
moments, failure of the observation map to resolve these moments is equivalent to the map not being injective, which implies the map cannot synchronize from arbitrary Maxwellian initializations.
\end{proof}

\subsection{Proof of Lemma~\ref{lem:Sh-identification}}
\label{app:proof:Sh-identification}
\begin{proof}[Proof of Lemma~\ref{lem:Sh-identification}]
Set
\[
    q_0:=K_h*r_0[f],
    \qquad
    \mathbf q_1:=K_h*\br_1[f],
    \qquad
    q_2:=K_h*r_2[f].
\]
Recalling the definition of $\Phi_M$,
\[
    \Phi_M[f](\bx,\bv)
    =
    \gamma_1 q_0(\bx)
    +\gamma_2\bv\cdot\mathbf q_1(\bx)
    +\frac{\gamma_3}{2}|\bv|^2q_2(\bx),
\]
its gradients are thus given by
\[
    \nabla_{\bv}\Phi_M
    =
    \gamma_2\mathbf q_1+\gamma_3\bv q_2,
\]
and
\[
    \nabla_{\bx}\Phi_M
    =
    \gamma_1\nabla q_0
    +\gamma_2(\nabla\mathbf q_1)^{\top}\bv
    +\frac{\gamma_3}{2}|\bv|^2\nabla q_2.
\]
If $f\in\mathcal S_h$, then all three residuals, and hence
$q_0,\mathbf q_1,q_2$, vanish. Therefore,
$\nabla_{\bx,\bv}\Phi_M=0$. Conversely, suppose that
$\nabla_{\bx,\bv}\Phi_M=0$ almost everywhere with respect to measure
$f(\bx,\bv)\diff\bx\diff\bv$. Since
\[
    \nabla_{\bv}\Phi_M
    =
    \gamma_2\mathbf q_1(\bx)
    +\gamma_3\bv q_2(\bx),
\]
we have
\[
    0
    =
    \int_{\mathbb T^d}
    \left[
    \int_{\mathbb R^d}
    \left|
        \gamma_2\mathbf q_1(\bx)
        +\gamma_3\bv q_2(\bx)
    \right|^2
    f(\bx,\bv)\,\diff\bv
    \right]\diff\bx .
\]
The inner integral is nonnegative and  vanishes for almost
every $\bx\in\mathbb T^d$.  For each such $\bx$, the equality
\begin{equation}\label{eq:equality}
\gamma_2\mathbf q_1(\bx)+\gamma_3\bv q_2(\bx)=0
\end{equation}
holds for almost every $\bv$ with respect to the measure $f(\bx,\bv)\diff\bv$. Since
$\densr{f}(\bx)>0$ for almost every $\bx \in \mathbb T^d$ and $f(\bx,\cdot)$ is a density, which means $f(\bx,\cdot) \geq 0$, we have that $f(\bx,\cdot)\neq 0 $ for almost every $\bv \in \mathbb R^d$, otherwise $\densr{f}(\bx)=0$. So the set $S_\bx := \{\bv: f(\bx,\bv)>0\}$ has positive Lebesgue measure. If we define $G_\bx := \{\bv: \gamma_2\bq_1(\bx)+\gamma_3\bv\bq_2(\bx)=0\}$ and $B_\bx := \{\bv: \gamma_2\bq_1(\bx)+\gamma_3\bv\bq_2(\bx)\neq 0\}$, we have that $S_\bx = (S_\bx \cap G_\bx)\cup (S_\bx \cap B_\bx)$. The set $S_\bx \cap B_\bx$ has Lebesgue measure 0 since 
\[
0 = \int_{B_\bx} f(\bx,\bv) \diff \bv = \int_{B_\bx\cap S_\bx} f(\bx,\bv) \diff \bv + \int_{B_\bx\backslash S_\bx} f(\bx,\bv) \diff \bv = \int_{B_\bx\cap S_\bx} f(\bx,\bv) \diff \bv,
\]
where we used that $f(\bx,\bv) = 0$ on $B_\bx \backslash S_\bx$. Since $S_\bx$ has positive Lebesgue measure and $S_\bx\cap B_\bx$ has zero Lebesgue measure, $S_\bx \cap G_\bx$ must also have positive Lebesgue measure,  which means there are at least 2 distinct velocities $\bv_1$ and $\bv_2$ in $S_\bx \cap G_\bx$. So the equality \eqref{eq:equality} holds
at two distinct velocities $\bv_1$ and $\bv_2$. Subtracting the two
identities gives
\[
    \gamma_3(\bv_1-\bv_2)q_2(\bx)=0,
\]
and thus $q_2(\bx)=0$. Substitution then gives
$\mathbf q_1(\bx)=0$.

The position component now reduces to
$\nabla_{\bx}\Phi_M=\gamma_1\nabla q_0$. Thus
$\nabla q_0=0$ almost everywhere, so $q_0$ is constant on
$\mathbb T^d$. We have that
\[
    \int_{\mathbb T^d}q_0\,\diff\bx
    =
    \int_{\mathbb T^d}r_0\,\diff\bx
    =0,
\]
where the first equality uses $\int_{\mathbb T^d}K_h \diff \bx = 1$ so convolution with $K_h$ preserves the integral, and the second equality uses that both $f,f^\dagger \in \mathcal P_{2,M}$ have total mass $M = |\mathbb T^d|$. Therefore $q_0=0$.

By the consistency of the observations with the true solution,
\[
\begin{aligned}
    r_0
    =K_h*\bigl(\densr{f}-\densr{f^\dagger}\bigr),
    \br_1
    =K_h*\bigl(\momr{f}-\momr{f^\dagger}\bigr),
    r_2
    =K_h*\bigl(\kinr{f}-\kinr{f^\dagger}\bigr).
\end{aligned}
\]
Since $q_0=K_h*r_0=0$ and $K_h$ is even, which implies it is self-adjoint on $L^2(\mathbb T^d)$,
\[
\begin{aligned}
0
&=
\left\langle
\densr{f}-\densr{f^\dagger},q_0
\right\rangle\\
&=
\left\langle
K_h*\bigl(\densr{f}-\densr{f^\dagger}\bigr),r_0
\right\rangle
=
\|r_0\|_{L^2}^2.
\end{aligned}
\]
Thus $r_0=0$, and the same calculation componentwise for the momentum
residual, gives
\[
    \|\br_1\|_{L^2}^2=0,
    \qquad
    \|r_2\|_{L^2}^2=0.
\]
Therefore $f\in\mathcal S_h$.
\end{proof}

\subsection{Proof of Proposition~\ref{prop:pythagoras}}
\label{app:proof:pythagoras}
\begin{proof}[Proof of Proposition~\ref{prop:pythagoras}]
With $\KL(p\|q)=\int(p\log\tfrac pq-p+q)$ from~\eqref{eq:gen-kl}, subtracting gives
\[
\KL(f\|\pi^\obs)-\KL(f\|\mathcal M[f])
=\int \Bigl[f\,\log\tfrac{\mathcal M[f]}{\pi^\obs}+\pi^\obs-\mathcal M[f]\Bigr]\,\diff\bz .
\]
At each $\bx$,
\(
\log(\mathcal M[f]/\pi^\obs)
=\log\tfrac{\densr{f}}{\rho^\obs}-\tfrac d2\log\tfrac{\tempr{f}}{T^\obs}
-\tfrac{|\bv-\velr{f}|^2}{2\tempr{f}}+\tfrac{|\bv-\bu^\obs|^2}{2T^\obs},
\)
so \(\log(\mathcal M[f]/\pi^\obs)\) lies in the span of
$1,\bv,\tfrac12|\bv|^2$. We have that $f$ and $\mathcal M[f]$ have the same moments against these three statistics, i.e. $\rho[\mathcal M[f]] = \rho[f]$, $\bj[\mathcal M[f]] = \bj[f]$, and $\mathcal K [\mathcal M[f]]= \mathcal K[f]$, so $\int f\,\log(\mathcal M[f]/\pi^\obs)
=\int \mathcal M[f]\,\log(\mathcal M[f]/\pi^\obs)$. The right-hand side therefore
equals $\int[\mathcal M[f]\log\tfrac{\mathcal M[f]}{\pi^\obs}-\mathcal M[f]+\pi^\obs]
=\KL(\mathcal M[f]\|\pi^\obs)$, which is~\eqref{eq:pythagoras}.
\end{proof}

\subsection{Proof of Proposition~\ref{prop:C-global}}
\label{app:proof:C-global}
\begin{proof}[Proof of Proposition~\ref{prop:C-global}]
For any test function $\psi(\bv)$, the
weak form of \eqref{eq:C-VP} is
\begin{equation}\label{eq:C-global-weak}
    \frac{\diff}{\diff t}\int \psi\,g_t\diff\bx\diff\bv
    =\int \nabla_\bv\psi\cdot
        \bigl(\field{g_t}+\bb_{C,v}[g_t]\bigr)\,g_t\diff\bx\diff\bv
    +\int \psi\,Q[g_t]\diff\bx\diff\bv,
\end{equation}
where the transport and position-channel terms are $\bx$-divergences and drop
against $\nabla_\bx\psi=0$ on the periodic torus. Boundary terms also vanish on the periodic torus, with the finite second moment of $\bv$ handling the $\bv$-boundary. For
$\psi\in\{1,\,v_i,\,|\bv-\bc|^2\}$ with $\bc\in\mathbb R^d$ fixed, the
collision term \(\int \psi\,Q[g_t]\diff\bx\diff\bv\) also vanishes by \eqref{eq:Q-conservation}.

For the mass, taking $\psi=1$ gives $\dot M=0$. Solvability of
\eqref{eq:poisson} on $\mathbb T^d$ requires
$\int(\densr{g_t}-1)\diff\bx=\int -\Delta_\bx \phi(\bx,t)\diff\bx=0$, since the integral of the Laplacian of a periodic function is 0. Thus, $M=|\mathbb T^d|$.

For the momentum, the self-field exerts no net force on the neutral torus:
with $\field{g}=-\nabla_\bx\phi$ and $\densr{g}-1=-\Delta_\bx\phi$,
\begin{equation}\label{eq:C-no-net-force}
\begin{aligned}
    \int_{\mathbb T^d}\densr{g}\field{g}\diff\bx
    &= \int_{\mathbb T^d}\densr{g}\field{g}- \field{g} +\field{g}\diff\bx\\
    &=\int_{\mathbb T^d}(\densr{g}-1)\field{g}\diff\bx\\
    &=\int_{\mathbb T^d}\Delta_\bx\phi\nabla_\bx\phi\diff\bx\\
    &=-\tfrac12\int_{\mathbb T^d}\nabla_\bx|\nabla_\bx\phi|^2\diff\bx\\
    &=0,
\end{aligned}
\end{equation}
where the second equality uses
$\int_{\mathbb T^d}\field{g}\diff\bx=-\int_{\mathbb T^d}\nabla_\bx\phi
\diff\bx=0$, since the integral of the gradient of a periodic function is 0. Taking $\psi=v_i$ in \eqref{eq:C-global-weak}, the field term
vanishes by \eqref{eq:C-no-net-force}. At $\varepsilon=0$, we have $\Theta^\obs=T^\obs$ and $\regtemp{g}=\temp{g}$;
since $\bu^\obs$ and $T^\obs$ are spatially constant and $\int K_h=1$, the
convolution acts as the identity on the observation part of
\eqref{eq:Ch-drift}, and
\[
    \bb_{C,\bv}[g_t]
    =
    -\frac{\gamma}{T^\obs}(\bv-\bu^\obs)
    +\gamma\Bigl[\bv\,K_h*\bigl(\temp{g_t}^{-1}\bigr)
    -K_h*\bigl(\vel{g_t}/\temp{g_t}\bigr)\Bigr].
\]
For the first term, we have that
\begin{equation*}
\begin{aligned}
\int_{\mathbb T^d \times\mathbb R^d} -\frac{\gamma}{T^\obs}(\bv-\bu^\obs) g_t\diff \bx \diff \bv &= -\frac{\gamma}{T^\obs}\left(\int_{\mathbb T^d \times \mathbb R^d} \bv g_t \diff \bx \diff \bv -\bu^\obs \int_{\mathbb T^d \times \mathbb R^d} g_t \diff \bx \diff \bv\right)\\
&=-\frac{\gamma}{T^\obs}\left(\int_{\mathbb T^d} \bj[g_t]\diff \bx - \bu^\obs M\right)\\
&= -\frac{\gamma M}{T^\obs} \left(\overline \bu [g_t] -\bu^\obs\right),
\end{aligned}
\end{equation*}
where we used \eqref{eq:C-global-moments}.
For the second term, we have
\begin{equation*}
\begin{aligned}
&\int_{\mathbb T^d \times \mathbb R^d}\Bigl[\bv\,K_h*\bigl(\temp{g_t}^{-1}\bigr)
    -K_h*\bigl(\vel{g_t}/\temp{g_t}\bigr)\Bigr] g_t\diff\bx\diff\bv \\ = &\int_{\mathbb T^d} K_h*\bigl(\temp{g_t}^{-1}\bigr)\int_{\mathbb R^d} \bv g_t \diff \bv \diff \bx - \int_{\mathbb T^d}K_h*\bigl(\vel{g_t}/\temp{g_t}\bigr) \int_{\mathbb R^d} g_t \diff \bv \diff \bx\\
    =& \int_{\mathbb T^d} \bj[g_t] K_h*\bigl(\temp{g_t}^{-1}\bigr) \diff \bx - \int_{\mathbb T^d}\rho[g_t]K_h*\bigl(\vel{g_t}/\temp{g_t}\bigr) \diff \bx\\
    =& \int_{\mathbb T^d}  \temp{g_t}^{-1} \left(K_h*\bj[g_t]\right) \diff \bx - \int_{\mathbb T^d}\bigl(\vel{g_t}/\temp{g_t}\bigr) \left(K_h*\rho[g_t]\right)\diff \bx\\
    =&\int_{\mathbb T^d}
    \frac{\mom{g_t}-\dens{g_t}\,\vel{g_t}}{\temp{g_t}}\diff\bx\\
    =&0,
\end{aligned}
\end{equation*}
where we used that $K_h$ is even and hence self-adjoint on $L^2(\mathbb T^d)$,
$K_h*\densr{g_t}=\dens{g_t}$ and $K_h*\momr{g_t}=\mom{g_t}$
by \eqref{eq:smoothed-hydrodynamic-moments}, and $\dens{g_t}\vel{g_t}=\mom{g_t}$
by \eqref{eq:smoothed-derived-fields}. Hence,
for $\psi = v_i$, \eqref{eq:C-global-weak} becomes:
\begin{equation*}
\begin{aligned}
\frac{\diff }{\diff t}\int_{\mathbb T^d \times \mathbb R^d} v_i g_t \diff \bx \diff \bv = -\frac{\gamma M}{T^\obs} \left(\overline \bu [g_t] -\bu^\obs\right).   
\end{aligned}
\end{equation*}
Since $M$ is
constant, \eqref{eq:C-global-u} follows from \eqref{eq:C-global-moments}.

Lastly, for temperature, we take $\psi(\bv)=|\bv-\bc|^2$ and then set
$\bc=\overline{\bu}[g_t]$. The motion of the center does not contribute,
because
\[
    \frac{\diff}{\diff t}\bigl(dM\,\overline{T}[g_t]\bigr)
    =\Bigl[\frac{\diff}{\diff t}\int|\bv-\bc|^2 g_t\diff\bx\diff\bv
     \Bigr]_{\bc=\overline{\bu}[g_t]}
    -2\,\frac{\diff\overline{\bu}[g_t]}{\diff t}\cdot
     \int\bigl(\bv-\overline{\bu}[g_t]\bigr)g_t\diff\bx\diff\bv ,
\]
and the last integral equals $M\overline{\bu}[g_t]-M\overline{\bu}[g_t]=0$.
With $\nabla_\bv\psi=2(\bv-\overline{\bu}[g_t])$, the weak form
\eqref{eq:C-global-weak} produces three contributions. The field term is
the electric work,
\begin{align*}
    2\int\bigl(\bv-\overline{\bu}[g_t]\bigr)\cdot\field{g_t}\,
      g_t\diff\bx\diff\bv
    &=2\int_{\mathbb T^d}\field{g_t}\cdot\momr{g_t}\diff\bx
     -2\,\overline{\bu}[g_t]\cdot
      \int_{\mathbb T^d}\densr{g_t}\,\field{g_t}\diff\bx\\
    &=2\int_{\mathbb T^d}\field{g_t}\cdot\momr{g_t}\diff\bx
\end{align*}
by \eqref{eq:C-no-net-force}. The observation part of the feedback gives:
\begin{equation*}
\begin{aligned}
&-\frac{2\gamma}{T^\obs}
    \int_{\mathbb T^d\times \mathbb R^d}\bigl(\bv-\overline{\bu}[g_t]\bigr)\cdot(\bv-\bu^\obs)\,
      g_t\diff\bx\diff\bv \\
      =& -\frac{2\gamma}{T^\obs}
    \int_{\mathbb T^d\times \mathbb R^d}\bigl(\bv-\overline{\bu}[g_t]\bigr)\cdot(\bv-\overline{\bu}[g_t]+\overline{\bu}[g_t]-\bu^\obs)\,
      g_t\diff\bx\diff\bv\\
      =& -\frac{2\gamma}{T^\obs}\int_{\mathbb T^d \times \mathbb R^d} |\bv-\overline \bu[g_t]|^2 g_t \diff \bx \diff \bv + (\overline{\bu}[g_t]-\bu^\obs) \int_{\mathbb T^d \times \mathbb R^d} (\bv-\overline \bu[g_t])g_t \diff \bx \diff \bv\\
      =& -\frac{2\gamma}{T^\obs}\int_{\mathbb T^d \times \mathbb R^d} |\bv-\overline \bu[g_t]|^2 g_t \diff \bx \diff \bv\\
      =& -\frac{2\gamma}{T^\obs}dM\overline T[g_t],
\end{aligned}
\end{equation*}
where we used \eqref{eq:C-global-moments} and $\int(\bv-\overline{\bu}[g_t])\,g_t\diff\bx\diff\bv=0$.

For the model part, the velocity integrals give, pointwise in $\bx$,
\[
    \int_{\mathbb R^d}\bigl(\bv-\overline{\bu}[g_t]\bigr)\cdot\bv\,
    g_t\diff\bv
    =2\kinr{g_t}-\overline{\bu}[g_t]\cdot\momr{g_t},
    \qquad
    \int_{\mathbb R^d}\bigl(\bv-\overline{\bu}[g_t]\bigr)g_t\diff\bv
    =\momr{g_t}-\densr{g_t}\,\overline{\bu}[g_t],
\]
and the self-adjointness of $K_h*$ turns the raw moments into their
smoothed counterparts, so the model part contributes
\begin{equation*}
\begin{aligned}
&\int_{\mathbb T^d \times \mathbb R^d}2\gamma(\bv-\overline \bu[g_t])\cdot\Bigl[\bv\,K_h*\bigl(\temp{g_t}^{-1}\bigr)
    -K_h*\bigl(\vel{g_t}/\temp{g_t}\bigr)\Bigr] g_t\diff\bx\diff\bv\\
    =&2\gamma\int_{\mathbb T^d}
    \Bigl[
    \frac{2\kin{g_t}-\overline{\bu}[g_t]\cdot\mom{g_t}}{\temp{g_t}}
    -\frac{\bigl(\mom{g_t}-\dens{g_t}\,\overline{\bu}[g_t]\bigr)
      \cdot\vel{g_t}}{\temp{g_t}}
    \Bigr]\diff\bx\\
    =&2\gamma\int_{\mathbb T^d}
    \frac{2\kin{g_t}-\dens{g_t}|\vel{g_t}|^2}{\temp{g_t}}\diff\bx\\
    =& 2\gamma\,d\int_{\mathbb T^d}\dens{g_t}\diff\bx\\
    =&2\gamma\,dM,
\end{aligned}
\end{equation*}
where the cross terms cancel because $\dens{g_t}\vel{g_t}=\mom{g_t}$,
the third equality uses
$d\,\dens{g_t}\temp{g_t}=2\kin{g_t}-\dens{g_t}|\vel{g_t}|^2$
from \eqref{eq:smoothed-derived-fields}, and the last uses
$\int\dens{g_t}\diff\bx=\int\densr{g_t}\diff\bx=M$ by the unit mass of
$K_h$.
Summing the three contributions,
\[
    \frac{\diff}{\diff t}\bigl(dM\,\overline{T}[g_t]\bigr)
    =-\frac{2\gamma\,dM}{T^\obs}\bigl(\overline{T}[g_t]-T^\obs\bigr)
    +2\int_{\mathbb T^d}\field{g_t}\cdot\momr{g_t}\diff\bx ,
\]
which is \eqref{eq:C-global-T} after dividing by the constant $dM$.
\end{proof}

\section{Proofs for Section~\ref{sec:wp-finite-N}}
\label{app:sec:wp-finite-N}
Throughout, $L$ is the coefficient-bound constant of Lemma~\ref{lem:reg-bounds} on
$\mathcal P_{R_4}=\{\mu\in\mathcal P_4:\int|\bv|^4\diff\mu\le R_4\}$ for a
finite radius $R_4$. It depends on $R_4$, $\varepsilon$, and the fixed
data of Assumption~\ref{ass:wp-reg}, never on $t$. Write $\sigma:=\sqrt{2\nu\Theta_{\mathrm{LB}}}$,
recall the affine--isotropic velocity drift
$\bb^a_\bv(t,\bz,\mu)=\alpha^a[\mu](t,\bx)\bv+\beta^a[\mu](t,\bx)$ with the slope upper
bound $\alpha^a[\mu]\le\bar\alpha_a$ of \eqref{eq:gain-diss}. No relation
between $\bar\alpha_a$ and $\nu$ is assumed. The inhomogeneous coefficients are bounded by
$\|\beta^a[\mu]\|_\infty+\|\field{\mu}\|_\infty\le\bar c<\infty$, uniformly on
$\mathcal P_{R_4}$ and in $t$.
\subsection{proof of Lemma \ref{lem:field}}
\label{subapp:lem:field}
\begin{proof}
Since $\int_{\mathbb T^d}S_{\Delta x}*B_0\mu=1=\int_{\mathbb T^d}1$, the source
$S_{\Delta x}*B_0\mu-1$ has zero mean, and convolution associativity gives
$\field{\mu}=K_E*B_0\mu$ with $K_E=-\nabla_\bx(-\Delta_\bx)^{-1}S_{\Delta x}$, a fixed kernel
independent of $\mu$. To see $K_E\in C^1_b$, commute derivatives onto $S_{\Delta x}$: $K_E=-\nabla_\bx(-\Delta_\bx)^{-1}S_{\Delta x} = -(-\Delta_\bx)^{-1}\nabla_\bx S_{\Delta x}$, which is bounded because $\nabla_\bx S_{\Delta x} \in L^\infty$ since $S_{\Delta x}\in W^{2,\infty}_\bx$ and Green's function $(-\Delta_\bx)^{-1}$ on the torus is integrable. Also, $\nabla_\bx K_E=-\nabla_\bx^2(-\Delta_\bx)^{-1}S_{\Delta x}=-(-\Delta_\bx)^{-1}\nabla_\bx^2S_{\Delta x}$,
which is bounded because $S_{\Delta x}\in W^{2,\infty}_\bx$ implies $\nabla_\bx^2S_{\Delta x}\in L^\infty$
and the torus Green's function is integrable.
The field bound is then
$\|\field{\mu}\|_{W^{1,\infty}}=\|K_E*B_0\mu\|_{W^{1,\infty}}\le\|K_E\|_{W^{1,\infty}}$,
since $B_0\mu$ is a probability measure and convolution with it does not increase the
sup norm.

For the Lipschitz bound, fix $\bx\in\mathbb T^d$ and let $\pi$ be an optimal
$W_2$-coupling of $\mu$ and $\mu'$, with marginals $(\by,\bw)\sim\mu$ and
$(\by',\bw')\sim\mu'$. Taking $B_0\mu,B_0\mu'$ to be the position marginals of $\mu,\mu'$,
\[\begin{aligned}
    \field{\mu}(\bx)-\field{\mu'}(\bx)
  =&\int K_E(\bx-\by)\,B_0\mu(\diff\by)-\int K_E(\bx-\by')\,B_0\mu'(\diff\by')\\
  =&\int\bigl[K_E(\bx-\by)-K_E(\bx-\by')\bigr]\diff\pi .
\end{aligned}
\]
Since $K_E\in C^1_b$, the mean value theorem gives
$|K_E(\bx-\by)-K_E(\bx-\by')|\le\|\nabla K_E\|_\infty\,|\by-\by'|$, and Cauchy--Schwarz
on the probability measure $\pi$ yields
\[\begin{aligned}
  |\field{\mu}(\bx)-\field{\mu'}(\bx)|
  \le&\|\nabla K_E\|_\infty\!\int|\by-\by'|\diff\pi\\
  \le&\|\nabla K_E\|_\infty\Bigl(\!\int|\by-\by'|^2\diff\pi\Bigr)^{1/2}\\
  \le&\|\nabla K_E\|_\infty\,W_2(\mu,\mu') ,
  \end{aligned}
\]
the last inequality because
$\int|\by-\by'|^2\diff\pi\le\int\bigl(|\by-\by'|^2+|\bw-\bw'|^2\bigr)\diff\pi
=W_2(\mu,\mu')^2$, $\pi$ being $W_2$-optimal for $\mu,\mu'$. Taking the supremum over
$\bx$ gives
$\|\field{\mu}-\field{\mu'}\|_{L^\infty_\bx}\le\|\nabla K_E\|_\infty\,W_2(\mu,\mu')$.
\end{proof}

\subsection{proof of  Lemma \ref{lem:reg-bounds}}
\label{app:sub:reg-bounds}
\begin{proof}
We claim that for every
$\mu,\mu'\in\mathcal P_{R_4}$ and $i\in\{0,1,2\}$,
\begin{equation}
  \|K_h*B_i\mu\|_{W^{2,\infty}_\bx}\le L,
  \qquad
  \|K_h*B_i\mu-K_h*B_i\mu'\|_{W^{1,\infty}_\bx}\le L\,W_2(\mu,\mu'),
  \label{eq:pf-mom}
\end{equation}
where $B_0 \mu = \int_{\mathbb R^d} \mu \diff \bv, B_1\mu = \int_{\mathbb R^d} \bv \mu\diff \bv, B_2\mu=\int_{\mathbb R^d} \tfrac{1}{2}|\bv|^2\mu \diff\bv$. 
For the value bound, derivatives move to the kernel,
$\nabla_\bx^k(K_h*B_i\mu)=(\nabla_\bx^kK_h)*B_i\mu$, and by Jensen's inequality on the
probability measure $\mu$, $\left(\int|\bv|\,\diff\mu\right)^4\le \int|\bv|^4 \diff \mu \le R_4$ and $\left(\int|\bv|^2\,\diff\mu\right)^2\le \int|\bv|^4 \diff \mu \le R_4$ so
$\int|\bv|\,\diff\mu\le R_4^{1/4}$ and $\int|\bv|^2\,\diff\mu\le R_4^{1/2}$, so that
$\int|B_i|\,\diff\mu\le1+R_4^{1/4}+\frac12R_4^{1/2}$ and
$|\nabla_\bx^k(K_h*B_i\mu)(\bx)|\le\|\nabla_\bx^kK_h\|_\infty(1+R_4^{1/4}+\frac12R_4^{1/2})$ for $k\le2$,
giving the first bound with $L=\|K_h\|_{W^{2,\infty}}(1+R_4^{1/4}+\frac12R_4^{1/2})$.

For the Lipschitz bound, let $\pi$ be an optimal $W_2$-coupling of $\mu,\mu'$, with
marginals in $(\by,\bw)$ and $(\by',\bw')$. Take $B_i\mu,B_i\mu'$ to be the position marginals of $\mu,\mu'$ again. The binding case is $i=2$. Write, for
$k\in\{0,1\}$ and any $\bx$,
\[
\begin{aligned}
   &\nabla_\bx^k\kin{\mu}(\bx)-\nabla_\bx^k\kin{\mu'}(\bx)
   =\int\Bigl[\nabla_\bx^kK_h(\bx-\by)\frac12|\bw|^2
            -\nabla_\bx^kK_h(\bx-\by')\frac12|\bw'|^2\Bigr]\diff\pi\\
            =& \int\bigl[\nabla_\bx^kK_h(\bx-\by)-\nabla_\bx^kK_h(\bx-\by')\bigr]\frac12|\bw|^2
+\nabla_\bx^kK_h(\bx-\by')\frac12\bigl(|\bw|^2-|\bw'|^2\bigr)\diff\pi .
\end{aligned}
\]
By the mean value theorem and Cauchy--Schwarz, the first term is bounded by
\[
  \frac12\|\nabla_\bx^{k+1}K_h\|_\infty
  \Bigl(\!\int|\by-\by'|^2\diff\pi\Bigr)^{1/2}\Bigl(\!\int|\bw|^4\diff\pi\Bigr)^{1/2}
  \le\frac12\|\nabla_\bx^{k+1}K_h\|_\infty\,R_4^{1/2}\,W_2(\mu,\mu').
\] With $|\bw|^2-|\bw'|^2=(\bw-\bw')\!\cdot\!(\bw+\bw')$, Cauchy-Schwarz, and
$\bigl(\int|\bw+\bw'|^2\diff\pi\bigr)^{1/2}\le2R_4^{1/4}$, the second term is bounded by
\[
\begin{aligned}
\frac12 \|\nabla_\bx^{k}K_h\|_\infty\int |\bw - \bw'||\bw+\bw'|\diff \pi&\le \frac12 \|\nabla_\bx^{k}K_h\|_\infty\left(\int |\bw-\bw'|^2\diff\pi\right)^{1/2}\left(\int |\bw+\bw'|^2\diff\pi\right)^{1/2}\\
&\le \|\nabla_\bx^{k}K_h\|_\infty R_4^{1/4}W_2(\mu,\mu')
\end{aligned}
\]
Summing over $k=0,1$ gives
\eqref{eq:pf-mom} for $i=2$. For $i = 0$, we have for $k\in\{0,1\}$ and any $\bx$,
\[
\nabla^k_\bx \mathcal K_h[\mu] (\bx) - \nabla^k_\bx\mathcal K_h[\mu'](\bx) = \int \left[\nabla^k_\bx K_h(\bx-\by)-\nabla^k_\bx K_h(\bx-\by')\right]\diff \pi,
\]
which is bounded by
\[
\|\nabla^{k+1}_\bx K_h\|_\infty \int |\by-\by'|^2 \diff \pi \le \|\nabla^{k+1}_\bx K_h\|_\infty W_2(\mu,\mu'),
\]
where we used mean value theorem and Cauchy-Schwarz. Summing over $k=0,1$ gives \eqref{eq:pf-mom} for $i=0$. For $i=1$, $k\in\{0,1\}$, and any $\bx$,
\[
\begin{aligned}
&\nabla^k_\bx \mathcal K_h[\mu] (\bx) - \nabla^k_\bx\mathcal K_h[\mu'](\bx) = \int \left[\nabla^k_\bx K_h(\bx-\by)\bw - \nabla^k_\bx K_h(\bx-\by') \bw' \right]\diff \pi\\
=& \int \left[\nabla^k_\bx K_h(\bx - \by) - \nabla^k_\bx K_h(\bx-\by')\right]\bw + \nabla^k_\bx K_h(\bx-\by')(\bw-\bw')\diff \pi
\end{aligned}
\]
Using mean value theorem and Cauchy-Schwarz, this is bounded by
\[
\begin{aligned}
&\|\nabla^{k+1}_\bx K_h \|_\infty\left(\int |\by-\by'|^2\diff \pi\right)^{1/2}\left(\int |\bw|^2 \diff \pi\right)^{1/2} + \|\nabla^k_\bx K_h\|_\infty\left(\int|\bw - \bw'|^2 \diff \pi\right)^{1/2}\\\le& \left(\|\nabla^{k+1}_\bx K_h\|_\infty R_4^{1/4} + \|\nabla^k_\bx K_h\|_\infty\right)W_2(\mu,\mu')
\end{aligned}
\]
Summing over $k=0,1$ gives \eqref{eq:pf-mom} for $i=1$.

For the bulk velocity $\vel{\mu}$ and the temperature $\temp{\mu}$, we claim that
\begin{equation}
\begin{gathered}
  \|\vel{\mu}\|_{W^{2,\infty}_\bx}+\|\temp{\mu}\|_{W^{2,\infty}_\bx}\le \tilde L,\\
  \|\vel{\mu}-\vel{\mu'}\|_{W^{1,\infty}_\bx}+\|\temp{\mu}-\temp{\mu'}\|_{W^{1,\infty}_\bx}\le \tilde L\,W_2(\mu,\mu').
\end{gathered}
  \label{eq:pf-quot}
\end{equation}
For the bulk velocity, $\bu_h[\mu]=\bj_h[\mu]/\rho_h[\mu]$ is bounded by $L/\kappa_h$,
$\nabla_\bx\vel{\mu}=\nabla_\bx\mom{\mu}/\dens{\mu}-\mom{\mu}\otimes\nabla_\bx\dens{\mu}/\dens{\mu}^2$
is bounded by $L/\kappa_h+L^2/\kappa_h^2$, and the second $\bx$-derivative
\[
\nabla_\bx^2 \bu_h[\mu] = \frac{\nabla^2_\bx \bj_h[\mu]}{\rho_h[\mu]}-2\frac{\nabla_\bx\bj_h[\mu]\otimes\nabla_\bx\rho_h[\mu]}{\rho_h[\mu]^2} -\frac{\bj_h[\mu]\otimes\nabla^2_\bx \rho_h[\mu]}{\rho_h[\mu]^2}+2\frac{\bj_h[\mu]\otimes\nabla_\bx\rho_h[\mu]^{\otimes2}}{\rho_h[\mu]^3}
\]
is bounded by $L/\kappa_h + 3L^2/\kappa_h^2 + 2L^3/\kappa_h^3$, where we used \eqref{eq:Kh-floor} and
\eqref{eq:pf-mom}. \[
\vel{\mu}-\vel{\mu'}=\frac{\mom{\mu}-\mom{\mu'}}{\dens{\mu}}
-\frac{\mom{\mu'}(\dens{\mu}-\dens{\mu'})}{\dens{\mu}\dens{\mu'}}
\]
is bounded in $L^\infty_\bx$ by $L\,W_2(\mu,\mu')/\kappa_h+L^2W_2(\mu,\mu')/\kappa_h^2$ using \eqref{eq:Kh-floor} and
\eqref{eq:pf-mom} and 
\[
\begin{aligned}
&\nabla_\bx\left(\vel{\mu}-\vel{\mu'}\right) \\= &\frac{\nabla_\bx\bj_h[\mu]}{\rho_h[\mu]} - \frac{\nabla_\bx \bj_h[\mu']}{\rho_h[\mu']}-\frac{\bj_h[\mu]\otimes\nabla_\bx \rho_h[\mu]}{\rho_h[\mu]^2}+\frac{\bj_h[\mu']\otimes\nabla_\bx \rho_h[\mu']}{\rho_h[\mu']^2}\\
=& \frac{\nabla_\bx \bj_h[\mu]-\nabla_\bx \bj_h[\mu']}{\rho_h[\mu]}- \frac{\nabla_\bx \bj_h[\mu'](\rho_h[\mu]-\rho_h[\mu'])}{\rho_h[\mu]\rho_h[\mu']} - \\&\left(\frac{(\bj_h[\mu]-\bj_h[\mu'])\otimes\nabla_\bx\rho_h[\mu]}{\rho_h[\mu]^2} + \frac{\bj_h[\mu']\otimes(\nabla_\bx\rho_h[\mu]-\nabla_\bx\rho_h[\mu'])}{\rho_h[\mu]^2} - \right.\\&\left.\frac{\bj_h[\mu']\otimes\nabla_\bx \rho_h[\mu'](\rho_h[\mu]^2-\rho_h[\mu']^2)}{\rho_h[\mu^2]\rho_h[\mu']^2}\right)
\end{aligned}
\]
is bounded by $LW_2(\mu,\mu')/\kappa_h+ 3L^2W_2(\mu,\mu')/\kappa_h^2+2L^4W_2(\mu,\mu')/\kappa_h^4$, where we used \eqref{eq:Kh-floor} and
\eqref{eq:pf-mom} and $\rho_h[\mu]^2-\rho_h[\mu']^2=(\rho_h[\mu]-\rho_h[\mu'])(\rho_h[\mu]+\rho_h[\mu'])$. Hence, we have bounds for $\|\bu_h[\mu]\|_{W^{2,\infty}_\bx}$ and $\|\bu_h[\mu]-\bu_h[\mu']\|_{W^{1,\infty}_\bx}$ of the form in \eqref{eq:pf-quot}.

The temperature
is controlled identically once it is split as
$\temp{\mu}=\frac2d\,\kin{\mu}/\dens{\mu}-\frac1d|\vel{\mu}|^2$. The quotient
$\kin{\mu}/\dens{\mu}$ obeys the same estimate as $\vel{\mu}$, now using
$\|\kin{\mu}\|_{W^{2,\infty}_\bx}\le L$ from \eqref{eq:pf-mom} with $i=2$ together
with $\dens{\mu}\ge\kappa_h$, while the algebra property of $W^{2,\infty}_\bx$ gives
\[\||\vel{\mu}|^2\|_{W^{2,\infty}_\bx}\le C_1\|\vel\mu\|_{W^{2,\infty}_\bx}^2=C_1\left(\frac{3L}{\kappa_h}+\frac{4L^2}{\kappa_h^2}+\frac{2L^3}{\kappa_h^3}\right)^2,\] 
where we summed the $L^\infty_\bx$ bounds for $\bu_h[\mu]$, $\nabla_\bx\bu_h[\mu]$, and $\nabla^2_\bx\bu_h[\mu]$ derived earlier. Via
$|\vel{\mu}|^2-|\vel{\mu'}|^2=(\vel{\mu}-\vel{\mu'})\!\cdot\!(\vel{\mu}+\vel{\mu'})$ and the algebra property of $W^{1,\infty}_\bx$,
\[
\begin{aligned}
\||\vel{\mu}|^2-|\vel{\mu'}|^2\|_{W^{1,\infty}_\bx}&\le C_2 \|\vel{\mu}-\vel{\mu'}\|_{W^{1,\infty}_\bx}\|\vel{\mu}+\vel{\mu'}\|_{W^{1,\infty}_\bx}\\
&\le C_2 \left(\frac{2L}{\kappa_h}+\frac{4L^2}{\kappa_h^2}+\frac{2L^4}{\kappa_h^4}\right)\left(\frac{4L}{\kappa_h}+\frac{2L^2}{\kappa_h^2}\right)W_2(\mu,\mu')
\end{aligned}
\]
which is obtained from summing the $L^\infty_\bx$ bounds involving $\bu_h$. Summing the $\mathcal K_h[\mu]/\rho_h[\mu]$ and $|\bu_h[\mu]|^2$ terms yields bounds for $\temp{\mu}$ in \eqref{eq:pf-quot}.

For Method~C, the regularized temperature
$\regtemp{\mu}=\temp{\mu}+\varepsilon$ of \S\ref{sec:formC} satisfies
$\regtemp{\mu}\ge\varepsilon$ with no temperature lower bound assumed, so the
$\|\temp{\mu}\|_{W^{2,\infty}_\bx}\le \tilde L$ and $\|\temp{\mu}-\temp{\mu'}\|_{W^{1,\infty}_\bx}\le \tilde L W_2(\mu,\mu')$ bounds just proved and the smooth map $t\mapsto t^{-1}$ on $[\varepsilon,\infty)$
(derivative bounded by $\varepsilon^{-2}$) give
\begin{equation}
\begin{aligned}
  \|(\regtemp{\mu})^{-1}\|_{W^{2,\infty}_\bx}\le&\|(\regtemp{\mu})^{-1}\|_{L^\infty_\bx}+\|\nabla_\bx \Theta[\mu](\regtemp{\mu})^{-2}]\|_{L^\infty_\bx}+\|\nabla_\bx^2\Theta[\mu](\regtemp{\mu})^{-2}\|_{L^\infty_\bx}+\\&\|2(\nabla_\bx\Theta[\mu])^2(\regtemp{\mu})^{-3}\|_{L^\infty_\bx}\\\le& \frac{1}{\varepsilon} + \frac{2\tilde L}{\varepsilon^2}+\frac{2\tilde L^2}{\varepsilon^3}.
\end{aligned}
  \label{eq:pf-Tinv-1}
\end{equation}
Additionally, we have that $(\Theta[\mu])^{-1}-(\Theta[\mu'])^{-1}$ is bounded in $L^\infty_\bx$ by $\tilde{L}W_2(\mu,\mu')/\varepsilon^2$ using mean value theorem, and 
\begin{equation*}
\begin{aligned}
\nabla_\bx((\Theta[\mu])^{-1}-(\Theta[\mu'])^{-1}) &= -\frac{\nabla_\bx\Theta[\mu]}{(\Theta[\mu])^2}+\frac{\nabla_\bx\Theta[\mu']}{(\Theta[\mu'])^2}\\
&=-\frac{\nabla_\bx \Theta[\mu]-\nabla_\bx\Theta[\mu']}{(\Theta[\mu])^2} -\nabla_\bx\Theta[\mu']\left(\frac{1}{(\Theta[\mu])^2}-\frac{1}{(\Theta[\mu'])^2}\right),
\end{aligned}
\end{equation*}
which is bounded by $\tilde L W_2(\mu,\mu')/\varepsilon^2+2\tilde L^2 (\tilde L+\varepsilon)W_2(\mu,\mu')/\varepsilon^4$, where we used $\tfrac{1}{(\Theta[\mu])^2}-\tfrac{1}{(\Theta[\mu'])^2} = \tfrac{(\Theta[\mu'])^2-(\Theta[\mu])^2}{(\Theta[\mu])^2(\Theta[\mu'])^2}=\tfrac{(\Theta[\mu']-\Theta[\mu])(\Theta[\mu']+\Theta[\mu])}{(\Theta[\mu])^2(\Theta[\mu'])^2}=\tfrac{(T_h[\mu']-T_h[\mu])(T_h[\mu']+T_h[\mu]+2\varepsilon)}{(\Theta[\mu])^2(\Theta[\mu'])^2}$. Hence, we have
\begin{equation}
\|(\regtemp{\mu})^{-1}-(\regtemp{\mu'})^{-1}\|_{W^{1,\infty}_\bx}\le \frac{2\tilde LW_2(\mu,\mu')}{\varepsilon^2}+\frac{2\tilde L^2(\tilde L+\varepsilon)W_2(\mu,\mu')}{\varepsilon^4}.
\label{eq:pf-Tinv-2}
\end{equation}

For Methods~A and~B, the velocity drift read off from \eqref{eq:split-v} is
$\bb^a_\bv=\alpha^a[\mu]\bv+\beta^a[\mu]$ with
\[
  \alpha^A=\alpha^B=-\gamma_3\,K_h*(\kin{\mu}-\mathcal K^\obs),\quad
  \beta^A=\beta^B=-\gamma_2\,K_h*(\mom{\mu}-\bj^\obs).
\]
From here on, we use $L$ as a generic constant depending on $R_4, \varepsilon$, and the fixed data in Assumption \ref{ass:wp-reg}, absorbing the explicit bounds derived above. Since the observed fields $\rho^\obs, \bu^\obs,T^\obs$ lie in $W^{2,\infty}_\bx$ uniformly in time, \eqref{eq:smoothed-derived-fields} implies that $\mathcal K^\obs,\bj^\obs\in W^{2,\infty}_\bx$ uniformly in time as well. Using this and \eqref{eq:pf-mom}, we have
\[
\begin{aligned}
  \|\alpha^A\|_{W^{1,\infty}_\bx}
  &=\gamma_3\|K_h*(\kin{\mu}-\mathcal K^\obs)\|_{W^{1,\infty}_\bx}
  \le\gamma_3\|\kin{\mu}-\mathcal K^\obs\|_{W^{1,\infty}_\bx}\\
  &\le\gamma_3\bigl(\|\kin{\mu}\|_{W^{2,\infty}_\bx}+\|\mathcal K^\obs\|_{W^{2,\infty}_\bx}\bigr)\le L,\\
  \|\beta^A\|_{W^{1,\infty}_\bx}
  &=\gamma_2\|K_h*(\mom{\mu}-\bj^\obs)\|_{W^{1,\infty}_\bx}
  \le\gamma_2\|\mom{\mu}-\bj^\obs\|_{W^{1,\infty}_\bx}\\
  &\le\gamma_2\bigl(\|\mom{\mu}\|_{W^{2,\infty}_\bx}+\|\bj^\obs\|_{W^{2,\infty}_\bx}\bigr)\le L,
\end{aligned}
\]
while the observed term cancels in each difference, leaving
\[
\begin{aligned}
  \|\alpha^A[\mu]-\alpha^A[\mu']\|_{L^\infty_\bx}
  &=\gamma_3\|K_h*(\kin{\mu}-\kin{\mu'})\|_{L^\infty_\bx}
  \le\gamma_3\|\kin{\mu}-\kin{\mu'}\|_{L^\infty_\bx}\\
  &\le\gamma_3\|\kin{\mu}-\kin{\mu'}\|_{W^{1,\infty}_\bx}\le L\,W_2(\mu,\mu'),\\
  \|\beta^A[\mu]-\beta^A[\mu']\|_{L^\infty_\bx}
  &=\gamma_2\|K_h*(\mom{\mu}-\mom{\mu'})\|_{L^\infty_\bx}
  \le\gamma_2\|\mom{\mu}-\mom{\mu'}\|_{L^\infty_\bx}\\
  &\le\gamma_2\|\mom{\mu}-\mom{\mu'}\|_{W^{1,\infty}_\bx}\le L\,W_2(\mu,\mu').
\end{aligned}
\] For Method~C, \eqref{eq:Ch-drift} with the $\varepsilon$-shift,
\[
  \alpha^C=-\gamma K_h*\Bigl(\frac{1}{\Theta^{\obs}}-\frac{1}{\regtemp{\mu}}\Bigr),
  \quad
  \beta^C=\gamma K_h*\Bigl(\frac{\bu^\obs}{\Theta^{\obs}}-\frac{\vel{\mu}}{\regtemp{\mu}}\Bigr).
\]
Since $\Theta^{\obs}\ge\varepsilon$ and $\bu^\obs,T^\obs\in W^{2,\infty}_\bx$,
the observed reciprocal $(\Theta^{\obs})^{-1}$ and product
$\bu^\obs(\Theta^{\obs})^{-1}$ are bounded in $W^{2,\infty}_\bx$ by $L$, and we use the algebra property of $W^{1,\infty}_\bx$ to get:
\[
\begin{aligned}
  \|\alpha^C\|_{W^{1,\infty}_\bx}
  &\le\gamma\bigl\|(\regtemp{\mu})^{-1}-(\Theta^{\obs})^{-1}\bigr\|_{W^{1,\infty}_\bx}\\
  &\le\gamma\bigl(\|(\regtemp{\mu})^{-1}\|_{W^{2,\infty}_\bx}+\|(\Theta^{\obs})^{-1}\|_{W^{2,\infty}_\bx}\bigr)\le L,\\
  \|\beta^C\|_{W^{1,\infty}_\bx}
  &\le\gamma\bigl\|\vel{\mu}(\regtemp{\mu})^{-1}-\bu^\obs(\Theta^{\obs})^{-1}\bigr\|_{W^{1,\infty}_\bx}\\
  &\le\gamma C\bigl(\|\vel{\mu}\|_{W^{2,\infty}_\bx}\|(\regtemp{\mu})^{-1}\|_{W^{2,\infty}_\bx}
       +\|\bu^\obs\|_{W^{2,\infty}_\bx}\|(\Theta^{\obs})^{-1}\|_{W^{2,\infty}_\bx}\bigr)\le L,
\end{aligned}
\]
while the observed term cancels in each difference, so
\[
\begin{aligned}
&  \|\alpha^C[\mu]-\alpha^C[\mu']\|_{L^\infty_\bx}
  \le \gamma\bigl\|(\regtemp{\mu})^{-1}-(\regtemp{\mu'})^{-1}\bigr\|_{L^\infty_\bx}\\
  \le&\gamma\bigl\|(\regtemp{\mu})^{-1}-(\regtemp{\mu'})^{-1}\bigr\|_{W^{1,\infty}_\bx}
  \le L\,W_2(\mu,\mu'),\\
  &\|\beta^C[\mu]-\beta^C[\mu']\|_{L^\infty_\bx}\le \gamma\bigl\|\vel{\mu}(\regtemp{\mu})^{-1}-\vel{\mu'}(\regtemp{\mu'})^{-1}\bigr\|_{L^\infty_\bx}\\
  \le&\gamma\|\vel{\mu}-\vel{\mu'}\|_{L^\infty_\bx}\,\|(\regtemp{\mu})^{-1}\|_{L^\infty_\bx}
     +\gamma\|\vel{\mu'}\|_{L^\infty_\bx}\,\bigl\|(\regtemp{\mu})^{-1}-(\regtemp{\mu'})^{-1}\bigr\|_{L^\infty_\bx}\\
  \le& L\,W_2(\mu,\mu').
\end{aligned}
\]
This proves \eqref{eq:wp-v}--\eqref{eq:wp-vmu} for all $a\in\{A,B,C\}$.

For Methods~A and~C, the position drift read off
from \eqref{eq:weighted-descent-velocity} and \eqref{eq:Ch-drift} has the form
\[
  \bb^a_\bx(t,\bz,\mu)=-\frac{P^a[\mu](\bx,\bv)}{w(\bx,\bv)},
  \qquad
  w(\bx,\bv)=1+\frac{|\bv-\bu^\obs(\bx)|^2}{V_*^2}\ge1,
\]
with $P^a[\mu]=P^a_0[\mu](\bx)+P^a_1[\mu](\bx)^\top\bv+P^a_2[\mu](\bx)|\bv|^2$ a polynomial
in $\bv$ of degree $\le2$.
For Method~A,
\[
\begin{aligned}
  P^A_0[\mu]&=\gamma_1\,\nabla_\bx\bigl[K_h*\bigl(K_h*B_0\mu-\rho^\obs\bigr)\bigr],\\
  P^A_1[\mu]&=\gamma_2\,\nabla_\bx\bigl[K_h*\bigl(K_h*B_1\mu-\bj^\obs\bigr)\bigr],\\
  P^A_2[\mu]&=\frac{\gamma_3}{2}\,\nabla_\bx\bigl[K_h*\bigl(K_h*B_2\mu-\mathcal K^\obs\bigr)\bigr].
\end{aligned}
\]
With Assumption~\ref{ass:wp-reg} and Step~1,
\[
  \|P^A_j[\mu]\|_{W^{1,\infty}_\bx}\le L,
  \qquad
  \|P^A_j[\mu]-P^A_j[\mu']\|_{L^\infty_\bx}\le L\,W_2(\mu,\mu'),
  \qquad j=0,1,2,
\]
since $K_h*B_0\mu-\rho^\obs,K_h*B_1\mu-\bj^\obs, K_h*B_2\mu-\mathcal K^\obs \in W^{2,\infty}_\bx$ and convolution does not increase the sup norm, which implies $\nabla_\bx\bigl[K_h*\bigl(K_h*B_0\mu-\rho^\obs\bigr)\bigr],\nabla_\bx\bigl[K_h*\bigl(K_h*B_1\mu-\bj^\obs\bigr)\bigr],\nabla_\bx\bigl[K_h*\bigl(K_h*B_2\mu-\mathcal K^\obs\bigr)\bigr] \in W^{1,\infty}_\bx$.

For Method~C, collecting \eqref{equ:regularized} the powers of $\bv$ ,
\[
\begin{aligned}
  P^C_2[\mu]&=\frac{\gamma}{2}\,\nabla_\bx \left[K_h*\bigl((\Theta^\obs)^{-1}-\regtemp{\mu}^{-1}\bigr)\right],
  \\
  P^C_1[\mu]&=\gamma\nabla_\bx \left[K_h*\bigl(\vel{\mu}\regtemp{\mu}^{-1}-\bu^\obs(\Theta^\obs)^{-1}\bigr)\right],\\
  P^C_0[\mu]&=\gamma\,\nabla_\bx \Bigl[K_h* \Bigl(
      \log\frac{\dens{\mu}}{\rho^\obs}-\frac d2\log\frac{\regtemp{\mu}}{\Theta^\obs}
       -\frac{|\vel{\mu}|^2}{2\regtemp{\mu}}\\
  &\hspace{7em}
       +\frac{|\bu^\obs|^2}{2\Theta^\obs}
       +\frac{d\varepsilon}{2}\bigl((\Theta^\obs)^{-1}-\regtemp{\mu}^{-1}\bigr)\Bigr)\Bigr].
\end{aligned}
\]
The lower bounds $\dens{\mu}\ge\kappa_h$, $\regtemp{\mu}\ge\varepsilon$, $\rho^\obs\ge\rho^\obs_{\min}$,
$\Theta^\obs\ge\varepsilon$ with \eqref{eq:pf-mom}--\eqref{eq:pf-Tinv-2}, the algebra property of $W^{2,\infty}_\bx$, and the smooth map $t\mapsto\log t$  on $[\varepsilon,L+\varepsilon]$ and $[\kappa_h,L]$ for $\Theta[\mu]$ and $\rho_h[\mu]$, respectively, yields
\[
\begin{gathered}
  \bigl\|\regtemp{\mu}^{-1}\bigr\|_{W^{2,\infty}_\bx},\
  \bigl\|\log\regtemp{\mu}\bigr\|_{W^{2,\infty}_\bx},\
  \bigl\|\log\dens{\mu}\bigr\|_{W^{2,\infty}_\bx},\\
  \bigl\|\vel{\mu}\bigr\|_{W^{2,\infty}_\bx},\
  \bigl\||\vel{\mu}|^2\regtemp{\mu}^{-1}\bigr\|_{W^{2,\infty}_\bx}\ \le L,
\end{gathered}
\]
and the observed analogues $\le L$ by Assumption~\ref{ass:wp-reg}, thus
\[
\begin{aligned}
  \|P^C_2[\mu]\|_{W^{1,\infty}_\bx}
  &\le\frac\gamma2\bigl(\|\regtemp{\mu}^{-1}\|_{W^{2,\infty}_\bx}+\|(\Theta^\obs)^{-1}\|_{W^{2,\infty}_\bx}\bigr)\le L,\\
  \|P^C_1[\mu]\|_{W^{1,\infty}_\bx}
  &\le\gamma\bigl(\|\vel{\mu}\regtemp{\mu}^{-1}\|_{W^{2,\infty}_\bx}+\|\bu^\obs(\Theta^\obs)^{-1}\|_{W^{2,\infty}_\bx}\bigr)\le L,\\
  \|P^C_0[\mu]\|_{W^{1,\infty}_\bx}
  &\le\gamma\bigl(\|\log\dens{\mu}\|_{W^{2,\infty}_\bx}+\frac d2\|\log\regtemp{\mu}\|_{W^{2,\infty}_\bx}\\
& +\frac12\||\vel{\mu}|^2\regtemp{\mu}^{-1}\|_{W^{2,\infty}_\bx}+\frac{d\varepsilon}2\|\regtemp{\mu}^{-1}\|_{W^{2,\infty}_\bx}\\
& +\text{(observed analogues)}\bigr)\le L.
\end{aligned}
\]
The observed terms cancel in each difference, so \eqref{eq:pf-quot}--\eqref{eq:pf-Tinv-2} give
\[
  \|P^C_j[\mu]-P^C_j[\mu']\|_{L^\infty_\bx}\le L\,W_2(\mu,\mu'),\qquad j=0,1,2,
\]
using a similar mean value theorem argument and adding and subtracting mixed terms as in the previous steps. We have that 
\[
\begin{aligned}
\frac{1+|\bv|+|\bv|^2}{w(\bx,\bv)} &= \frac{1+|\bv|+|\bv|^2}{1+\tfrac{|\bv-\bu^\obs(\bx)|^2}{V_*^2}} = \frac{1+|\bv|+|\bv|^2}{1+\tfrac{|\bv|^2-2|\bv||\bu^\obs|+|\bu^\obs|^2}{V_*^2}}\\&\le\frac{1+|\bv|+|\bv|^2}{1+\tfrac{|\bv|^2-2L|\bv|}{V_*^2}},
\end{aligned}
\]
and as $|\bv|\to\infty$, this converges to $V_*^2$. Since $\bv \mapsto (1+|\bv|+|\bv|^2)/w(\bx,\bv)$ is continuous on $\mathbb R^d$, we must have $(1+|\bv|+|\bv|^2)/w\le V_*^2+C$ uniformly in $\bx$.
Then for $a = A,C$, with
$\|P^a_j[\mu]\|_{W^{1,\infty}_\bx}\le L$, and $w\ge1$,
\[
  |\bb^a_\bx|=\frac{|P^a[\mu]|}{w}\le L\,\frac{1+|\bv|+|\bv|^2}{w}\le L.
\]
With $\nabla_\bv P^a=P^a_1+2P^a_2\bv$, $\nabla_\bv w=2(\bv-\bu^\obs)/V_*^2$ and
$\nabla_\bx w=-2(\bv-\bu^\obs)^\top\nabla_\bx\bu^\obs/V_*^2$,
\[
\begin{aligned}
      |\nabla_\bv\bb^a_\bx|
  \le& L\,\frac{1+|\bv|}{w}+|\bb^a_\bx|\,\frac{2}{V_*^2}\frac{|\bv-\bu^\obs|}{w}\le L,\\
  |\nabla_\bx\bb^a_\bx|
  \le& L\,\frac{1+|\bv|+|\bv|^2}{w}+|\bb^a_\bx|\,\frac{2\|\bu^\obs\|_{W^{1,\infty}_\bx}}{V_*^2}\frac{|\bv-\bu^\obs|}{w}\le L,
\end{aligned}
\]
since $\bv\mapsto|\bv-\bu^\obs|/w(\bx,\bv)$ is bounded by $V_*/2$ on $\mathbb R^d$ uniformly in $\bx$.
Combining these two gives $|\bb^a_\bx(t,\bz,\mu)-\bb^a_\bx(t,\bz',\mu)|\le L|\bz-\bz'|$, which with the last
display is \eqref{eq:wp-x} for $a\in\{A,C\}$.
Since $w$ independent of $\mu$,
\[
  |\bb^a_\bx(t,\bz,\mu)-\bb^a_\bx(t,\bz,\mu')|=\frac{|P^a[\mu]-P^a[\mu']|}{w}\le L\,W_2(\mu,\mu').
\]
For Method~B, the position drift from \eqref{eq:split-x} is
$\bb^B_\bx=-\gamma_1\nabla_\bx\left[K_h*\left(K_h*B_0\mu-\rho^\obs\right)\right]$,
which is independent of $\bv$. By \eqref{eq:pf-mom} for $i=0$,
$\rho^\obs\in W^{2,\infty}_\bx$, and the non-expansiveness of $K_h*$,
\[
\begin{aligned}
&|\bb^B_\bx|\le\gamma_1\bigl(\|\nabla_\bx\dens{\mu}\|_{L^\infty_\bx}+\|\nabla_\bx\rho^\obs\|_{L^\infty_\bx}\bigr)\le \gamma_1\bigl(\|\dens{\mu}\|_{W^{2,\infty}_\bx}+\|\rho^\obs\|_{W^{2,\infty}_\bx}\bigr)\le L,\\ &|\nabla_\bx\bb^B_\bx|\le\gamma_1\bigl(\|\nabla^2_\bx\dens{\mu}\|_{L^\infty_\bx}+\|\nabla^2_\bx\rho^\obs\|_{L^\infty_\bx}\bigr)\le\gamma_1\bigl(\|\dens{\mu}\|_{W^{2,\infty}_\bx}+\|\rho^\obs\|_{W^{2,\infty}_\bx}\bigr)\le L,
\end{aligned}
\]
while $\nabla_\bv\bb^B_\bx=0$, and by the $i=0$ Lipschitz bound of \eqref{eq:pf-mom},
\[
  |\bb^B_\bx[\mu]-\bb^B_\bx[\mu']|\le\gamma_1\|\nabla_\bx\dens{\mu}-\nabla_\bx\dens{\mu'}\|_{L^\infty_\bx}\le\gamma_1\|\dens{\mu}-\dens{\mu'}\|_{W^{2,\infty}_\bx}\le L\,W_2(\mu,\mu').
\]
This is \eqref{eq:wp-x} for $a=B$, completing the proof.
\end{proof}
\subsection{Proof of Lemma \ref{lem:local-wp}} \label{app:sub:local-wp}
\begin{lemma}[Empirical coefficient regularity]
\label{lem:empirical-coefficient-regularity}
Let Assumption~\ref{ass:wp-reg} hold, fix \(a\in\{A,B,C\}\) and \(N_p\ge1\), and set
\(
  \mathcal Z:=\mathbb T^d\times\mathbb R^d.
\)
For a particle configuration
\(
  \mathbf Z=(\bz^1,\ldots,\bz^{N_p})\in\mathcal Z^{N_p},
  \bz^i=(\bx^i,\bv^i),
\)
define its empirical measure by
\(
  \mu_{\mathbf Z}^{N_p}:=\frac1{N_p}\sum_{j=1}^{N_p}\delta_{\bz^j}.
\)
Let
\( 
  \varrho_{\mathcal Z}(\bz,\widetilde\bz)
  :=
  (
    d_{\mathbb T^d}(\bx,\widetilde\bx)^2
    +|\bv-\widetilde\bv|^2
  )^{1/2},
\)
and equip \(\mathcal Z^{N_p}\) with
\(
  \varrho_{N_p}(\mathbf Z,\widetilde{\mathbf Z})
  :=
  \max_{1\le j\le N_p}
  \varrho_{\mathcal Z}(\bz^j,\widetilde\bz^j).
\)
Then the drift of the finite-particle system is locally Lipschitz in the
particle configuration. Specifically, for every \(R<\infty\), there exists
\(L_R<\infty\), independent of \(N_p\) and \(t\), such that, whenever
\[
  \max_{1\le j\le N_p}|\bv^j|\le R,
  \qquad
  \max_{1\le j\le N_p}|\widetilde\bv^j|\le R,
\]
one has, for every \(i=1,\ldots,N_p\),
\begin{equation}
\begin{aligned}
&\Bigl| \bv^i+\bb^a_\bx(t,\bz^i,\mu_{\mathbf Z}^{N_p})
  -\widetilde\bv^i
  -\bb^a_\bx
     \bigl(t,\widetilde\bz^i,\mu_{\widetilde{\mathbf Z}}^{N_p}\bigr)
 \Bigr|\\
&+
 \Bigl|
   \field{\mu_{\mathbf Z}^{N_p}}(\bx^i)
   +\bb^a_\bv(t,\bz^i,\mu_{\mathbf Z}^{N_p})
   -\nu\bv^i
   -\field{\mu_{\widetilde{\mathbf Z}}^{N_p}}(\widetilde\bx^i)
   -\bb^a_\bv
      \bigl(t,\widetilde\bz^i,\mu_{\widetilde{\mathbf Z}}^{N_p}\bigr)
   +\nu\widetilde\bv^i
 \Bigr| \\
 &\le
L_R\,\varrho_{N_p}(\mathbf Z,\widetilde{\mathbf Z}).
\label{eq:empirical-drift-lip}
\end{aligned}
\end{equation}
\end{lemma}
\begin{proof}
Consider the coupling
\(
  \pi_{\mathbf Z,\widetilde{\mathbf Z}}^{N_p}
  :=
  \frac1{N_p}\sum_{j=1}^{N_p}
  \delta_{(\bz^j,\widetilde\bz^j)}
  \in
  \Gamma\bigl(
    \mu_{\mathbf Z}^{N_p},
    \mu_{\widetilde{\mathbf Z}}^{N_p}
  \bigr).
\)
By the definition of the Wasserstein-2 distance,
\[
\begin{aligned}
W_2^2(\mu_\bZ^{N_p},\mu_{\widetilde \bZ}^{N_p}) &\le \int_{\mathcal Z \times\mathcal Z} \varrho^2_{\mathcal Z}(\bz,\widetilde \bz) \diff \pi^{N_p}_{\bZ,\widetilde\bZ}=\frac{1}{N_p}\sum^{N_p}_{j=1}\varrho_{\mathcal Z}^2(\bz^j,\widetilde \bz^j)\\
&\le \frac{1}{N_p}\sum_{j=1}^{N_p} \varrho_{N_p}^2(\bZ,\widetilde\bZ)=\varrho_{N_p}(\bZ,\widetilde\bZ)
\end{aligned}
\]
Hence,
\begin{equation}
  W_2\bigl(
    \mu_{\mathbf Z}^{N_p},
    \mu_{\widetilde{\mathbf Z}}^{N_p}
  \bigr)
  \le
  \varrho_{N_p}(\mathbf Z,\widetilde{\mathbf Z}).
\label{eq:empirical-W2}
\end{equation}
Moreover,
\[
  \int_{\mathcal Z}|\bv|^4
  \,\mu_{\mathbf Z}^{N_p}(\diff\bz)
  \le R^4,
  \qquad
  \int_{\mathcal Z}|\bv|^4
  \,\mu_{\widetilde{\mathbf Z}}^{N_p}(\diff\bz)
  \le R^4.
\]
So both empirical measures belong to the moment ball
\(\mathcal P_{R^4}\) appearing in Lemma~\ref{lem:reg-bounds}.
The position-drift estimate follows immediately from
\eqref{eq:wp-x} and \eqref{eq:empirical-W2}:
\[
\begin{aligned}
\Bigl|
  \bb^a_\bx(t,\bz^i,\mu)
  - \bb^a_\bx(t,\widetilde\bz^i,\widetilde\mu)
\Bigr| &\le L_{R^4}(\varrho_{\mathcal Z}(\bz^i,\widetilde\bz^i)+W_2(\mu,\widetilde\mu))\\
&\le L_{R^4}\varrho_{N_p}(\bZ,\widetilde\bZ).
\end{aligned}
\]
For the velocity drift, write
\[
  \bb^a_\bv(t,\bz,\mu)
  =
  \alpha^a[\mu](t,\bx)\bv+\beta^a[\mu](t,\bx).
\]
Then, using \eqref{eq:wp-v}--\eqref{eq:wp-vmu},
\begin{align*}
&\Bigl|
  \bb^a_\bv(t,\bz^i,\mu)
  - \bb^a_\bv(t,\widetilde\bz^i,\widetilde\mu)
\Bigr|
\\
=&
\Bigl|
  \alpha^a[\mu](t,\bx^i)\bv^i
  -
  \alpha^a[\widetilde\mu](t,\widetilde\bx^i)
  \widetilde\bv^i
  +
  \beta^a[\mu](t,\bx^i)
  -
  \beta^a[\widetilde\mu](t,\widetilde\bx^i)
\Bigr|
\\
\le&
  \bigl|\alpha^a[\mu](t,\bx^i)\bigr|
  \,|\bv^i-\widetilde\bv^i|
  +
  \bigl|
    \alpha^a[\mu](t,\bx^i)
    -
    \alpha^a[\mu](t,\widetilde\bx^i)
  \bigr|
  \,|\widetilde\bv^i|\\
  &   +
  \bigl|
    \alpha^a[\mu](t,\widetilde\bx^i)
    -
    \alpha^a[\widetilde\mu](t,\widetilde\bx^i)
  \bigr|
  |\widetilde\bv^i|
  +
  \bigl|
    \beta^a[\mu](t,\bx^i)
    -
    \beta^a[\mu](t,\widetilde\bx^i)
  \bigr|
\\
&\quad
  +
  \bigl|
    \beta^a[\mu](t,\widetilde\bx^i)
    -
    \beta^a[\widetilde\mu](t,\widetilde\bx^i)
  \bigr|
\\
&\le
  L_{R}|\bv^i-\widetilde\bv^i|
  +
  L_{R}|\widetilde\bv^i|\,
  d_{\mathbb T^d}(\bx^i,\widetilde\bx^i)
\\
&\quad
  +
  L_{R}|\widetilde\bv^i|\,
  W_2(\mu,\widetilde\mu)
  +
  L_{R}d_{\mathbb T^d}(\bx^i,\widetilde\bx^i)
  +
  L_{R}W_2(\mu,\widetilde\mu)
\\
&\le
  L_R
  \left(
    \varrho_{\mathcal Z}(\bz^i,\widetilde\bz^i)
    +
    W_2(\mu,\widetilde\mu)
  \right)\le
  L_R\,\varrho_{N_p}(\mathbf Z,\widetilde{\mathbf Z}),
\end{align*}
where in the last two lines we used
\(
|\widetilde\bv^i|\le R
\),
\eqref{eq:empirical-W2}, and
\(
\varrho_{\mathcal Z}(\bz^i,\widetilde\bz^i)
\le \varrho_{N_p}(\mathbf Z,\widetilde{\mathbf Z})
\).

For the position component, adding and subtracting
\(\bb_\bx^a(t,\widetilde\bz^i,\mu)\) gives
\begin{align*}
&\Bigl|
  \bv^i+\bb_\bx^a(t,\bz^i,\mu)
  -
  \widetilde\bv^i
  -
  \bb_\bx^a(t,\widetilde\bz^i,\widetilde\mu)
\Bigr|
\\
&\le
  |\bv^i-\widetilde\bv^i|
  +
  \Bigl|
    \bb_\bx^a(t,\bz^i,\mu)
    -
    \bb_\bx^a(t,\widetilde\bz^i,\mu)
  \Bigr|
\\
&\quad
  +
  \Bigl|
    \bb_\bx^a(t,\widetilde\bz^i,\mu)
    -
    \bb_\bx^a(t,\widetilde\bz^i,\widetilde\mu)
  \Bigr|
\\
&\le
  |\bv^i-\widetilde\bv^i|
  +
  L_{R^4}\varrho_{\mathcal Z}(\bz^i,\widetilde\bz^i)
  +
  L_{R^4}W_2(\mu,\widetilde\mu)
\\
&\le
  L_R\,\varrho_{N_p}(\mathbf Z,\widetilde{\mathbf Z}),
\end{align*}
where we used \eqref{eq:wp-x} and
\eqref{eq:empirical-W2}.
For the self-consistent field, Lemma~\ref{lem:field}
gives
\begin{align*}
&\Bigl|
  \field{\mu}(\bx^i)
  -
  \field{\widetilde\mu}(\widetilde\bx^i)
\Bigr|
\\
&\le
  \Bigl|
    \field{\mu}(\bx^i)
    -
    \field{\mu}(\widetilde\bx^i)
  \Bigr|
  +
  \Bigl|
    \field{\mu}(\widetilde\bx^i)
    -
    \field{\widetilde\mu}(\widetilde\bx^i)
  \Bigr|
\\
&\le
  L\,d_{\mathbb T^d}(\bx^i,\widetilde\bx^i)
  +
  L\,W_2(\mu,\widetilde\mu)
\\
&\le
  L\,\varrho_{N_p}(\mathbf Z,\widetilde{\mathbf Z}).
\end{align*}
Consequently,
\begin{align*}
&\Bigl|
  \field{\mu}(\bx^i)
  +
  \bb_\bv^a(t,\bz^i,\mu)
  -
  \nu\bv^i
  -
  \field{\widetilde\mu}(\widetilde\bx^i)
  -
  \bb_\bv^a(t,\widetilde\bz^i,\widetilde\mu)
  +
  \nu\widetilde\bv^i
\Bigr|
\\
&\le
  \Bigl|
    \field{\mu}(\bx^i)
    -
    \field{\widetilde\mu}(\widetilde\bx^i)
  \Bigr|
  +
  \Bigl|
    \bb_\bv^a(t,\bz^i,\mu)
    -
    \bb_\bv^a(t,\widetilde\bz^i,\widetilde\mu)
  \Bigr|
  +
  \nu|\bv^i-\widetilde\bv^i|
\\
&\le
  L_R\,\varrho_{N_p}(\mathbf Z,\widetilde{\mathbf Z}).
\end{align*}
Combining the position and velocity estimates proves
\eqref{eq:empirical-drift-lip}.
\end{proof}

\begin{proof}[Proof of Lemma~\ref{lem:local-wp}]
Lift the positions periodically and write
$\mathbf Z=(\mathbf X,\mathbf V)\in\mathbb R^{2dN_p}=:\mathsf E$, an open set with
$\partial \mathsf E=\emptyset$. In integral form, the lifted particle system
\eqref{eq:nudged-particles} is an SDE of the type treated in
\cite[Theorem~D1]{birrell2020langevin}, with constant forcing
$\mathbf Z_0$ (a continuous semimartingale valued in $\mathsf E$), drift
$\mathbf B$ given by the bracketed terms of \eqref{eq:nudged-particles},
driven by a $dN_p$-dimensional Brownian motion with constant diffusion
matrix
\[
  \boldsymbol\Sigma
  =
  \begin{pmatrix}
    \mathbf{0}_{dN_p\times dN_p}\\
    \sqrt{2\nu\Theta_{\mathrm{LB}}}\,\mathbf{I}_{dN_p}
  \end{pmatrix},
\]
so the noise enters the velocity components only.

We verify the hypotheses of \cite[Theorem~D1]{birrell2020langevin}.
$\mathbf B$ is Borel measurable in $(t,\mathbf z)$ by
Assumption~\ref{ass:wp-reg}, and $\boldsymbol\Sigma$ is constant. Fix a
compact set $K\subset\mathsf E$ and choose $R$ with
$K\subset\{\|\mathbf v\|_{\max}\le R\}$. By
Lemmas~\ref{lem:reg-bounds} and~\ref{lem:empirical-coefficient-regularity},
the drift is bounded on $\{\|\mathbf v\|_{\max}\le R\}$ and Lipschitz in
the configuration with a
constant $L_R$ uniform in $t$ and $N_p$. That estimate is stated for the
metric $\varrho_{N_p}$, whose position component is the torus distance.
Since
$d_{\mathbb T^d}(\bx,\widetilde\bx)\le|\bx-\widetilde\bx|$ and the
coefficients are periodic in the lifted positions, the same bound holds
for the Euclidean distance on $\mathsf E=\mathbb R^{2dN_p}$, uniformly
over the lifted position, and in particular on every compact $K$.
\cite[Theorem~D1]{birrell2020langevin} therefore yields a pathwise-unique
maximal strong solution up to an blow-up time $\tau_\infty:=e$,
requiring no moment assumption on $\mathbf Z_0$.

Because $\partial \mathsf E=\emptyset$, alternative~(3) of the blow-up in \cite[Theorem~D1]{birrell2020langevin} is vacuous, so on
$\{\tau_\infty<\infty\}$, there is a sequence $t_n\uparrow\tau_\infty$ with
$|\mathbf Z_{t_n}|\to\infty$. It remains to show this forces the
velocities to blow up. By the position-drift bound of
Lemma~\ref{lem:reg-bounds}, whenever $\sup_{s<t}\|\mathbf V_s\|_{\max}\le R$,
one has $|\bb^a_\bx(s,\cdot)|\le L$, so for the lifted positions and every
$t<\tau_\infty$,
\[
  |\widetilde\bX^i_t-\widetilde\bX^i_0|
  \le\int_0^t\bigl(|\bV^i_s|+|\bb^a_\bx(s,\bZ^i_s,\mu^{N_p}_s)|\bigr)\,\diff s
  \le t\bigl(\sup_{s<t}\|\mathbf V_s\|_{\max}+L\bigr).
\]
Hence if $\sup_{t<\tau_\infty}\|\mathbf V_t\|_{\max}<\infty$, then
$\sup_{t<\tau_\infty}|\mathbf Z_t|<\infty$, contradicting
$|\mathbf Z_{t_n}|\to\infty$. Therefore,
$\limsup_{t\uparrow\tau_\infty}\max_{1\le i\le N_p}|\bV^i_t|=\infty$. For
\[
\tau_R = \inf\{t\ge 0:\|\bV_t\|_{\max}\ge R\},
\]
we have that $\tau_R\le \tau_\infty$ for every finite $R$, since on $[0,\tau_R)$, velocities are  $<R$ by the definition of $\tau_R$ so the solution hasn't blown up yet. Hence, $\lim\limits_{R\to\infty}\tau_R\le\tau_\infty$. For the other direction, we have for all $t<\tau_\infty$, $R_0:=\sup\limits_{0\le s\le t} \|\bV_s\|_{\max}<\infty$ and $\tau_{R_0+1}>t$. So using that $\tau_R$ is nondecreasing in $R$, $\lim\limits_{R\to\infty}\tau_R = \sup\limits_R \tau_R \ge \tau_{R_0+1}>t$ for all $t<\tau_\infty$, which implies $\lim\limits_{R\to\infty}\tau_R \ge \tau_\infty$. Therefore, finite-time blow-up can occur only through the velocity variables, and
\[
  \tau_\infty=\lim_{R\to\infty}\tau_R.
\] On the original torus state space, where positions
live in the compact $\mathbb T^{dN_p}$, the restriction of blow-up to the
velocities is immediate.
\end{proof}
\subsection{Proof of Lemma \ref{lem:moment-generator}}
\label{app:sub:moment-generator}
\begin{proof}
For a smooth function $\varphi=\varphi(\bv)$, define the velocity generator:
\[  
\mathscr A_{t,\mu}^a\varphi(\bz)
:=
\Bigl(
\field{\mu}(\bx)
+\bb_{\bv}^a(t,\bz,\mu)
-\nu\bv
\Bigr)\cdot\nabla_{\bv}\varphi(\bv)
+\nu\Theta_{\mathrm{LB}}\Delta_{\bv}\varphi(\bv).
\]
For $\varphi_q(\bv)=|\bv|^q$, $q\ge2$,
\(
  \nabla_{\bv}\varphi_q
  =
  q|\bv|^{q-2}\bv,
  \Delta_{\bv}\varphi_q
  =
  q(q+d-2)|\bv|^{q-2},
\)
and hence
\[
\begin{aligned}
&   \mathcal D_q^a(t,\mu)
:=
  \int
  \mathscr A_{t,\mu}^a\varphi_q(\bz)
  \,\mu(\diff\bz)\\
  & = 
\int
\left[
q|\bv|^{q-2}\bv\cdot
\Bigl(
\field{\mu}(\bx)
+\bb_{\bv}^a(t,\bz,\mu)
-\nu\bv
\Bigr)
+
\nu\Theta_{\mathrm{LB}}q(q+d-2)|\bv|^{q-2}
\right]
\mu(\diff\bz). 
\end{aligned}
\]
It therefore suffices to prove that, for every $\zeta>0$,
\begin{equation}\label{equ:D_bound}
  \mathcal D_q^a(t,\mu)
  \le
  q(\bar\alpha_a-\nu+\zeta)m_q(\mu)
  +
  C_q(\zeta),
  \qquad
  \mu\in\mathcal P_q,
\end{equation}
where $m_q(\mu) = \int_{\mathcal Z} |\bv|^q \diff\mu(\bx,\bv)$.
From $\kappa_h\le K_h\le K_\infty$ and
$\int K_h=1$,
\[
  \kappa_h\le\dens{\mu}\le K_\infty,\qquad
  0\le\kin{\mu}\le\tfrac{K_\infty}2 m_2(\mu),\qquad
  |\mom{\mu}|^2\le 2\dens{\mu}\kin{\mu},
\]
where the last is by Cauchy--Schwarz:
\[
\begin{aligned}
|\mom{\mu}|^2 &= \left|\int_{\mathcal Z} \bv K_h(\bx-\by)\mu(\diff \by,\diff \bv)\right|^2\\
&= \left|\int_{\mathcal Z} \bv \sqrt{K_h(\bx-\by)}\sqrt{K_h(\bx-\by)}\mu(\diff \by,\diff \bv)\right|^2\\
&\le \left|\int_{\mathcal Z} |\bv|^2 K_h(\bx-\by)\mu(\diff \by,\diff \bv)\right|\left|\int_{\mathcal Z} K_h(\bx-\by)\mu(\diff \by,\diff \bv)\right|\\
&= 2\rho_h[\mu]\mathcal K_h[\mu].
\end{aligned}
\] With $\dens{\mu}\ge\kappa_h$, these reproduce the bulk-velocity
bound $\|\vel{\mu}\|_\infty^2 = \sup_\bx \frac{|\mom{\mu}|^2}{\dens{\mu}^2} \leq \sup_\bx\frac{2\kin{\mu}}{\dens{\mu}}\leq \frac{K_\infty}{\kappa_h}m_2(\mu)$.

For Methods A and B, the velocity feedback is
\[
  \bb^a_\bv(t,\bz,\mu)
  =
  -\gamma_2 K_h*(\mom{\mu}-\bj^\obs)(\bx)
  -\gamma_3\bv\,K_h*(\kin{\mu}-\mathcal K^\obs)(\bx), \quad a \in\{A,B\}.
\]
Since $K_h\ge0$ and $\int_{\mathbb T^d}K_h=1$, Jensen's inequality
and $|\mom{\mu}|^2\le 2\dens{\mu}\kin{\mu}
\le 2K_\infty\kin{\mu}$ give
\begin{equation}
  |K_h*\mom{\mu}|^2
  \le
  K_h*|\mom{\mu}|^2
  \le
  2K_\infty\,K_h*\kin{\mu}.
  \label{eq:outer-jK-bound}
\end{equation}
For $a\in\{A,B\}$,
\[
\begin{aligned}
|\bv|^{q-2}\bv\cdot\bb^a_\bv
={}&
-\gamma_2|\bv|^{q-2}\bv\cdot(K_h*\mom{\mu})
-\gamma_3|\bv|^q(K_h*\kin{\mu})
\\
&+
\gamma_2|\bv|^{q-2}\bv\cdot(K_h*\bj^\obs)
+\gamma_3|\bv|^q(K_h*\mathcal K^\obs).
\end{aligned}
\]
The terms depending on $\mu$ are dissipative up to lower order.
Indeed, we have that
\[
\begin{aligned}
&-\gamma_3(K_h*\kin{\mu})|\bv|^q
-\gamma_2|\bv|^{q-2}\bv\cdot(K_h*\mom{\mu})\\
&\qquad\le
-\gamma_3(K_h*\kin{\mu})|\bv|^q
+\gamma_2|K_h*\mom{\mu}|\,|\bv|^{q-1}
\\
&\qquad\le
-\frac{\gamma_3}{2}(K_h*\kin{\mu})|\bv|^q
+\frac{\gamma_2^2K_\infty}{\gamma_3}|\bv|^{q-2}
\le
\frac{\gamma_2^2K_\infty}{\gamma_3}|\bv|^{q-2},
\end{aligned}
\]
where we used  Young's inequality 
$ab \le \tfrac{\epsilon a^2}{2} +\tfrac{b^2}{2\epsilon}$
with $\epsilon = \gamma_3$ and
\[
a = |\bv|^{q/2}\sqrt{K_h*\mathcal K_h[\mu]}, \qquad b= \frac{\gamma_2|\bv|^{q/2-1}|K_h*\bj_h[\mu]|}{\sqrt{K_h*\mathcal K_h[\mu]}},
\]
along with \eqref{eq:outer-jK-bound} for the second term:
\[
\begin{aligned}
\gamma_2|K_h*\bj_h[\mu]||\bv|^{q-1} &\le \frac{\gamma_3|\bv|^qK_h*\mathcal K_h[\mu]}{2}+\frac{\gamma_2^2|\bv|^{q-2}|K_h*\bj_h[\mu]|^2}{2\gamma_3 K_h*\mathcal K_h[\mu]}\\
&\le \frac{\gamma_3|\bv|^qK_h*\mathcal K_h[\mu]}{2}+ \frac{\gamma_2^2 K_\infty}{\gamma_3}|\bv|^{q-2}.
\end{aligned}
\]
For the observed terms, the positivity and unit-mass normalization of $K_h$
give
\[
0\le
\gamma_3K_h*\mathcal K^\obs
\le
\gamma_3\|\mathcal K^\obs\|_{L^\infty_{t,\bx}}
=
\bar\alpha_a,
\qquad
\|K_h*\bj^\obs\|_{L^\infty_{t,\bx}}
\le
\|\bj^\obs\|_{L^\infty_{t,\bx}}.
\]
Thus,
\[
\begin{aligned}
&\gamma_2|\bv|^{q-2}\bv\cdot(K_h*\bj^\obs)
+\gamma_3|\bv|^q(K_h*\mathcal K^\obs)
\\
&\qquad\le
\bar\alpha_a|\bv|^q
+\gamma_2\|\bj^\obs\|_{L^\infty_{t,\bx}}|\bv|^{q-1}.
\end{aligned}
\]
Combining the two estimates and integrating against $\mu$ yields
\begin{equation}
\int |\bv|^{q-2}\bv\cdot\bb^a_\bv\,\mu(\diff\bz)
\le
\bar\alpha_a m_q(\mu)
+
C\bigl(m_{q-1}(\mu)+m_{q-2}(\mu)\bigr),
\qquad a\in\{A,B\}.
  \label{eq:AB-q-bound}
\end{equation} 
For Method C, the regularized velocity feedback can be written as
\[
\begin{aligned}
\bb^C_\bv(t,\bz,\mu)
=&
\gamma
K_h*
\left(
\regtemp{\mu}^{-1}
-
(\Theta^\obs)^{-1}
\right)(\bx)\,\bv\\
&+
\gamma
K_h*
\left(
\bu^\obs(\Theta^\obs)^{-1}
-
\vel{\mu}\regtemp{\mu}^{-1}
\right)(\bx).
\end{aligned}
\]
Since $K_h\ge0$, $\int_{\mathbb T^d}K_h=1$, and
$\regtemp{\mu},\Theta^\obs\ge\varepsilon$, the negative term
\[
-\gamma
\bigl(K_h*(\Theta^\obs)^{-1}\bigr)
|\bv|^q
\]
in the integral may be ignored, while
\[
K_h*\regtemp{\mu}^{-1}\le \varepsilon^{-1}.
\]
Moreover,
\[
\left\|
K_h*(\vel{\mu}\regtemp{\mu}^{-1})
\right\|_{L^\infty_\bx}
\le
\frac{1}{\varepsilon}
\|\vel{\mu}\|_{L^\infty_\bx}
\le
\frac{1}{\varepsilon}
\sqrt{\frac{K_\infty}{\kappa_h}}\,
m_2(\mu)^{1/2},
\]
whereas
\[
\left\|
K_h*(\bu^\obs(\Theta^\obs)^{-1})
\right\|_{L^\infty_{t,\bx}}
\le C
\]
by Assumption~\ref{ass:wp-reg}. Consequently,
\[
\begin{aligned}
&\int
|\bv|^{q-2}\bv\cdot
\bb^C_\bv
\,\mu(\diff\bz)\le
\frac{\gamma}{\varepsilon}m_q(\mu)
+
\frac{\gamma}{\varepsilon}
\sqrt{\frac{K_\infty}{\kappa_h}}\,
m_2(\mu)^{1/2}m_{q-1}(\mu)
+
C m_{q-1}(\mu).
\end{aligned}
\]
By Lyapunov's inequality, $m_2(\mu)^{1/2}\le m_q(\mu)^{1/q}$ and $m_{q-1}(\mu)\le m_q(\mu)^{(q-1)/q}$, so
\[
m_2(\mu)^{1/2}m_{q-1}(\mu)
\le
m_q(\mu),
\]
and therefore
\begin{equation}
\int
|\bv|^{q-2}\bv\cdot
\bb^C_\bv(t,\bz,\mu)
\,\mu(\diff\bz)
\le
\frac{\gamma}{\varepsilon}
\left(
1+\sqrt{\frac{K_\infty}{\kappa_h}}
\right)
m_q(\mu)
+
C m_{q-1}(\mu).
\label{eq:C-q-bound}
\end{equation}

Recall that
\[
\begin{aligned}
\mathcal D_q^a(t,\mu)
={}&
q\int |\bv|^{q-2}\bv\cdot \field{\mu}(\bx)\,\mu(\diff\bz)
+
q\int |\bv|^{q-2}\bv\cdot \bb_\bv^a(t,\bz,\mu)\,\mu(\diff\bz)
\\
&\quad
-q\nu m_q(\mu)
+
\nu\Theta_{\mathrm{LB}}q(q+d-2)m_{q-2}(\mu).
\end{aligned}
\]
By Lemma~\ref{lem:field}, there is a constant $C_E$ such that
$\|\field{\mu}\|_{L^\infty_\bx}\le C_E$ uniformly in $\mu$. Hence,
\[
\left|
\int |\bv|^{q-2}\bv\cdot \field{\mu}(\bx)\,\mu(\diff\bz)
\right|
\le
C_E m_{q-1}(\mu).
\]
On the other hand, \eqref{eq:AB-q-bound} and \eqref{eq:C-q-bound} give, for all
$a\in\{A,B,C\}$,
\[
\int
|\bv|^{q-2}\bv\cdot \bb_\bv^a(t,\bz,\mu)
\,\mu(\diff\bz)
\le
\bar\alpha_a m_q(\mu)
+
C_q\bigl(m_{q-1}(\mu)+m_{q-2}(\mu)\bigr).
\]
Substituting these bounds into the displayed formula for $\mathcal D_q^a$ yields
\[
\mathcal D_q^a(t,\mu)
\le
q(\bar\alpha_a-\nu)m_q(\mu)
+
C_q\bigl(m_{q-1}(\mu)+m_{q-2}(\mu)\bigr).
\]
Finally, for every $0\le p<q$ and every $\epsilon'>0$, Young's inequality gives
\[
m_p(\mu)\le \epsilon' m_q(\mu)+C_{q,p}(\epsilon').
\]
Indeed for $\delta := \tfrac{q}{p}\epsilon'$, $|\bv|^p = (\delta|\bv|^q)^{p/q}\delta^{-p/q}$ and using Young's inequality with exponents $\tfrac{q}{p}$ and $\tfrac{q}{q-p}$, we have
\[
\begin{aligned}
|\bv|^p &\le \frac{p\delta|\bv|^q}{q} + \frac{q-p}{q}\delta^{-\tfrac{p}{q-p}}\\
&= \epsilon' |\bv|^q +C_{q,p}(\epsilon'),
\end{aligned}
\]
and integrating over $\mu(\diff \bz)$ gives the desired inequality.
Choosing \(\epsilon'\) small enough so that the lower-order terms contribute at most
\(q\zeta m_q(\mu)\), we obtain \eqref{equ:D_bound}.

Applying the generator of \eqref{eq:nudged-particles} to the configuration function
$\mathbf Z\mapsto m^{N_p}_q(\mathbf Z)=\frac1{N_p}\sum_{i=1}^{N_p}|\bv^i|^q$, for
$\mathbf Z=(\bz^1,\dots,\bz^{N_p})$, $\bz^i=(\bx^i,\bv^i)$, and using that the drift and the
constant diffusion  particle-wise, gives
\[
\bigl(\mathcal L_{N_p,t}m^{N_p}_q\bigr)(\mathbf Z)
=\int \mathscr A^a_{t,\mu^{N_p}_{\mathbf Z}}\varphi_q\,\diff\mu^{N_p}_{\mathbf Z}
=\mathcal D^a_q\bigl(t,\mu^{N_p}_{\mathbf Z}\bigr),
\qquad \mu^{N_p}_{\mathbf Z}:=\tfrac1{N_p}\textstyle\sum_{i=1}^{N_p}\delta_{\bz^i}.
\]
Evaluating at $\mathbf Z=\bZ_t$ up to the blow-up time and applying
\eqref{equ:D_bound} with $\mu=\mu^{N_p}_t\in\mathcal P_q$ gives \eqref{eq:Mq-generator}.
\end{proof}

\subsection{Proof of Corollary~\ref{cor:finite-time-moment}}
\label{app:sub:finite-time-moment}
\begin{proof}[Proof of Corollary~\ref{cor:finite-time-moment}]
Apply It\^o's formula to the stopped process
\(m_q^{N_p}(t\wedge\tau_R)\), where \(\tau_R\) is defined in
Lemma~\ref{lem:local-wp}. The stopped martingale has zero expectation, and
\eqref{eq:Mq-generator} gives
\[
\begin{aligned}
  \mathbb E\,m_q^{N_p}(t\wedge\tau_R)
  &\le
  \mathbb E m_q^{N_p}(0)
  +
  \mathbb E\int_0^{t\wedge\tau_R}
  \left[
    q(\bar\alpha_a-\nu+\zeta)m_q^{N_p}(s)+C_q(\zeta)
  \right]\diff s  \\
  &\le
  \mathbb E m_q^{N_p}(0)
  +
  \int_0^t
  \left[
    \max\{q(\bar\alpha_a-\nu+\zeta),0\}
    \mathbb E\,m_q^{N_p}(s\wedge\tau_R)
    +C_q(\zeta)
  \right]\diff s .
\end{aligned}
\]
Gronwall's inequality yields
\[
  \sup_{0\le t\le T_{\mathrm{fin}}}
  \mathbb E\,m_q^{N_p}(t\wedge\tau_R)
  \le
  \bigl(\mathbb E m_q^{N_p}(0)+C_q(\zeta)T_{\mathrm{fin}}\bigr)
  \exp\!\left(
    \max\{q(\bar\alpha_a-\nu+\zeta),0\}\,T_{\mathrm{fin}}
  \right).
\]
On the event \(\{\tau_R\le T_{\mathrm{fin}}\}\), continuity of the velocity paths implies
\[
  m_q^{N_p}(T_{\mathrm{fin}}\wedge\tau_R)
  =
  m_q^{N_p}(\tau_R) = \frac{1}{N_p}\sum\limits_{i=1}^{N_p} |\bv^i_{\tau_R}|^q
  \ge \frac{1}{N_p}|\bv^{i^*}_{\tau_R}|^q=
  \frac{R^q}{N_p},
\]
where $1\le i^*\le N_p$ is a particle index such that $|\bv^{i^*}_{\tau_R}|=R$.
Therefore, by Markov's inequality,
\[
  \mathbb P(\tau_R\le T_{\mathrm{fin}})
  \le \mathbb P\left(m^{N_p}_q(T_{\mathrm{fin}}\wedge \tau_R)\ge \frac{R^q}{N_p}\right) \le
  \frac{N_p}{R^q}
  \mathbb E\,m_q^{N_p}(T_{\mathrm{fin}}\wedge\tau_R)
  \xrightarrow[R\to\infty]{}0 .
\]
Since \(\tau_R\uparrow\tau_\infty\), this proves
\(\tau_\infty>T_{\mathrm{fin}}\) almost surely. As \(T_{\mathrm{fin}}<\infty\) was arbitrary, we have that $\tau_\infty = \infty$ almost surely, and the finite particle
system does not blow up globally.
Finally, \(t\wedge\tau_R\to t\) as $R\to \infty$ almost surely for each fixed \(t\). Fatou's lemma
applied to the stopped bound gives \eqref{eq:finite-time-Mq-bound}.
\end{proof}

\subsection{Proof of Proposition~\ref{prop:wp-finiteN}}
\label{app:sub:wp-finiteN}
\begin{proof}[Proof of Proposition~\ref{prop:wp-finiteN}]
By Lemma~\ref{lem:local-wp}, the particle system admits a unique strong solution up
to the maximal time \(\tau_\infty\). Applying
Corollary~\ref{cor:finite-time-moment} with \(q=2\) gives
\[
  \mathbb P(\tau_\infty\le T_{\mathrm{fin}})=0
  \qquad
  \text{for every } T_{\mathrm{fin}}<\infty .
\]
Hence \(\tau_\infty=\infty\) almost surely, and the local strong solution is global.
The global pathwise uniqueness follows from the uniqueness in
Lemma~\ref{lem:local-wp}, since the maximal lifetime is almost surely infinite.
\end{proof}

\fi


\bibliographystyle{unsrt}
\bibliography{main}
\end{document}